\documentclass[reqno]{amsart}
\usepackage[top=1in, bottom=1in, left=1.3in, right=1.3in]{geometry}
\usepackage{soul}
\usepackage{amsmath}
\usepackage{amssymb}
\usepackage{amsthm}
\usepackage{verbatim}
\usepackage{graphicx}
\usepackage{graphics}
\usepackage{color}
\usepackage{enumerate}
\usepackage{mathrsfs}
\usepackage{booktabs}
\usepackage[toc,page]{appendix}
\usepackage{enumerate}
\usepackage{mathrsfs}
\usepackage{fancyhdr}
\usepackage{latexsym, amsfonts, amssymb, amsmath,  amsthm}
\usepackage{amsopn, amstext, amscd,pifont}
\usepackage{amssymb,bm, cite}
\usepackage[all]{xy}
\usepackage{color}
\usepackage{blkarray}
\usepackage{hyperref}
\usepackage{float}
\usepackage{listings}
\usepackage{setspace}
\usepackage{paralist}
\usepackage{float}
\usepackage{fancyhdr}

\theoremstyle{plain}
\newtheorem{theorem}{Theorem}[section]
\newtheorem{proposition}[theorem]{Proposition}
\newtheorem{lemma}[theorem]{Lemma}
\newtheorem{corollary}[theorem]{Corollary}

\theoremstyle{definition}
\newtheorem{definition}[theorem]{Definition}
\newtheorem{example}[theorem]{Example}
\newtheorem{remark}[theorem]{Remark}

\numberwithin{equation}{section}
\graphicspath{{images/}}

\begin{document}
\title{Iteration at the Boundary of Newton Maps}
\author{Hongming Nie}
\address{Department of Mathematics, Indiana University\\
831 East Third Street, Bloomington, Indiana 47405, USA}
 \email{nieh@indiana.edu}
\date{\today}
\maketitle
\begin{abstract}
Let $\{N_t\}$ be a holomorphic family of degree $d\ge 3$ Newton maps. By studying the related Berkovich dynamics, we obtain an estimate of the weak limit of the maximal measures of $N_t$. Moreover, we give a complete description of the rescaling limits for $\{N_t\}$.
\end{abstract}

\section{Introduction}
Let $P$ be a degree $d\ge 2$ monic polynomial with distinct roots. It can be written as
$$P(z)=\prod_{i=1}^d(z-r_i),$$
where $r_1,\dots,r_d$ are distinct complex numbers. The Newton map of $P$ is defined by 
$$N_P(z)=z-\frac{P(z)}{P'(z)}.$$ 
We also write $N_{\{r_1,\cdots,r_d\}}$ to indicate the roots of $P$.\par
As an algorithm for finding the roots of a polynomial, it is also known as Newton's method, which is a generally convergent for quadratic polynomials. Cayley \cite{Cayley79} noticed the difficulties for Newton's method to find roots of higher degree polynomials. However, there is a finite universal set such that for any root of any suitably normalized polynomial of degree $d$, there is an initial point in the universal set whose orbit converges to this root under iterations of the corresponding Newton's method \cite{Hubbard01}. We refer \cite{Friedman89, Manning92, McMullen87, Schleicher16, Sutherland89} for more details. \par 
As a dynamical system, Newton map exhibits nice properties. The Julia sets of Newton maps are always connected \cite{Shishikura09}. Moreover, the cubic Newton maps have locally connected Julia sets except in some very specific case \cite{Roesch08}. For the combinatorial properties of Newton maps, every postcritically finite cubic Newton map can be constructed by mating two cubic polynomials \cite{Tan97}. For more dynamical properties of Newton maps, we refer \cite{Lodge15a, Lodge15b, Mikulich11}. We also refer \cite{Haeseler88} for an overview of Newton maps. And there also contains detailed background of known results for Newton maps in \cite{Roesch15}. \par
The goal of the present article is to investigate the limiting behaviors of a degenerate holomorphic family of Newton maps. We mainly focus on the weak limits of measures of maximal entropy and the rescaling limits. In the subsequent article \cite{Nie-2}, we will study the compactifications of the related moduli spaces. Now we set up the statement.\par
For $d\ge 2$, the space $\mathrm{Rat}_d$ of degree $d$ complex rational maps can be naturally identified to a dense open subset of $\mathbb{P}^{2d+1}$. And in the projective coordinates, for $f\in\mathrm{Rat}_d$, we can write
$$f([X:Y])=[F_a([X:Y]):F_b([X:Y])],$$
where $F_a([X:Y])$ and $F_b([X:Y])$ are two relatively prime homogeneous polynomials of degree $d$. Generally, for $f\in\mathbb{P}^{2d+1}$, it determines the coefficients for a pair of homogeneous polynomials of degree $d$. We can 
write
$$f([X:Y])=[F_a([X:Y]):F_b([X:Y])]=H_f(X,Y)[f_a([X:Y]):f_b([X:Y])],$$
where $H_f(X,Y)=\mathrm{gcd}(F_a([X:Y]),F_b([X:Y]))$ and $\hat f:=[f_a([X:Y]):f_b([X:Y])]$ is a rational map of degree at most $d$. The indeterminacy locus $I(d)\subset\mathbb{P}^{2d+1}$ is the collection of $f=H_f\hat f$, where $\hat f$ is a constant in $\mathbb{P}^1$ and this constant is a hole of $f$, that is, it is a zero of $H_f$.\par
Following DeMarco \cite{DeMarco05}, we can associate each $f\in\mathbb{P}^{2d+1}$ a probability measure $\mu_f$. If $f\in\mathrm{Rat}_d$, the measure $\mu_f$ is the unique measure such that the measure-theoretic entropy of $f$ attains its maximum. If $f=H_f\hat f\in\mathbb{P}^{2d+1}\setminus\mathrm{Rat}_d$ with $\deg\hat f\ge 1$, the measure $\mu_f$ is an atomic measure defined on the holes of $f$ and all their preimages under $\hat f$. If $f=H_f\hat f\in\mathbb{P}^{2d+1}\setminus\mathrm{Rat}_d$ with $\deg\hat f=0$, the measure $\mu_f$ is an atomic measure defined on the holes of $f$. Then the map $f\to\mu_f$ is continuous if and only if $f\not\in I(d)$ \cite[Theorem 0.1]{DeMarco05}.\par 
If we restrict to holomorphic families $\{f_t\}$, then function $f_t\to\mu_{f_t}$ is better behaved. Let $\mathbb{P}^1_{\mathrm{Ber}}$ be the Berkovich space over the completion of the field of formal Puiseux series and let $\mathbf{f}:\mathbb{P}^1_{\mathrm{Ber}}\to\mathbb{P}^1_{\mathrm{Ber}}$ be the associated map for $\{f_t\}$. If $f_t$ converges to a point in $I(d)$, by regarding the Berkovich dynamical system $(\mathbf{f},\mathbb{P}_{\mathrm{Ber}}^1)$ as the limit of dynamical systems $(f_t,\widehat{\mathbb{C}})$ as $t\to 0$, DeMarco and Faber \cite[Theorem B]{DeMarco14} proved the weak limit $\mu$ of measures $\mu_{f_t}$ exists and it is a countable sum of atoms. Later they \cite{DeMarco16} gave a method to compute the measure $\mu$ by introducing and studying what they call an analytically stable pair $(\mathbf{f},\Gamma)$, where $\Gamma\subset\mathbb{P}^1_{\mathrm{Ber}}$ is a finite set of type II points. Combining with \cite[Theorem 0.1]{DeMarco05}, the measure $\mu$ satisfying
\begin{equation}\label{measure-lower-bound}
\mu(\{c\})\ge\frac{d_c(f)}{d+d_c(f)},
\end{equation}
where $c$ is the constant value of $\hat f$ and $d_c(f)$ is its depth, that is, $d_c(f)$ is the multiplicity of $c$ as a zero of $H_f$.\par
We now apply this to Newton maps. Suppose $\{N_{r(t)}\}$ is a holomorphic family of degree $d\ge 2$ Newton maps with $r(t)=\{r_1(t),\cdots,r_d(t)\}$. If $N_{r(t)}$ converges to a point $N\in I(d)$, as $t\to 0$, then all $r_i(t)$s converge to $z=\infty$, see Lemma \ref{Newton-indeterminacy}. Thus $z=\infty$ is the unique hole of $N$ and the depth $d_\infty(N)=d$. Let $\mu$ be the weak limit of the maximal measures $\mu_t$ for $N_{r(t)}$. The inequality (\ref{measure-lower-bound}) implies
\begin{equation}\label{Newton-lower-bound-1/2}
\mu(\{\infty\})\ge\frac{1}{2}.
\end{equation}
But the lower bound in inequality (\ref{Newton-lower-bound-1/2}) is not sharp. Our first result is to give a better lower bound for $\mu(\{\infty\})$.
\begin{theorem}\label{theorem-measure}
Let $\{N_t\}$ be a holomorphic family of degree $d\ge 2$ Newton maps. Suppose $N_t$ converges to a point in $I(d)$. Let $\mu_t$ be the maximal measure for $N_t$ and let $\mu$ be the weak limit of measures $\mu_t$. Then 
\begin{equation}\label{Newton-lower-bound}
\mu(\{\infty\})\ge\frac{d}{2d-1}.
\end{equation}
In particular, if the Gauss point $\xi_g\in\mathbb{P}^1_{\mathrm{Ber}}$ is a periodic point for the associated map $\mathbf{N}$ for $\{N_t\}$, then $\mu=\delta_\infty$.
\end{theorem}
We point out that in general the lower bound in inequality (\ref{Newton-lower-bound}) is not sharp either. For example, for quadratic Newton maps, we always have $\mu=\delta_\infty$, see Corollary \ref{quadratic-Newton-measure}.\par 
To prove Theorem \ref{theorem-measure}, we consider the associated map $\mathbf{N}:\mathbb{P}_{\mathrm{Ber}}^1\to\mathbb{P}_{\mathrm{Ber}}^1$ for $\{N_t\}$. Let $V$ be the set of type II repelling fixed points of $\mathbf{N}$. Consider the convex hull $H_V^{\xi_g}$ of $V$ and the Gauss point $\xi_g$. If the orbit $\mathcal{O}(\xi_g)$ of $\xi_g$ under the map $\mathbf{N}$ does not intersects with $H_V^{\xi_g}$, we consider the vertices set $\Gamma=\{\xi_g\}$ and study the equilibrium $\Gamma$-measures on $\mathbb{P}_{\mathrm{Ber}}^1$. Then in this case we can prove $\mu=\delta_\infty$. If the orbit $\mathcal{O}(\xi_g)$ intersects with $H_V^{\xi_g}$, we construct the vertices sets $\Gamma$ such that the pairs $(\mathbf{N},\Gamma)$ are analytically stable and then apply \cite[Theorem C]{DeMarco16}.\par
Since the indetermiancy locus $I(d)$ is the set of points where the iterate maps $\mathbb{P}^{2d+1}\dasharrow\mathbb{P}^{2d^n+1}$ are discontinuous \cite[Lemma 2.2]{DeMarco05}, it is interesting to explore the limits of holomorphic families converging to points in $I(d)$ under iterations. More naturally, we work on the moduli space $\mathrm{rat}_d:=\mathrm{Rat}_d/\mathrm{PGL}_2(\mathbb{C})$, modulo the action by conjugation of the group of M\"obius transformations.\par
In quadratic case, the existence of such interesting limits was observed by Stimson \cite{Stimson93} in his analysis of the asymptotic behavior of some algebraic curves in $\mathrm{rat}_2$. The idea of rescaling limits also appeared in Epstein's work \cite{Epstein00} proving some hyperbolic components are precompact in $\mathrm{rat}_2$. Kiwi\cite{Kiwi15} defined rescaling limits explicitly. For a holomorphic family $\{f_t\}$, the rescaling limits arise as limits $M_t^{-1}\circ f_t^q\circ M_t\to g $ of rescaled iterates where the convergence is locally uniform outside some finite subset of $\mathbb{P}^1$. By studying the associated dynamical system $(\mathbf{f},\mathbb{P}^1_{\mathrm{Ber}})$ and relating the rescalings to the periodic repelling type II points in $\mathbb{P}^1_{\mathrm{Ber}}$, Kiwi proved for any given holomorphic one-parameter family of degree  $d\ge 2$ rational maps, there are at most $2d-2$ dynamically independent rescalings such that the corresponding rescaling limits are not postcritically finite \cite[Theorem 1]{Kiwi15}. Later, applying the trees of spheres, which is based on the Deligne-Mumford compactifications of the moduli spaces of stable curves, and considering the dynamical systems between trees of spheres, Arfeux proved the same results \cite[Theorem 1]{Arfeux17} about the interesting rescaling limits. Moreover, the existence of a periodic sphere for a dynamical cover between trees of spheres corresponds to the existence of a rescaling limit \cite[Theorem 2, Theorem 3]{Arfeux17}. Recently, Arfeux and Cui \cite{Arfeux16} announced for any arbitrary large integer $n$, there is a holomorphic family in $\mathrm{Rat}_d$ for $d\ge 3$ such that it has $n$ dynamically independent rescalings leading to postcritically finite and nonmonomial rescaling limits. \par 
For a holomorphic family of degree $d\ge 3$ Newton maps, we give a complete description for its rescaling limits.
\begin{theorem}\label{theorem-rescaling-limits-Newton}
\textit{Let $\{N_t\}$ be a holomorphic family of degree $d\ge 3$ Newton maps. Then 
\begin{enumerate}
\item up to equivalence, $\{N_t\}$ has at most $d-1$ rescalings of period $1$. Moreover, each such rescaling leads to a rescaling limit that is conjugate to a degenerated Newton map. 
\item $\{N_t\}$ has at most $d-3$ dynamically independent rescalings of periods at least $2$. Moreover, the corresponding rescaling limit for each such rescaling is conjugate to a polynomial of degree at most $2^{d-3}$.
\end{enumerate}}
\end{theorem}
Let $\mathbf{N}:\mathbb{P}^1_{\mathrm{Ber}}\to\mathbb{P}^1_{\mathrm{Ber}}$ be the associated map for $\{N_t\}$. The period $1$ rescalings for $\{N_t\}$ correspond to the type II repelling fixed points of $\mathbf{N}$. To prove part $(1)$, we show the type II repelling fixed points are the points that are the branch points in the convex hull whose endpoints are the type I fixed points of $\mathbf{N}$. To prove $(2)$, we first consider the ramification locus $\mathcal{R}_{\mathbf{N}}$ and give sufficient and necessary conditions for the existence of generalized rescalings of period at least $2$. Then we count the available critical points of $\mathbf{N}$. It will give us the upper bound of the rescalings of higher periods. The main ingredient we use to show the corresponding rescaling limits are conjugate to polynomials is Theorem \ref{theorem-measure}. If there exists a rescaling $\{M_t\}$ of period $q\ge 2$, then up to dynamical dependence, we can assume $M_t$ is affine. By the continuity of the composition maps, $M_t^{-1}\circ N_t\circ M_t$ converges to a point in $I(d)$. We consider the associated map $\widetilde{\mathbf{N}}=\mathbf{M}^{-1}\circ\mathbf{N}\circ\mathbf{M}$ of the holomorphic family $\{M_t^{-1}\circ N_t\circ M_t\}$. Then the Gauss point $\xi_g$ is a periodic point of $\widetilde{\mathbf{N}}$. Thus, by Theorem \ref{theorem-measure}, the weak limit of the maximal measures of $M_t^{-1}\circ N_t\circ M_t$ is the measure $\delta_\infty$. By the invariance of the maximal measures and the continuity of the map $f\to\mu_f$ if $f\not\in I(d)$, we know the measure associated to the subalgebraic limit for $M_t^{-1}\circ N^q_t\circ M_t$ is $\delta_\infty$. Hence, we can show the rescaling limit is a polynomial. We mention here that $(2)$ can also be obtained from the Berkovich dynamics of the map $\mathbf{N}$, see Remark \ref{dynamics-poly}. Moreover, the subalgrbraic limit of $M_t^{-1}\circ N^q_t\circ M_t$ has a unique hole at $\infty$. Therefore, the upper bound of the degree of the rescaling limits implies this subalgrbraic limit is not semistable, see Proposition \ref{rescaling-limit-not-semistable}. \par 
\subsection*{Outline}
In section \ref{preliminaries}, we recall the relevant backgrounds of Newton maps and prove there is a unique indeterminacy point in boundary of the space of Newton maps, see Proposition \ref{Newton-closure-indeterminacy-singleton}. Section \ref{Berkovich} is devoted to describe the basic Berkovich dynamics of Newton maps. It also contains a backgrounds of Berkovich spaces and related dynamics of rational maps. The goal of section \ref{measure} is to study the degenerate holomorphic families of Newton maps in measure-theoretic view. We first state DeMarco and Faber's results about limiting measures, and then prove Theorem \ref{theorem-measure}. Finally we prove Theorem \ref{theorem-rescaling-limits-Newton} and related results in section \ref{rescaling}. Kiwi's results are included in this section.

\section{Preliminaries}\label{preliminaries}
\subsection{Iterate Maps on $\mathrm{Rat}_d$}
For $d\ge 1$, consider the map 
$$\Psi_n:\mathrm{Rat}_d\to\mathrm{Rat}_{d^n},$$
sending $f$ to its $n$-th iterate $f^n$. Then $\Psi_n$ is proper for all $n\ge 2$ and all $d\ge 2$ \cite[Corollary 0.3]{DeMarco05}. The iterate map $\Psi_n$ extends to a rational map
$$\Psi_n:\mathbb{P}^{2d+1}\dasharrow\mathbb{P}^{2d^n+1}.$$
Define $I(d)\subset\mathbb{P}^{2d+1}$ is the collection of $f=H_f\hat f$ with $\deg\hat f=0$ and the constant value of $\hat f$ is a hole of $f$, that is 
$$I(d)=\{f=H_f\hat f\in\mathbb{P}^{2d+1}:\hat f\equiv h\in\mathbb{P}^1, H_f(h)=0\}.$$\par
The following result, due to DeMarco, claims the indeterminacy loci of the maps $\Psi_n$ are $I(d)$ and gives the formula for iteration.
\begin{lemma}\cite[Lemma 2.2]{DeMarco05}\label{iterate-indeterminacy}
For $d\ge 1$ and $n\ge 2$, the indeterminacy locus $I(\Psi_n)$ of the iterate map $\Psi_n:\mathbb{P}^{2d+1}\dasharrow\mathbb{P}^{2d^n+1}$ is $I(d)$. Moreover, for $f=H_f\hat f\in\mathbb{P}^{2d+1}\setminus I(d)$, 
$$f^n=\Psi_n(f)=\left(\prod_{k=0}^{n-1}(H_f\circ\hat f^k)^{d^{n-k-1}}\right)\hat f^n.$$
\end{lemma}
More general, for the composition maps, we have 
\begin{lemma}\cite[Lemma 2.6]{DeMarco07}\label{composition}
The composition map 
$$\mathcal{C}_{d,e}:\mathbb{P}^{2d+1}\times\mathbb{P}^{2e+1}\dasharrow\mathbb{P}^{2de+1},$$
which sends a pair $(f,g)$ to the composition $f\circ g$ is continuous away from
$$I(d,e)=\{(f,g)=(H_f\hat f,H_g\hat g):\hat g\equiv c\in\mathbb{P}^1, H_f(c)=0\}.$$
Moreover, for $(f,h)=(H_f\hat f,H_g\hat g)\in\mathbb{P}^{2d+1}\times\mathbb{P}^{2e+1}\setminus I(d,e)$,
$$\mathcal{C}_{d,e}(f,g)=(H_g)^dH_f(P_g,Q_g)\hat f\circ\hat g,$$
where $\hat g=[P_g:Q_g]$.
\end{lemma}

\subsection{Convergence}
In this subsection, we consider various notations of convergence of maps in the space of (degenerate) rational maps. Mainly, we define the subalgebraic convergence and show the relations between subalgebraic convergence and locally uniform convergence on the complement of a finite set in $\mathbb{P}^1$.\par
First recall that for $t\in\mathbb{D}$, a collection $\{f_t\}\subset\mathbb{P}^{2d+1}$ is a 1-dimensional holomorphic family of degree $d\ge 1$ rational maps if the map $\mathbb{D}\to\mathbb{P}^{2d+1}$, sending $t$ to $f_t$, is a holomorphic map such that $f_t\in\mathrm{Rat}_d$ if $t\not=0$. We say $\{f_t\}$ is degenerate if $f_0\not\in\mathrm{Rat}_d$. A holomorphic family of degree $1$ rational maps is called a moving frame.\par
\begin{definition}
For a holomorphic family $\{f_t\}\subset\mathbb{P}^{2d+1}$, we say $f_t$ converges to $f=H_f\hat f\in\mathbb{P}^{2d+1}$ subalgebraically if $f_t$ converges to $f$ in $\mathbb{P}^{2d+1}$ as $t\to 0$. If, in addition, $\deg\hat f=d$, we say $f_t$ converges to $f\in\mathbb{P}^{2d+1}$ algebraically.
\end{definition} 
Let's use the notations $f_t\xrightarrow{sa} f$ and $f_t\xrightarrow{a}f$ for subalgebraic convergence and algebraic convergence, respectively. Moreover, for a finite subset $S\subset\mathbb{P}^1$, denote $f_t\xrightarrow{\bullet}\hat f$ on $\mathbb{P}^1\setminus S$ if $f_t$ converges to $f$ locally uniformly on $\mathbb{P}^1\setminus S$, and denote $f_t\rightrightarrows\hat f$ if $f_t$ converges to $f$ uniformly on $\mathbb{P}^1$. \par
\begin{lemma}\cite[Lemma 3.2 ]{Kiwi15}\label{subalg-imply-locally-uniform} 
For a holomorphic $\{f_t\}\subset\mathrm{Rat}_d$, suppose that $f_t\xrightarrow{sa} f=H_f\hat f$. Then $f_t\xrightarrow{\bullet}\hat f$ on $\mathbb{P}^1\setminus\{h:H_f(h)=0\}$.
\end{lemma}
Conversely, we have
\begin{lemma}\label{locally-uniform-imply-subalg}
Let $\{f_t\}\subset\mathrm{Rat}_d$ be a holomorphic family. Suppose that $f_t\xrightarrow{\bullet}\hat f\in\mathrm{Rat}_e$ on $\mathbb{P}^1\setminus S$, where $S\subset\mathbb{P}^1$ is finite. Then there exists a degree $d-e$ homogeneous polynomial $H$ such that $f_t\xrightarrow{sa} H\hat f$ with $\{h: H(h)=0\}\subset S$.
\end{lemma}
\begin{proof}
Suppose that $f_t\xrightarrow{sa}g=H_g\hat g$, as $t\to 0$. Then by Lemma \ref{subalg-imply-locally-uniform}, we have $f_t\xrightarrow{\bullet}\hat g$ on $\mathbb{P}^1\setminus\{h:H_g(h)=0\}$. Thus $\hat g=\hat f$ on $\mathbb{P}^1\setminus(\{h:H_g(h)=0\}\cup S)$. So $\hat g=\hat f$ on $\mathbb{P}^1$. We get $\{h:H_g(h)=0\}\subset S$. Thus, $g=H_g\hat g=H_g\hat f$.
\end{proof}
\begin{corollary}Let $\{f_t\}\subset\mathrm{Rat}_d$ be a holomorphic family. Then $f_t\xrightarrow{a}f$ if and only if $f_t\rightrightarrows f$.
\end{corollary}
We use the following example to illustrate the subalgebraical convergence.\par
\begin{example}
Quadratic rational maps.\par
Here we consider the fixed-point normal form for quadratic rational maps \cite[Appendix C]{Milnor93}.
$$f(z)=\frac{z(z+\mu_1)}{\mu_2z+1}$$
with $\mu_1\mu_2\not=1$. Then $z=0$ and $z=\infty$ are fixed points of $f(z)$ with multipliers $\mu_1$ and $\mu_2$, respectively. And the third fixed point is $z=(1-\mu_1)/(1-\mu_2)$ with multiplier $\mu_3=(2-\mu_1-\mu_2)/(1-\mu_1\mu_2)$. \par
In projective coordinates, we can write 
$$f([X:Y])=[X^2+\mu_1XY:\mu_2XY+Y^2].$$
Then $f=[1:\mu_1:0:0:\mu_2:1]\in\mathbb{P}^5$. Consider the holomorphic family $f_t$ given by
$$f_t=[1:\mu_1(t):0:0:\mu_2(t):1]\in\mathbb{P}^5$$
with $\mu_1(t)\mu_2(t)\to 1$. Then $\{f_t\}$ is degenerate. Suppose $f_t\xrightarrow{sa}g=H_g\hat g$ as $t\to 0$. We have 
$$g([X:Y])=H_g(X,Y)\hat g([X:Y])=
\begin{cases}
XY[1:0], &\text{if}\ \mu_1(t)\to\infty,\\
XY[0:1], &\text{if}\ \mu_1(t)\to 0,\\
(X+cY)[X:Y], &\text{if}\ \mu_1(t)\to c\in\mathbb{C}\setminus\{0\},
\end{cases}$$
and
$$f_t([X:Y])\xrightarrow{\bullet}
\begin{cases}
[1:0]\ \text{on}\ \mathbb{P}^1\setminus\{[1:0],[0:1]\} &\text{if}\ \mu_1(t)\to\infty,\\
[0:1]\ \text{on}\ \mathbb{P}^1\setminus\{[1:0],[0:1]\} &\text{if}\ \mu_1(t)\to 0,\\
[X:Y]\ \text{on}\ \mathbb{P}^1\setminus\{[-c:1]\} &\text{if}\ \mu_1(t)\to c\in\mathbb{C}\setminus\{0\}.
\end{cases}$$ 
\end{example}

\subsection{(Degenerate) Newton Maps}
In this subsection, we state some basic background of the (degenerate) Newton maps. For $d\ge 2$, let $P(z)$ be a degree $d$ monic polynomial. Let $r=\{r_1,\cdots,r_d\}$ be the roots of $P(z)$ counted with multiplicity. Then $P(z)$ is uniquely determined by $r$. We may write $P_r(z)$ instead of $P(z)$ to emphasis its roots. In projective coordinates, we can write 
$$P([X:Y])=P_r([X:Y])=[F_a(X,Y):F_b(X,Y)]=\left[\prod_{i=1}^d(X-r_iY):Y^d\right].$$
Then the derivative $P'(z)$ of $P(z)$ can be written as
$$P'([X:Y])=P'_r([X:Y])=[G_a(X,Y):G_b(X,Y)]=\left[\sum_{j=1}^d\prod_{i\not=j}(X-r_iY):Y^{d-1}\right].$$
Define
\begin{align}
\label{Newton projective}
N_r([X:Y])=N_P([X:Y]):=[XG_a(X,Y)-F_a(X,Y):YG_a(X,Y)]\in\mathbb{P}^{2d+1}.
\end{align}
We say $N_r$ is a Newton map of degree $d$ if $N_r\in\mathrm{Rat}_d\subset\mathbb{P}^{2d+1}$. In other word, $N_r$ is a Newton map of degree $d$ if and only if $r_1,\cdots,r_d$ are $d$ distinct points in $\mathbb{C}$. In this case, in affine coordinates, we can write $N_r:\widehat{\mathbb{C}}\to\widehat{\mathbb{C}}$ by 
\begin{align}
\label{Newton affine}
N_r(z)=N_{P}(z)=z-\frac{P(z)}{P'(z)},
\end{align}
and say $N_r(z)$ is the Newton map of the polynomial $P_r(z)$.\par 
For a point $N\in\mathbb{P}^{2d+1}\setminus\mathrm{Rat}_d$, we say $N$ is a degenerate Newton map of degree $d$ if there exists a sequence $\{N_n\}$ of Newton maps of degree $d$ such that $N_n$ converges to $N$ in $\mathbb{P}^{2d+1}$, as $n\to\infty$. Let $N$ be a degenerate Newton map of degree $d$. Then there exist an integer $d\ge m\ge 0$ and a degree $d-m$ polynomial $\widetilde P$ such that 
\begin{align}
\label{Newton degenerate}
N([X:Y])=Y^mN_{\widetilde P}([X:Y])\in\mathbb{P}^{2d+1}.
\end{align}
If $m\not=0$, then $\infty$ is a hole of $N$ with depth $d_\infty(N)=m$. All other holes of $N$ are the  multiple roots $r_i$ of $\widetilde P$ and the corresponding depths $d_{r_i}(N)$ are $m_{\widetilde P}(r_i)-1$, where $m_{\widetilde P}(r_i)$ is the multiplicity of $r_i$ as a zero of $\widetilde P$. Note if $m=0$, then $\widetilde P$ is a degree $d$ polynomial with multiple roots, and the formula (\ref{Newton degenerate}) coincides to formula (\ref{Newton projective}). \par
In general, let $r$ be a set of $d$ points (not necessary distinct) in $\widehat{\mathbb{C}}$. Suppose there are $m\ge 0$ many $\infty$s in $r$. Let $\widetilde P$ be the polynomial of degree $d-m$ whose roots coincide to the points in $r$ that are not $\infty$. Then we define $N_r$ to be the (degenerate) Newton map $N_r$ by formula (\ref{Newton degenerate}).
\begin{example} 
Let $r=\{0,0,1,\infty\}$. Then $N_r$ has holes at $0$ and $\infty$, and each hole has depth $1$. More precisely, 
$$N_r([X:Y])=H_{N_r}(X,Y)\widehat{N}_r([X:Y])=XY[2X^2-XY:3XY-2Y^2].$$
Note here we have $\widehat{N}_r=\widehat{N}_{\{0,0,1,\infty\}}=\widehat{N}_{\{0,0,1\}}$.
\end{example}
For $f\in\mathrm{Rat}_d$, denote by $\mathrm{Fix}(f)$, $\mathrm{Zero}(f)$ and $\mathrm{Crit}(f)$ for the set of fixed points, the set of zeros and the set of critical points of $f$, respectively. Now we state some basic facts about the Newton map $N_r(z)$.
\begin{proposition}
For $d\ge 2$, let $r=\{r_1,\cdots,r_d\}$ be a set of $d$ distinct points in $\mathbb{C}$, and let $N_r$ be the Newton map of the polynomial $P_r$. Then 
\begin{enumerate}
\item The set of fixed points of $N_r$ is
$$\mathrm{Fix}(N_r)=\mathrm{Zero}(P_r)\cup\{\infty\}=\{r_1,\cdots, r_d,\infty\}.$$
All $r_i$'s are superattracting fixed points and $z=\infty$ is the unique repelling fixed point. Moreover, the multiplier of $N_{r}$ at $z=\infty$ is $d/(d-1)$.\par
\item The derivative of $N_r$ is 
$$N'_r(z)=\frac{P_r(z)P''_r(z)}{(P'_r(z))^2}.$$
In particular, $\mathrm{Crit}(N_r)=\mathrm{Zero}(P_r)\cup\mathrm{Zero}(P''_r).$
\end{enumerate} 
\end{proposition}
\begin{remark}\label{nondegenrate-Newton}
If $N_r$ is a degenerate Newton map, then we still have 
$$\mathrm{Fix}(\widehat{N}_r)=\mathrm{Zero}(P_r)\cup\{\infty\}=\{r_1,\cdots, r_d,\infty\}.$$
But $r_i$ may not be a superattracting fixed point. In fact, it is an attracting fixed point with multiplier $(m_{P_r}(r_i)-1)/d$, where $m_{P_r}(r_i)$ is the multiplicity of $r_i$ as a zero of $P_r$. Moreover, the depth $d_{r_i}(N_r)=m_{P_r}(r_i)-1$.
\end{remark}
For a Newton map $N_r$ of degree $d$ with $r=\{r_1,\cdots,r_d\}$, at each $r_i$, there is a neighborhood $U_i$ such that for each $z\in U_i$, $N^k_r(z)\to r_i$ as $k\to\infty$. The B$\ddot{\text{o}}$ttcher's theorem \cite[Theorem 9.1]{Milnor06B} implies that this convergence is at least quadratic. We say a critical point $c\in\mathrm{Crit}(N_r)$ is free if $c\in\mathrm{Zero}(P''_r)\setminus\mathrm{Zero}(P_r)$. If $N_r$ has an attracting cycle of period at least $2$, then at least one free critical point is attracted to this cycle.\par
In the case $d=2$, we can conjugate $N\in\mathrm{NM}_2$ by some $M\in \mathrm{PGL}_2(\mathbb{C})$ such that the two superattracting fixed points of $M^{-1}\circ N\circ M$ are $0$ and $\infty$. This fact is known by Cayley, see \cite{Alexander94} for more history. Thus, we have
\begin{proposition} 
Any quadratic Newton map $N_r(z)$ is conjugate to $z\to z^2$.
\end{proposition}

\subsection{The Space $\mathrm{NM}_d$}
For $d\ge 2$, denote by $\mathrm{NM}_d$ the set of degree $d$ Newton maps. Then 
$$\mathrm{NM}_d\subset\mathrm{Rat}_d\subset\mathbb{P}^{2d+1}.$$ 
Let $\overline{\mathrm{NM}}_d\subset\mathbb{P}^{2d+1}$ be the closure of $\mathrm{NM}_d$ in $\mathbb{P}^{2d+1}$.\par
Let $r=\{r_1,\cdots,r_d\}\subset\mathbb{C}$ be a set of $d$ distinct points. Then 
$$P_r(z)=\prod\limits_{i=1}^d(z-r_i)=\sum\limits_{k=0}^d(-1)^k\sigma_kz^{d-k}$$
and 
$$N_r(z)=z-\frac{1}{\sum\limits_{i=1}^d\frac{1}{z-r_i}}=\frac{\sum\limits_{k=0}^d(-1)^k(d-k-1)\sigma_kz^{d-k}}{\sum\limits_{k=0}^{d-1}(-1)^k(d-k)\sigma_kz^{d-k-1}},$$
where the $\sigma_j$'s are the elementary symmetric functions of $r_i$'s.
In projective coordinates, we can express the Newton map $N_r$ as following
\begin{align}
\label{Newton formula}
\begin{split}
 &N_r([X:Y])\\
 &=\left[\sum\limits_{k=0}^d(-1)^k(d-k-1)\sigma_kX^{d-k}Y^k:\sum\limits_{k=0}^{d-1}(-1)^k(d-k)\sigma_kX^{d-k-1}Y^{k+1}\right]\\
&=[(d-1)\sigma_0:-(d-2)\sigma_1:\cdots:(-1)^k(d-k-1)\sigma_k:\cdots:0:(-1)^{d+1}\sigma_d:\\
&\ \ \ \ \ \ 0:d\sigma_0:\cdots:(-1)^k(d-k)\sigma_k:\cdots:(-1)^{d-1}\sigma_{d-1}]\in\mathbb{P}^{2d+1}.
\end{split}
\end{align}\par
\begin{remark}
\begin{enumerate}
\item The formula (\ref{Newton formula}) of $N_r([X:Y])$ also works for the degenerate Newton map of degree $d$ for which $\infty$ is not a hole.
\item In general, let $r=\{r_1,\cdots,r_d\}$ be the set of $d$ points (not necessary distinct) in $\widehat{\mathbb{C}}$. Suppose there are $d_\infty$ many $r_i$s that are $\infty$. Then we can write $N_{r}([X:Y])=Y^{d_\infty}N_{\tilde r}([X:Y])$, where $\tilde r\subset r$ is the subset consisting of $r_i\in\mathbb{C}$. The formula \ref{Newton formula} gives the projective coordinate of  $N_{\tilde r}([X:Y])$. Hence, we can get the projective coordinates of $N_{r}$.
\end{enumerate}
\end{remark}
Since a holomorphic family of rational maps is a family in which the coefficients may be expressed as holomorphic functions of the parameter $t$, the locations of the fixed points are thus algebraic functions of this parameter. It is well-known that algebraic functions of one complex variable t are given by Puiseux series in $t$, see \cite{Ruiz}. Thus by a change in the parameterization of the form $t=s^k$ we may assume that the fixed points themselves are holomorphic functions of one parameter. In summary, when considering holomorphic families of Newton maps, it suffices for our purposes to consider holomorphic families where the roots $r_i(t)$ are holomorphic functions of $t$.\par 
Now we state an easy lemma to show the subalgebraic limit of a holomorphic family $\{N_{r(t)}\}$ of Newton maps. 
\begin{lemma}\label{Newton-limit}
Let $\{N_{r(t)}(z)\}$ be a holomorphic family of degree $d\ge 2$ Newton maps with $r(t)=\{r_1(t), \cdots, r_d(t)\}$. Suppose $r_{i}(t)\to r_{i}\in\widehat{\mathbb{C}}$ as $t\to 0$, then in projective coordinates, 
$$N_{r(t)}([X:Y])\xrightarrow{sa}N_{r}([X:Y]),$$
where $r=\{r_1,\cdots,r_d\}$.
\end{lemma}
We use the following example to illustrate Lemma \ref{Newton-limit}. And we draw the corresponding dynamics plane by the software ``fractalstream".
\begin{example}
\begin{enumerate}
\item Let $r(t)=\{0,1,i/t\}$. Consider the Newton map $N_{r(t)}$. As $t\to 0$, we have 
$$N_{r(t)}([X:Y])\xrightarrow{sa}N_{\{0,1,\infty\}}([X:Y]).$$
\item Let $\tilde{r}(t)=\{0,1,ti\}$. Consider the Newton map $N_{\tilde{r}(t)}$. As $t\to 0$, we have 
$$N_{\tilde{r}(t)}([X:Y])\xrightarrow{sa}N_{\{0,1,0\}}([X:Y]).$$
\end{enumerate}
\end{example}
\begin{figure}[h!]
  \centering
  \begin{minipage}[b]{0.3\textwidth}
   \includegraphics[width=\textwidth]{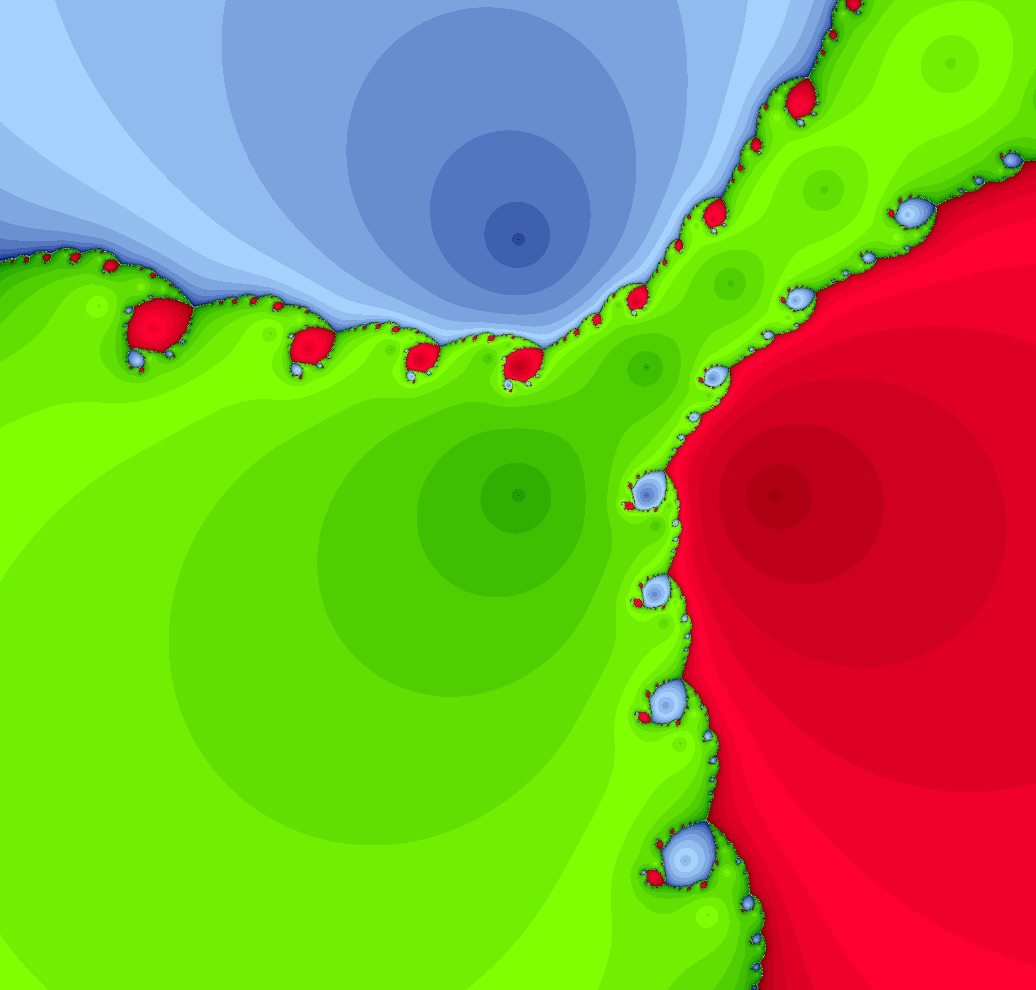}
  \end{minipage}
  \begin{minipage}[b]{0.3\textwidth}
   \includegraphics[width=\textwidth]{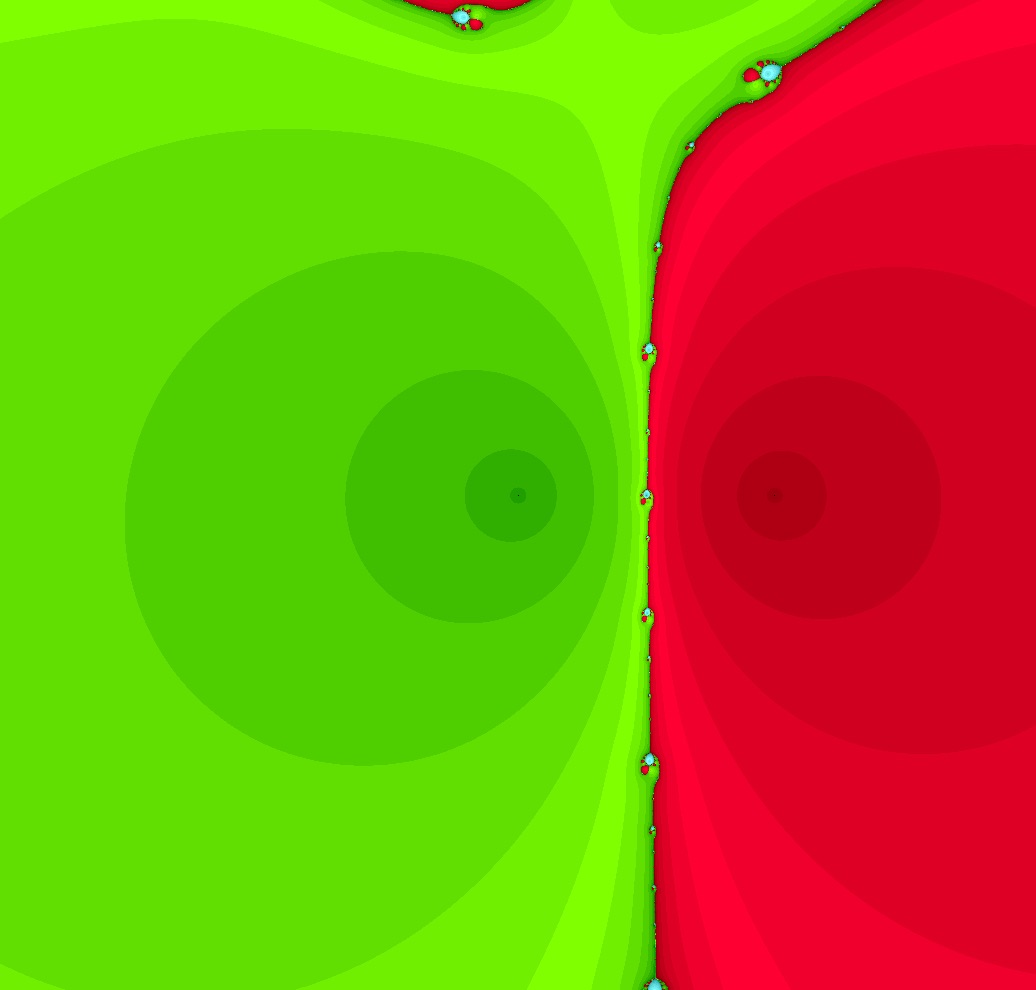}
  \end{minipage}
 \begin{minipage}[b]{0.3\textwidth}
    \includegraphics[width=\textwidth]{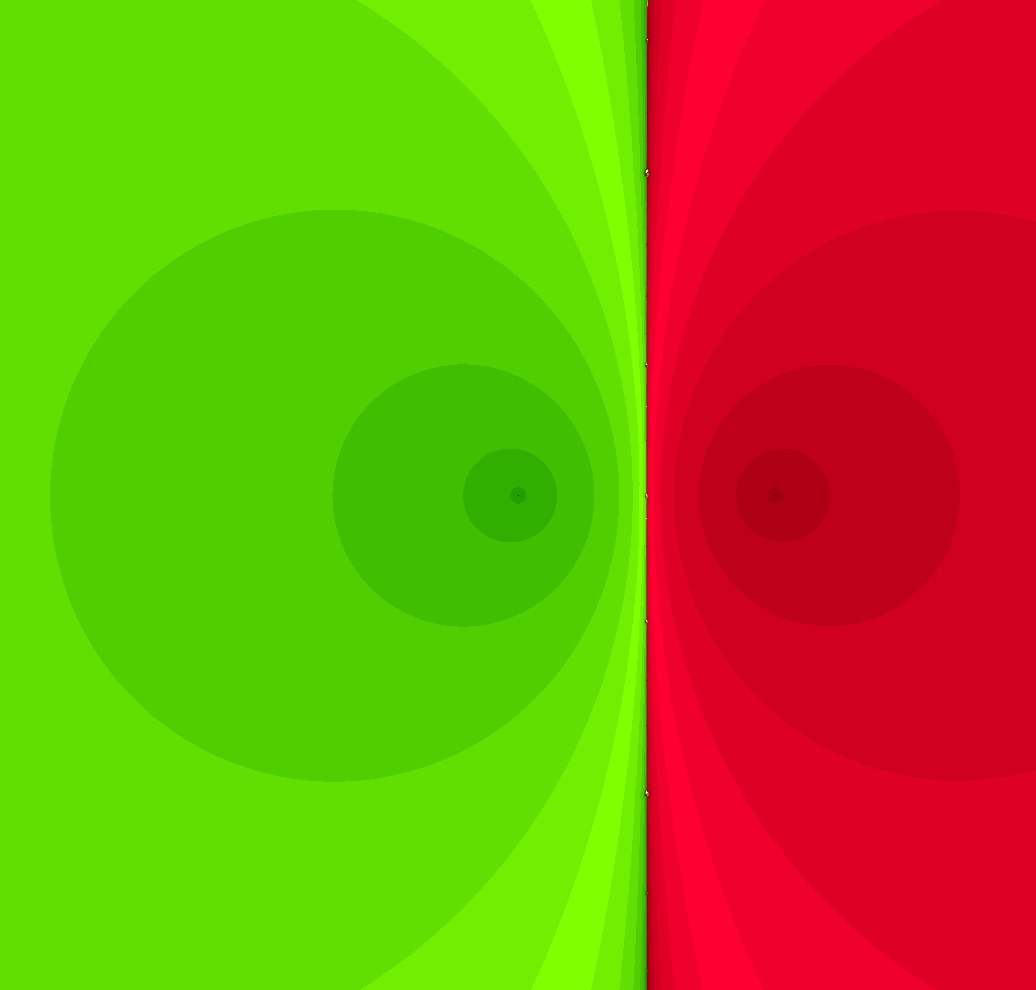}
     \end{minipage}
\caption{From left to right, the dynamical plane of the cubic Newton map of the polynomial $P(z)=z(z-1)(z-ti)$ for $t=1, 5, 25$.}
 \label{fig: Cubic Julia Up}
\end{figure}

\begin{figure}[h!]
  \centering
  \begin{minipage}[b]{0.3\textwidth}
    \includegraphics[width=\textwidth]{i}
  \end{minipage}
  \begin{minipage}[b]{0.3\textwidth}
    \includegraphics[width=\textwidth]{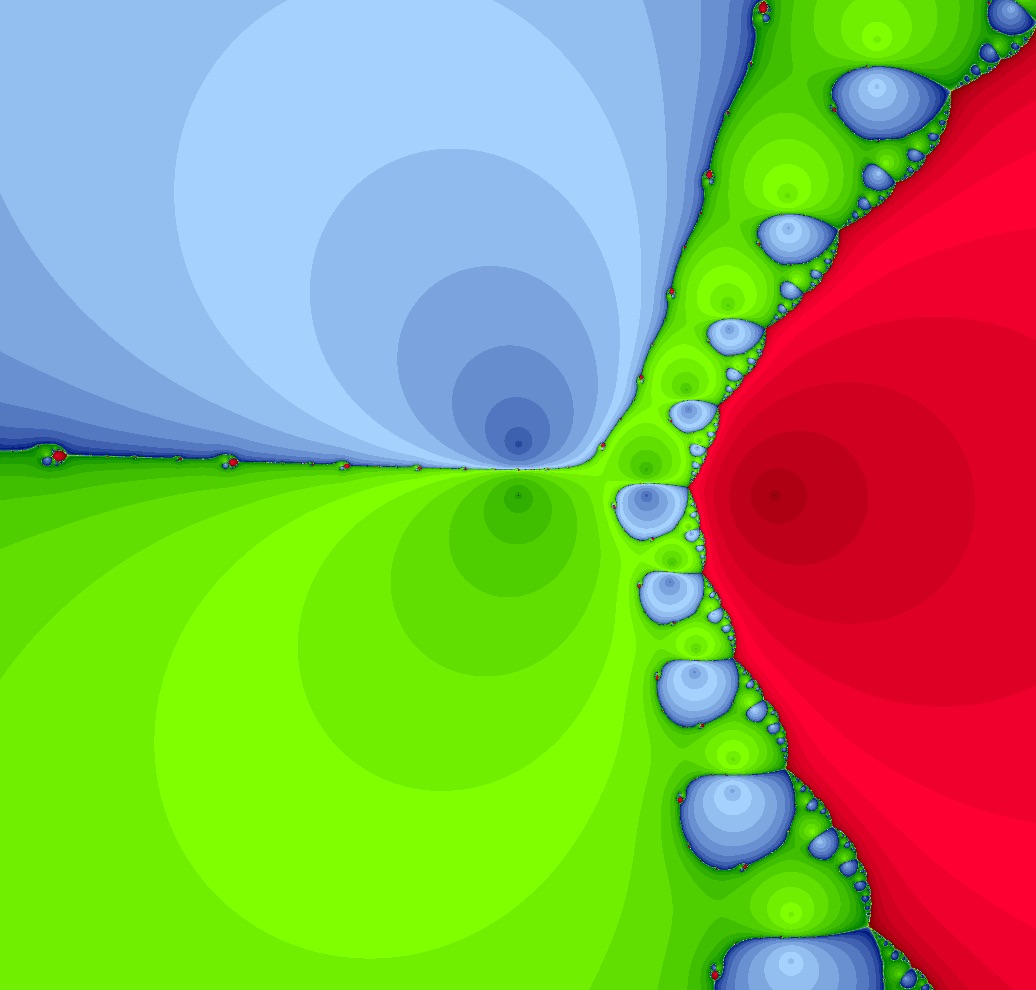}
  \end{minipage}
 \begin{minipage}[b]{0.3\textwidth}
    \includegraphics[width=\textwidth]{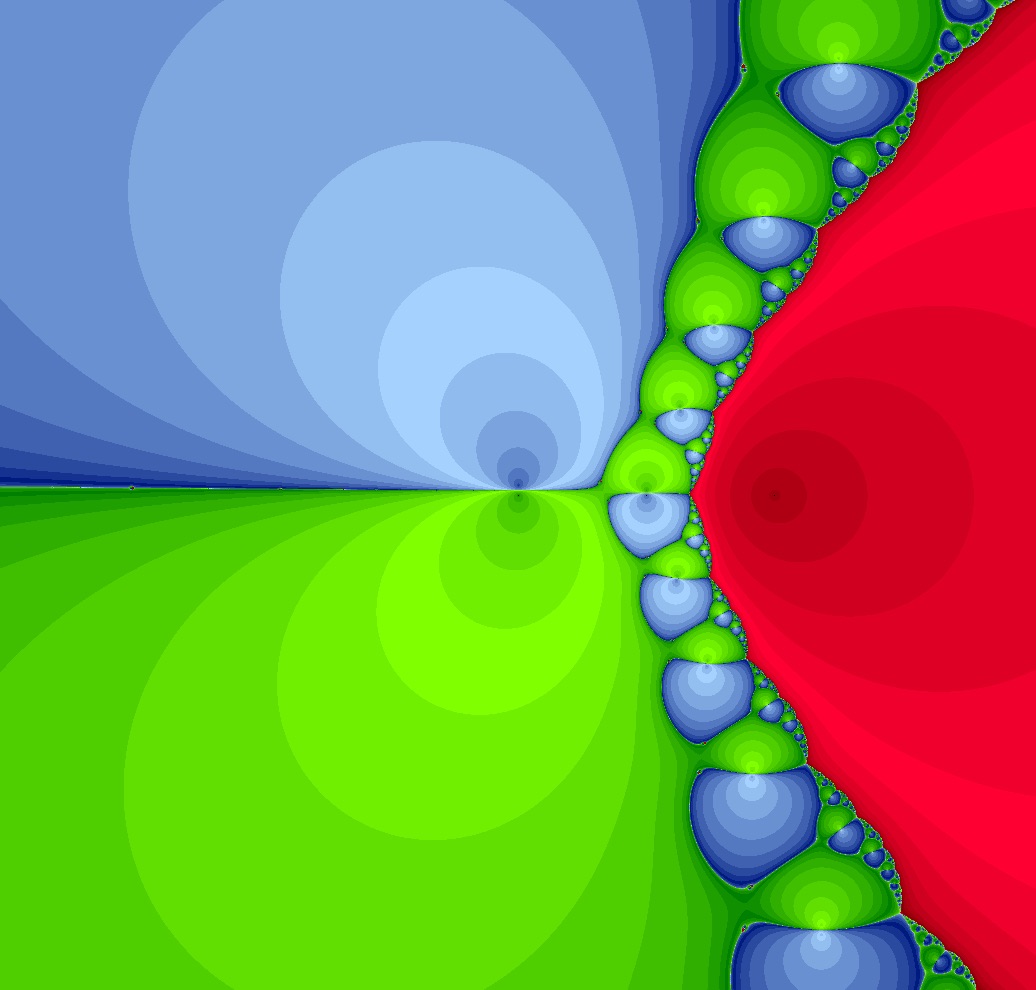}
      \end{minipage}
\caption{From left to right, the dynamical plane for the cubic Newton map of the polynomial $P(z)=z(z-1)(z-ti)$ for $t=1, 1/5, 1/25$.}
\label{fig: Cubic Julia Down}
\end{figure}
Recall $I(d)\subset\mathbb{P}^{2d+1}$ is the indeterminacy locus. Then 
\begin{lemma}\label{Newton-indeterminacy}
Let $\{N_{r(t)}\}$ be a holomorphic family of degree $d\ge 2$ Newton maps. Suppose $N_{r(t)}\xrightarrow{sa} N=H_N\widehat N$ as $t\to 0$. Then $N\in I(d)$ if and only if $r_i(t)\to\infty$ for all $i=1,\cdots,d$.
\end{lemma}
\begin{proof}
If $r_i(t)\to\infty$ for $i=1,\cdots,d$, then as $t\to 0$, we have $\sigma_d(t)\to\infty$ and for $j=1,\cdots,d-1$,
$$\frac{\sigma_j(t)}{\sigma_d(t)}\to 0.$$
Hence, by formula (\ref{Newton formula}), as $t\to 0$,
\begin{align*}
&N_{r(t)}([X:Y])\\
&=[(d-1)\sigma_0(t):-(d-2)\sigma_2(t):\cdots:(-1)^k(d-k-1)\sigma_k(t):\cdots:0:(-1)^{d+1}\sigma_d(t):\\
                              &\ \ \ \ \ \ 0:d\sigma_0(t):\cdots:(-1)^k(d-k)\sigma_k(t):\cdots:(-1)^{d-1}\sigma_{d-1}(t)]\\
                              &\to [0:0:\cdots:0:1:0:0:\cdots:0]\in\mathbb{P}^{2d+1}.
\end{align*}
Thus, $N([X:Y])=[Y^d:0]=Y^d[1:0]\in I(d)$.\par
Conversely, let $E=\{i: r_i(t)\not\to\infty\}$. By Lemma \ref{Newton-limit}, $\#E\le 1$ since $\deg\widehat N=0$. If $\#E=1$, suppose $i\in E$ and $r_i(t)\to a\in\mathbb{C}$, as $t\to 0$. Then $f([X:Y])=Y^d[a:1]$. Thus, in this case $N\not\in I(d)$. So $\#E=0$, that is, all $r_i(t)\to\infty$, as $t\to 0$.
\end{proof}
\begin{corollary}\label{Newton-unique-hole-infty}
If $N=H_N\widehat N\in\partial\overline{\mathrm{NM}}_d$ with $\deg\widehat N=0$, then $\mathrm{Hole}(N)=\{\infty\}$.
\end{corollary}
\begin{corollary} 
Let $\{N_{r(t)}\}$ be a holomorphic family of degree $d\ge 2$ Newton maps with $r(t)=\{r_1(t),\cdots,r_d(t)\}$. Suppose $r_{i_0}(t)\not\to\infty$ for some $1\le i_0\le d$ as $t\to 0$, and suppose $N_{r(t)}\xrightarrow{sa} N=H_N\widehat N$. Then for any $q\ge 1$, $N^q_{r(t)}\xrightarrow{sa}N^q$ and hence $N^q_{r(t)}\xrightarrow{\bullet}\widehat{N}^q$ outside a finite set.
\end{corollary}
\begin{proof}From Lemma \ref{Newton-indeterminacy}, we have $f\not\in I(d)$. It follows the Lemmas \ref{iterate-indeterminacy} and \ref{subalg-imply-locally-uniform}.
\end{proof}
Lemma \ref{Newton-indeterminacy} implies immediately the following fundamental result.
\begin{proposition}\label{Newton-closure-indeterminacy-singleton}
For $d\ge 2$,
$$\overline{\mathrm{NM}}_d\cap I(d)=\{Y^d[1:0]\}.$$
\end{proposition}

\section{Berkovich Dynamics}\label{Berkovich}

\subsection{Berkovich Space and Related Dynamics}
In this subsection, we summarize the definitions and main properties of Berkovich spaces and rational maps on Berkovich spaces used in this work. For more details, we refer \cite{Baker10,Berkovich90,Benedetto10,Jonsson15,Rivera03,Rivera05}.\par
Let $\mathbb{C}\{\{t\}\}$ be the field of formal Puiseux series over $\mathbb{C}$. Then $\mathbb{C}\{\{t\}\}$ can be equipped with a non-Archimedean absolute value $|\cdot|_{\mathbb{C}\{\{t\}\}}$ by defining 
$$\left|\sum_{n\in\mathbb{Z}}a_nt^{\frac{n}{m}}\right|_{\mathbb{C}\{\{t\}\}}=e^{-\frac{n_0}{m}},$$ 
where $n_0$ is the smallest integer $n\in\mathbb{Z}$ such that $a_n\not=0$. For more details of non-Archimedean fields, we refer \cite{Diagana16}. Then the field $\mathbb{C}\{\{t\}\}$ is algebraically closed but not complete. Let $\mathbb{L}$ be the completion of the field $\mathbb{C}\{\{t\}\}$, see \cite{Kiwi06,Kiwi14,Kiwi15}. It consists of all formal sums of the form $\sum_{n\ge 0}a_nt^{q_n}$, where $\{q_n\}$ is a sequence of rational numbers increasing to $\infty$ and $a_n\in\mathbb{C}$. Let $|\cdot|_\mathbb{L}$ be the extension of the non-Archimedean absolute value $|\cdot|_{\mathbb{C}\{\{t\}\}}$ so that $|a_0t^{q_0}+\cdots|_{\mathbb{L}}=e^{-q_0}$. To abuse of notations, we may write $|\cdot|$ for $|\cdot|_{\mathbb{L}}$. Then the ring of integers of the field $\mathbb{L}$ is 
$$\mathcal{O}_\mathbb{L}=\{z\in\mathbb{L}:|z|\le 1\}=\left\{\sum\limits_{n\ge 0}a_nt^{q_n}:q_n\ge 0\right\}$$
and the unique maximal ideal $\mathcal{M}_\mathbb{L}$ of $\mathcal{O}_\mathbb{L}$ consists of series with zero constant term, i.e.
$$\mathcal{M}_\mathbb{L}=\{z\in\mathbb{L}:|z|<1\}=\left\{\sum\limits_{n\ge 0}a_nt^{q_n}:q_n> 0\right\}.$$
The residue field $\mathcal{O}_\mathbb{L}/\mathcal{M}_\mathbb{L}$ is canonically isomorphic to $\mathbb{C}$.\par
Let $|\mathbb{L}^\times|$ be the value group of $\mathbb{L}$. Given $a\in\mathbb{L}$ and $r>0$, define 
$$D(a,r):=\{z\in\mathbb{L}:|z-a|<r\}\ \ \text{and}\ \ \overline{D}(a,r):=\{z\in\mathbb{L}:|z-a|\le r\}.$$
If $r\in|\mathbb{L}^\times|$, we say that $D(a,r)$ is an open rational disk in $\mathbb{L}$ and $\overline{D}(a,r)$ is a closed rational disk in $\mathbb{L}$. If $r\not\in|\mathbb{L}^\times|$, then $D(a,r)=\overline{D}(a,r)$. And we call it an irrational disk. Let $U(a,r)\subset \mathbb{L}$ be a disk centered at $a\in\mathbb{L}$ with radius $r>0$, that is, $U$ has the form $D(a,r)$ or $\overline{D}(a,r)$. Then if $b\in U(a,r)$, we have $U(a,r)=U(b,r)$. Moreover, the radius $r$ is the same as the diameter of $U(a,r)$, that is $r=\sup\{|z-w|,w\in U(a,r)\}$. Furthermore, if two disks have a nonempty intersection, then one must contain the other. Finally, we should mention here every disk in $\mathbb{L}$ is both open and closed under the topology of $\mathbb{L}$.\par 
The Berkovich affine line $\mathbb{A}_{\mathrm{Ber}}^1(\mathbb{L})$ is the set of all multiplicative seminorms on the ring $\mathbb{L}[z]$ of polynomials over $\mathbb{L}$, whose restriction to the field $\mathbb{L}\subset\mathbb{L}[z]$ is equal to the given absolute value $|\cdot|$. To ease notation, we may write $\mathbb{A}_{\mathrm{Ber}}^1$ for $\mathbb{A}_{\mathrm{Ber}}^1(\mathbb{L})$. For $a\in\mathbb{L}$ and $r\ge 0$, let $\xi_{a,r}$ be the seminorm defined by 
$$|f|_{\xi_{a,r}}=\sup\limits_{z\in\overline{D}(a,r)}|f(z)|$$ 
for $f\in\mathbb{L}[z]$. Then there are $4$ types of points in $\mathbb{A}_{\mathrm{Ber}}^1$:\par
1. Type I. $\xi_{a,0}$ for some $a\in\mathbb{L}$.\par
2. Type II. $\xi_{a,r}$ for some $a\in\mathbb{L}$ and $r\in|\mathbb{L}^\times|$. \par
3. Type III. $\xi_{a,r}$ for some $a\in\mathbb{L}$ and $r\notin|\mathbb{L}^\times|$. \par 
4. Type IV. A limit of seminorms $\{\xi_{a_i,r_i}\}_{i\ge 0}$, where the corresponding sequence of closed disks $\{\overline{D}_{a_i,r_i}\}_{i\ge 0}$ satisfies $\overline{D}_{a_{i+1},r_{i+1}}\subset\overline{D}_{a_i,r_i}$ and $\bigcap\limits_i\overline{D}_{a_i,r_i}=\emptyset$.\par 
We can identify $\mathbb{L}$ with the type I points in $\mathbb{A}_{\mathrm{Ber}}^1$ via $a\to\xi_{a,0}$. The point $\xi_{0,1}\in\mathbb{A}_{\mathrm{Ber}}^1$ is called the Gauss point and denoted by $\xi_g$. We put the Gelfand topology (weak topology) on $\mathbb{A}_{\mathrm{Ber}}^1$, which makes the map $\mathbb{A}_{\mathrm{Ber}}^1\to[0,+\infty)$, sending $\xi$ to $|f|_\xi$, continuous for each $f\in\mathbb{L}[z]$. Then $\mathbb{A}_{\mathrm{Ber}}^1$ is locally compact, Hausdorff and uniquely path-connected.\par
For $\xi,\xi'\in\mathbb{A}^1_{\mathrm{Ber}}$, we say $\xi\le\xi'$ if $|f|_\xi\le|f|_{\xi'}$ for all polynomials $f\in\mathbb{L}[z]$. For types I, II and III points, $\xi_{a_1,r_1}\le\xi_{a_2,r_2}$ if and only if $\overline{D}(a_1,r_1)\subset\overline{D}(a_2,r_2)$. Then $\le$ is a partial order on $\mathbb{A}^1_{\mathrm{Ber}}$. Denote by $\xi\vee\xi'$ the least upper bound with respect to the partial order $\le$. Then $\xi\vee\xi'$ always exists and is unique. The small metric on $\mathbb{A}_{\mathrm{Ber}}^1$ is defined by 
$$d(\xi,\xi')=2\mathrm{diam}(\xi\vee\xi')-\mathrm{diam}(\xi)-\mathrm{diam}(\xi'),$$
where $\mathrm{diam}(\xi)$ is the affine diameter of the point $\xi$, that is, if $\{\xi_{a_i,r_i}\}$ is a sequence of seminorms converging to $\xi$, then $\mathrm{diam}(\xi)=\lim\limits_{i\to\infty}r_i$. The topology on $\mathbb{A}_{\mathrm{Ber}}^1$ induced by the metric $d$ is called the strong topology. It is strictly finer than the weak topology. The Berkovich hyperbolic space is defined by 
$$\mathbb{H}_{\mathrm{Ber}}:=\mathbb{A}_{\mathrm{Ber}}^1\setminus\mathbb{L}.$$ 
Then we define the path distance metric $\rho$ on $\mathbb{H}_{\mathrm{Ber}}^1$ by
$$\rho(\xi,\xi')=\log\frac{\text{diam}(\xi\vee\xi')}{\text{diam}(\xi)}+\log\frac{\text{diam}(\xi\vee\xi')}{\text{diam}(\xi')}.$$
The restriction of the strong topology to $\mathbb{H}_{\mathrm{Ber}}^1$ coincides with the metric topology induced by $\rho$. The space $\mathbb{H}_{\mathrm{Ber}}^1$ is complete for this metric, but not locally compact.\par
The Berkovich projective line $\mathbb{P}_{\mathrm{Ber}}^1$ is obtained by gluing two copies of $\mathbb{A}_{\mathrm{Ber}}^1$ along $\mathbb{A}_{\mathrm{Ber}}^1\setminus\{0\}$ via the map $\xi\to 1/\xi$. Then we can associate the Gelfand topology and strong topology on $\mathbb{P}_{\mathrm{Ber}}^1$. Under the Gelfand topology, The Berkovich projective line $\mathbb{P}_{\mathrm{Ber}}^1$ is a compact, Hausdorff, uniquely path-connected topological space and contains the projective line $\mathbb{P}^1_{\mathbb{L}}$ as a dense subset.\par
The space $\mathbb{P}_{\mathrm{Ber}}^1$ has tree structure. For a point $\xi\in\mathbb{P}_{\mathrm{Ber}}^1$, we can define an equivalence relation on $\mathbb{P}_{\mathrm{Ber}}^1\setminus\{\xi\}$, that is, $\xi'$ is equivalent to $\xi''$ if $\xi'$ and $\xi''$ are in the same connected component of $\mathbb{P}_{\mathrm{Ber}}^1\setminus\{\xi\}$. Such an equivalence class $\vec{v}$ is called a direction at $\xi$. We say that the set $T_{\xi}\mathbb{P}_{\mathrm{Ber}}^1$ formed by all directions at $\xi$ is the tangent space at $\xi$. For $\vec{v}\in T_{\xi}\mathbb{P}_{\mathrm{Ber}}^1$, denote by $\mathbf{B}_{\xi}(\vec{v})^-$ the component of $\mathbb{P}_{\mathrm{Ber}}^1\setminus\{\xi\}$ corresponding to the direction $\vec{v}$. If $\xi\in\mathbb{P}_{\mathrm{Ber}}^1$ is a type I or IV point, $T_{\xi}\mathbb{P}_{\mathrm{Ber}}^1$ consists of a single direction.  If $\xi\in\mathbb{P}_{\mathrm{Ber}}^1$ is a type III point, $T_{\xi}\mathbb{P}_{\mathrm{Ber}}^1$ consists of two directions. If $\xi\in\mathbb{P}_{\mathrm{Ber}}^1$ is a type II point, the directions in $T_{\xi}\mathbb{P}_{\mathrm{Ber}}^1$ are in one-to-one correspondence with the elements in $\mathbb{P}^1$. Since the Gauss point $\xi_g$ is a type II point, we can identify $T_{\xi_g}\mathbb{P}_{\mathrm{Ber}}^1$ to $\mathbb{P}^1$ by the correspondence $T_{\xi_g}\mathbb{P}_{\mathrm{Ber}}^1\to\mathbb{P}^1$ sending $\vec{v}_x$ to $x$, where $\vec{v}_x$ is the direction at $\xi_g$ such that $\mathbf{B}_{\xi_g}(\vec{v}_x)^-$ contains all the type I points whose images are $x$ under the canonical reduction map $\mathbb{P}^1_{\mathbb{L}}\to\mathbb{P}^1$.\par 
The spherical metric on the projective $\mathbb{P}^1_{\mathbb{L}}$ is defined as follows: for points $z=[x:y]$ and $w=[u:v]$ in $\mathbb{P}^1_{\mathbb{L}}$,
$$\Delta(z,w):=\frac{|xv-yu|}{\max\{|x|,|y|\}\max\{|u|,|v|\}}.$$
Equivalently, 
$$\Delta(z,w):=
\begin{cases}
\frac{|z-w|}{\max\{1,|z|\}\max\{1,|w|\}}, &\text{if}\  z,w\in\mathbb{L},\\
\frac{1}{\max\{1,|z|\}}, &\text{if}\  z\in\mathbb{L},w=\infty.
\end{cases}$$\par
Recall a degree $d\ge 1$ rational map $\phi:\mathbb{P}^1_{\mathbb{L}}\to\mathbb{P}^1_{\mathbb{L}}$ is represented by a pair $\phi_a,\phi_b\in\mathbb{L}[X,Y]$ of degree $d$ homogeneous polynomials with no common factors, that is, $\phi([X:Y])=[\phi_a(X,Y):\phi_b(X,Y)]$ for all $[X:Y]\in\mathbb{P}^1_{\mathbb{L}}$. Equivalently, the map $\phi$ can be considered as the quotient of two relatively prime polynomials, of which the greatest degree is $d$. Then $\phi$ induces a map from $\mathbb{P}_{\mathrm{Ber}}^1$ to itself. We use the same notation $\phi$ for the induced map. For the Gauss point $\xi_g\in\mathbb{P}_{\mathrm{Ber}}^1$, the corresponding closed disk $\overline{D}(0,1)\subset\mathbb{L}$ is a disjoint union of open disks $D(a,1)$ for $a\in\mathbb{C}$.  Note for all but finitely many such $a$, we have $\phi(D(a,1))$ is an open disk $D(\phi(a),r)$ for some $r\in |\mathbb{L}^\times|$. Then $\phi(\xi_g)$ is the point $\xi_{\phi(a),r}\in\mathbb{P}_{\mathrm{Ber}}^1$ for any such choice of $a$. For an arbitrary type II point $\xi_{b,s}$, pick $M\in\mathrm{PGL}_2(\mathbb{L})$ such that $M(\xi_g)=\xi_{b,r}$. Then apply the previous discussion to $\phi\circ M$ and $\phi(\xi_{b,s})=\phi\circ M(\xi_g)$. Since type II points are dense in $\mathbb{P}_{\mathrm{Ber}}^1$, we can get the image $\phi(\xi)$ for any $\xi\in\mathbb{P}_{\mathrm{Ber}}^1$.\par
For a holomorphic family 
$$\left\{f_t(z)=\frac{a_d(t)z^d+\cdots+a_0(t)}{b_d(t)z^d+\cdots+b_0(t)}\right\}\subset\mathbb{C}(z)$$
of degree $d$ rational maps, let $\mathbf{a_d},\cdots,\mathbf{a_0},\mathbf{b_d},\cdots,\mathbf{b_0}$ be the power series expressions of the coefficients $a_d(t),\cdots,a_0(t),b_d(t),\cdots,b_0(t)$, respectively. Then the degree $d$ rational map $\mathbf{f}:\mathbb{P}^1_{\mathbb{L}}\to\mathbb{P}^1_{\mathbb{L}}$ given by 
$$\mathbf{f}(z)=\frac{\mathbf{a_d}z^d+\cdots+\mathbf{a_0}}{\mathbf{b_d}z^d+\cdots+\mathbf{b_0}}$$
is called, following Kiwi \cite{Kiwi15}, the rational map associated to $\{f_t\}$. Hence $\{f_t\}$ induces a map $\mathbf{f}:\mathbb{P}_{\mathrm{Ber}}^1\to\mathbb{P}_{\mathrm{Ber}}^1$.\par
Let $P(z)\in\mathcal{O}_\mathbb{L}[z]$ be a polynomial and let $\mathrm{Red}(P)$ be the image of $P(z)$ in $\mathcal{O}_\mathbb{L}[z]/\mathcal{M}_\mathbb{L}[z]$. We say $\mathrm{Red}(P)$ is the reduction of $P$. For a rational map $\phi=P/Q\in\mathbb{L}(z)$, where $P,Q\in\mathbb{L}[z]$ with no common zeros. We can normalize $\mathbf{\phi}$ such that $P,Q\in\mathcal{O}_\mathbb{L}(z)$ and the maximal absolute value of coefficients of $P$ and $Q$ is $1$. For a normalized rational map $\phi=P/Q\in\mathbb{L}(z)$, the reduction $\mathrm{Red}(\phi)$ is defined by 
$$\mathrm{Red}(\phi):=
\begin{cases}
\mathrm{Red}(P)/\mathrm{Red}(Q)\ \mathrm{if}\ Q\not\in\mathcal{M}_\mathbb{L}[z],\\
\infty\ \mathrm{if}\ Q\in\mathcal{M}_\mathbb{L}[z].
\end{cases}$$
In fact, if $\phi$ is induced from a holomorphic family, the reduction $\mathrm{Red}(\phi)$ coincides with the limit of $\phi(z)$ as $t\to 0$. For convenience, for an arbitrary rational map $\phi\in\mathbb{L}(z)$, we write $\mathrm{Red}(\phi)$ for the reduction of the normalization of $\phi$.\par 
\begin{lemma}
Let $\mathbf{\phi}(z),\mathbf{\psi}(z)\in\mathbb{L}(z)$ be two rational maps. Then\par
\begin{enumerate}
\item $\mathrm{Red}(\phi\cdot\psi)=\mathrm{Red}(\phi)\cdot\mathrm{Red}(\psi)$,
\item $\mathrm{Red}(\phi+\psi)=\mathrm{Red}(\phi)+\mathrm{Red}(\psi)$,
\item If $\deg\mathrm{Red}(\psi)\ge 1$, then $\mathrm{Red}(\phi\circ\psi)=\mathrm{Red}(\phi)\circ\mathrm{Red}(\psi)$.
\end{enumerate}
\end{lemma}
The following result, originally proved by Rivera-Letelier, gives a characteristic of rational maps with nonconstant reductions. 
\begin{proposition}\label{fix-gauss}\cite[Corollary 9.27]{Baker10}
Let $\mathbf{\phi}(z)\in\mathbb{L}(z)$ be a normalized rational map. Then $\deg\mathrm{Red}(\phi)\ge 1$ if and only if $\phi(\xi_g)=\xi_g$.
\end{proposition}
At each point $\xi\in\mathbb{P}^1_{\mathrm{Ber}}$, denote by $\deg_\xi\phi$ the local degree of $\phi$ at $\xi$. Then 
\begin{lemma}\label{dir-multi}\cite[Section 9.1]{Baker10}
Let $\mathbf{\phi}(z)\in\mathbb{L}(z)$ be a nonconstant rational map. Let $\xi\in\mathbb{P}^1_{\mathrm{Ber}}$ and $\vec{v}\in T_\xi\mathbb{P}^1_{\mathrm{Ber}}$. Then there is a positive integer $m$ and a point $\xi'\in\mathbf{B}_\xi(\vec{v})^-$ such that 
\begin{enumerate}
\item $\deg_{\xi''}\phi=m$ for all $\xi''\in(\xi,\xi')$, and
\item $\rho(\phi(\xi),\phi(\xi''))=m\rho(\xi,\xi'')$ for all $\xi''\in(\xi,\xi')$.
\end{enumerate}
\end{lemma}
The integer $m$ is called the directional multiplicity and denoted by $m_\phi(\xi,\vec{v})$.\par 
Let $\xi\in\mathbb{P}^1_{\mathrm{Ber}}$. For any $\vec{v}\in T_\xi\mathbb{P}_{\mathrm{Ber}}^1$, there is a unique $\vec{w}\in T_{\phi(\xi)}\mathbb{P}_{\mathrm{Ber}}^1$ such that for any $\xi'$ sufficiently near $\xi$, $\phi(\xi')\in\mathbf{B}_{\phi(\xi)}(\vec{w})^-$. Thus the rational map $\phi$ induces a map 
$$\phi_\ast:T_{\xi}\mathbb{P}_{\mathrm{Ber}}^1\to T_{\phi(\xi)}\mathbb{P}_{\mathrm{Ber}}^1,$$ 
sending the direction $\vec{v}$ to the corresponding direction $\vec{w}$. \par 
The following lemma gives the relations between local degrees and directional multiplicities.
\begin{lemma}\label{multi-dir-multi}\cite[Theorem 9.22]{Baker10}
Let $\phi\in\mathbb{L}(z)$ be a nonconstant rational map. Let $\xi\in\mathbb{P}^1_{\mathrm{Ber}}$. Then
\begin{enumerate}
\item for each tangent $\vec{w}\in T_{\phi(\xi)}\mathbb{P}^1_{\mathrm{Ber}}$, we have 
$$\deg_\xi\phi=\sum_{\substack{\vec{v}\in T_\xi\mathbb{P}^1_{\mathrm{Ber}}\\ \phi_\ast(\vec{v})=\vec{w}}}m_\phi(\xi,\vec{v});$$
\item the induced map $\phi_\ast:T_{\xi}\mathbb{P}_{\mathrm{Ber}}^1\to T_{\phi(\xi)}\mathbb{P}_{\mathrm{Ber}}^1$ is surjective. If $\xi$ is of type I, III, or IV, then $\deg_\xi\phi=m_\phi(\xi,\vec{v})$ for each $\vec{v}\in T_{\xi}\mathbb{P}_{\mathrm{Ber}}^1$.
\end{enumerate}
\end{lemma}
The following result allows us to compute the local degree of $\phi$ at a type II point.
\begin{proposition}\label{algebraic-reduction}\cite[Proposition 3.3]{Rivera03II}
Let $\phi\in\mathbb{L}(z)$ be a nonconstant rational function, and let $\xi\in\mathbb{P}^1_{\mathrm{Ber}}$ be a type II point. Set $\xi'=\phi(\xi)$ and choose $M_1,M_2\in\mathrm{PGL}_2(\mathbb{L})$ such that $M_1(\xi_g)=\xi$ and $M_2(\xi_g)=\xi'$. Let $\psi=M_2^{-1}\circ\phi\circ M_1$. Then
$\deg\mathrm{Red}(\psi)\ge 1$ and 
$$\deg_\xi(\phi)=\deg\mathrm{Red}(\psi).$$
For each $a\in\mathbb{P}^1$, if $\vec{v}_a\in T_\xi\mathbb{P}^1_{\mathrm{Ber}}$ is the associated tangent direction under the bijection between $T_\xi\mathbb{P}^1_{\mathrm{Ber}}$ and $\mathbb{P}^1$ afforded by $(M_1^{-1})_\ast$, we have
$$m_\phi(\xi,\vec{v}_a)=\deg_a\mathrm{Red}(\psi).$$
\end{proposition}
Now let's state a fundamental fact in Berkovich dynamics.
\begin{proposition}\label{fundamental-Berk}\cite[Lemma 2.1]{Rivera03II} Let $\phi:\mathbb{P}_{\mathrm{Ber}}^1\to\mathbb{P}_{\mathrm{Ber}}^1$ be a rational map of degree at least $1$. For $\xi\in\mathbb{P}_{\mathrm{Ber}}^1$ and $\vec{v}\in T_\xi\mathbb{P}_{\mathrm{Ber}}^1$, the image $\phi(\mathbf{B}_\xi(\vec{v})^-)$ always contains $\mathbf{B}_{\phi(\xi)}(\phi_\ast(\vec{v}))^-$, and either $\phi(\mathbf{B}_\xi(\vec{v})^-)=\mathbf{B}_{\phi(\xi)}(\phi_\ast(\vec{v}))^-$ or $\phi(\mathbf{B}_\xi(\vec{v})^-)=\mathbb{P}_{\mathrm{Ber}}^1$. Moreover
\begin{enumerate}
\item if $\phi(\mathbf{B}_\xi(\vec{v})^-)=\mathbf{B}_{\phi(\xi)}(\phi_\ast(\vec{v}))^-$, then each $\xi'\in\mathbf{B}_{\phi(\xi)}(\phi_\ast(\vec{v}))^-$ has $m_\phi(\xi,\vec{v})$ preimages in $\mathbf{B}_\xi(\vec{v})^-$, counting multiplicities;\par
\item if $\phi(\mathbf{B}_\xi(\vec{v})^-)=\mathbb{P}_{\mathrm{Ber}}^1$, there is an integer $s_\phi(\xi,\vec{v})>0$ such that each $\xi'\in\mathbf{B}_{\phi(\xi)}(\phi_\ast(\vec{v}))^-$ has $m_\phi(\xi,\vec{v})+s_\phi(\xi,\vec{v})$ preimages in $\mathbf{B}_\xi(\vec{v})^-$ and each $\xi''\in\mathbb{P}_{\mathrm{Ber}}^1\setminus\mathbf{B}_{\phi(\xi)}(\phi_\ast(\vec{v}))^-$ has $s$ preimages in $\mathbf{B}_\xi(\vec{v})^-$, counting multiplicities.
\end{enumerate}
\end{proposition}
Following Faber \cite{Faber13I}, the integer $s_\phi(\xi,\vec{v})$ is called the surplus multiplicity.\par 
The following lemma gives us a criterion for which Berkovich disk $\mathbf{B}_\xi(\vec{v})^-$ is mapped to another Berkovich disk.
\begin{lemma}\label{disk-to-disk}\cite[Theorem 9.42]{Baker10}
Let $\phi\in\mathbb{L}(z)$ be a nonconstant rational map. Then $\phi(\mathbf{B}_\xi(\vec{v})^-)$ is a Berkovich disk if and only if for each $c\in\mathbf{B}_\xi(\vec{v})^-$, the local degree of $\phi$ is nonincreasing on the directed segment $[\xi,c]$.  
\end{lemma}

\subsection{Berkovich Dynamics of Newton maps}
Fix $d\ge 3$. Let 
$$r=\{r_1,\cdots,r_d\}\subset\mathbb{L}$$ 
be a set of $d$ distinct points. Set $P_r(z)=\prod_{i=1}^d(z-r_i)\in\mathbb{L}[z]$. Same as complex case, the Newton map $N_r:\mathbb{P}^1_{\mathbb{L}}\to\mathbb{P}^1_{\mathbb{L}}$ of the polynomial $P_r$ is defined by 
$$N_r(z)=z-\frac{P_r(z)}{P_r'(z)}.$$
The map $N_r$ extends to a map $N_r:\mathbb{P}^1_{\mathrm{Ber}}\to\mathbb{P}^1_{\mathrm{Ber}}$. To ease notation, we write $N$ for $N_r$. In this subsection, we state some fundamental properties of $N$ in the Berkovich space $\mathbb{P}^1_{\mathrm{Ber}}$. These properties shed light on the case when we consider holomorphic families of Newton maps. Indeed, if $\{N_{r(t)}\}$ is a holomorphic family of degree $d$ Newton maps with $r(t)=\{r_1(t),\cdots,r_d(t)\}$, by regarding each $r_i(t)$ as an element in $\mathbb{L}$, the associated map for $\{N_{r(t)}\}$ is a Newton map in $\mathbb{L}(z)$.\par 
Since $\mathbb{L}$ is a complete and algebraically closed field of characteristic zero, the map $N$ has $2d-2$ critical points, counted with multiplicity, in $\mathbb{P}^1_{\mathbb{L}}$. Thus the collections $\mathrm{Crit}(N)\cup\{\infty\}$ is a finite collection of elements in $\mathbb{P}^1_{\mathbb{L}}$, which sits in the Berkovich space $\mathbb{P}^1_{\mathrm{Ber}}$. Let $H_{\mathrm{big}}$ be convex hull of $\mathrm{Crit}(N)\cup\{\infty\}$. As a topological space, $H_{\mathrm{big}}$ is homeomorphic to the underlying space of a finite tree. We may equip this undelying space with a natural graph-theoretic combinatorial structure so that we may speak of vertices and edges. Consider the branch points as the vertices of $H_{\mathrm{big}}$. Then $H_{\mathrm{big}}$ is a finite tree and it plays a central role in our study. Since it is the biggest of several trees we will consider, we indicate this by the subscript. \par
In general, the vertices of $H_{\mathrm{big}}$ come in two flavors: (1) those in the hyperbolic space $\mathbb{H}_{\mathrm{Ber}}$, which are called the internal vertices of the tree, and (2) those in $\mathbb{P}^1_\mathbb{L}$, i.e. the critical points and the point $\infty$, which are called the leaves of the tree. Note the fixed points of $N_r$ in $\mathbb{P}^1_{\mathbb{L}}$ are $r_i$s and $\infty$. Then the convex hull $H_{\mathrm{fix}}$ of $\{r_1,\cdots,r_d,\infty\}$ sits in $H_{\mathrm{big}}$. As for $H_{\mathrm{big}}$, we can consider vertices, internal vertices and leaves for $H_{\mathrm{fix}}$. Then a vertex of $H_{\mathrm{fix}}$ is a vertex of $H_{\mathrm{big}}$. Hence $H_{\mathrm{fix}}$ is a subtree of $H_{\mathrm{big}}$. Let $V$ be the set of internal vertices of $H_{\mathrm{fix}}$. We have
$$V=\{\xi\in H_{\mathrm{fix}}:\mathrm{Val}_{H_{\mathrm{fix}}}(\xi)\ge 3\},$$
where $\mathrm{Val}_E(\xi)$ is the valance of $\xi\in E$ in a connected subset $E\subset\mathbb{P}_{\mathrm{Ber}}^1$, that is, the number of components of $E\setminus\{\xi\}$. Thus each element in $V$ is an internal vertex of $H_{\mathrm{big}}$. The convex hull $H_V$ spanned by $V$ is a subset of $\mathbb{H}_{\mathrm{Ber}}$. It also play a key role. The path distance metric on $\mathbb{H}_{\mathrm{Ber}}$ gives $H_V$ a natural metric structure that we study in section \ref{rescaling}. Moreover, as we will discuss in \cite{Nie-2}, the combinatorial tree $H_V$ encodes the stratum in the Deligne-Mumford compactification of the moduli space $\mathcal{M}_{0,d+1}$  of Riemann surfaces of genus $0$ with $d+1$ marked points to which the configuration $\{r_1,\cdots,r_d,\infty\}$ converges.\par 
To make it clear, we list the subtrees we will study.
\begin{align*}
&H_{\mathrm{big}}:=\mathrm{Hull}(\mathrm{Crit}(N)\cup\{\infty\}),\\
&H_{\mathrm{fix}}:=\mathrm{Hull}(\{r_1,\cdots,r_d,\infty\}),\\
&H_{r}:=\mathrm{Hull}(\{r_1,\cdots,r_d\}),\\
&H_{\mathrm{crit}}:=\mathrm{Hull}(\mathrm{Crit}(N)),\\
&H_{V}:=\mathrm{Hull}(V),\\
&H_{V}^\infty:=\mathrm{Hull}(V\cup\{\infty\}),
\end{align*}
where $V:=\{\xi\in H_{\mathrm{fix}}:\mathrm{Val}_{H_{\mathrm{fix}}}(\xi)\ge 3\}$. \par
To illustrate the these subtrees, we first define visible point, which was introduced by Faber \cite{Faber13II}
\begin{definition}
For a connected subset $E\subset\mathbb{P}^1_{\mathrm{Ber}}$ and a point $\xi\in\mathbb{P}^1_{\mathrm{Ber}}$, we say a point $\xi'$ is \textit{the visible point from $\xi$ to $E$} if $\xi'$ is the unique point of closure $\overline{E}$ of $E$ (in strong topology) minimize the small metric from $\xi$ to $E$. Denote by $\pi_{E}(\xi)$ the visible point from $\xi$ to $E$. 
\end{definition}
Note if $\xi\in\overline{E}$, we have $\pi_{E}(\xi)=\xi$. If $\xi\not\in\overline{E}$, then $\pi_{E}(\xi)$ is the unique point on $\overline{E}$ disconnecting $E$ from $\xi$.\par  
Now we illustrate the several subtrees in $\mathbb{P}^1_{\mathrm{Ber}}$. As an motivating example, we consider
\begin{example}
Let $r_1=0, r_2=t, r_3=2t, r_4=1, r_5=1+t$ and $r_6=2$ be six points in $\mathbb{L}$. Set $r=\{r_1,r_2,r_3,r_4,r_5,r_6\}$ and consider the map $N_r$. Then $N_r$ has four free critical points $c_1,c_2,c_3$ and $c_4$. Figure \ref{big-hull} shows the convex hull $H_{\mathrm{big}}$ for $N_r$. The hull $H_{\mathrm{big}}$ has internal vertices $v_1,v_2,v_3\in\mathbb{H}_{\mathrm{Ber}}$ and leaves $r_1,\cdots,r_6,c_1,\cdots,c_4,\infty\in\mathbb{P}_{\mathbb{L}}^1$. Let $V=\{v_1,v_2,v_3\}$ and $H_V=\mathrm{Hull}(V)$. At $v_1$, the reduction of $N_r$ has degree $3$. Hence there are $4$ critical points of $N_r$ whose visible points to $H_V$ are $v_1$. At $v_2$, the map $(N_r)_\ast$ has $1$ superattracting fixed point and $2$ attracting fixed points. Thus there are at least two directions whose corresponding Berkovich disks containing free critical points, say $c_2,c_4$, attracted to these two attracting fixed points. At $v_3$, the reduction of $N_r$ has degree $2$.
\begin{figure}[h!]
\centering
\includegraphics[width=60mm]{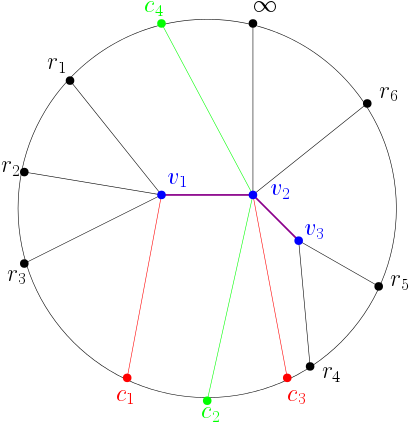}
\caption{The hull $H_{\mathrm{big}}\subset\mathbb{P}^1_{\mathrm{Ber}}$ for $N_r$. The hull $H_V$ consists of the edges $(v_1,v_2)$, $(v_2,v_3)$ and the vertices $v_1,v_2,v_3$. The points $c_1,c_2,c_3$ and $c_4$ are the free critical points of $N_r$. At $v_2$, the directions containing $c_2$ and $c_4$ are attracted to a attracting fixed points of the map $(N_r)_\ast$.}
\label{big-hull}
\end{figure}
\end{example}\par
For now, we focus on combinatorial aspects. The following theorem summarizes the properties of these subtrees. The proof depends on a sequence of lemmas established later in this subsection. 
\begin{theorem}\label{tree-properties}
For a Netwon map $N(z)\in\mathbb{L}(z)$ of degree $d\ge 3$, consider the subtrees defined as above. Then we have 
\begin{enumerate}
\item The set $V$ is a singleton if and only if the map $N$ has a potentially good reduction.
\item Let $\mathrm{Fix}_{\mathrm{Ber}}(N)$ be the set of fixed points of $N$ in $\mathbb{P}^1_{\mathrm{Ber}}$. Then 
$$\mathrm{Fix}_{\mathrm{Ber}}(N)=H_{V}^\infty\cup\{r_1,\cdots,r_d\}.$$
\item The visible points from elements in $\mathrm{Crit}(N)$ to $H_V^\infty$ lie in $V$. Moreover, for each $v\in V$, there are $2\deg_v N-2$ points, counted with multiplicity, in $\mathrm{Crit}(N)$ whose visible points to $H_{V}^\infty$ are $v$.
\item For each $v\in V$, the visible points from the preimages $N^{-1}(v)$ to $H_{\mathrm{fix}}$ lie in $V$.
\item The set $H_{\mathrm{fix}}$ is forward invariant and contains $H_{V}^\infty$. Moreover, edges of  $H_{\mathrm{fix}}$ maps to themselves. In particular, each point in $H_r\setminus H_V$ is attracted to some $r_i$.
\item The Berkovich ramification locus $\mathcal{R}_N$ of $N$ consists of the union of closed segments $[v,a]$s, where $a\in\mathrm{Crit}(N)$ and $v$ is the visible point from $a$ to $H_V$. 
\end{enumerate}
\end{theorem}
\begin{remark}
In Theorem \ref{tree-properties}, the assumption $d\ge 3$ is necessary. Indeed, if $d=2$, then $H_r=H_{\mathrm{crit}}$ and $V$ is always a singleton. Hence $H_{V}^\infty$ is a segment and the ramification locus is connected which is $H_r$. But Theorem \ref{tree-properties} $(2)-(5)$ are still true for $d=2$. 
\end{remark}
If $V=\{\xi\}$ is a singleton, then $\mathrm{Val}_{H_{\mathrm{fix}}}(\xi)=d$. Hence $\deg_{\xi}N=d$. Let $M\in\mathrm{PGL}_2(\mathbb{L})$ such that $M(\xi_g)=\xi$. Then $\deg_{\xi_g}M^{-1}\circ N\circ M=d$. By Proposition \ref{algebraic-reduction}, we have $\deg\mathrm{Red}(M^{-1}\circ N\circ M)=d$. Hence $N$ has a potentially good reduction. Conversely, if $N$ has a potentially good reduction, then there exits $M\in\mathrm{PGL}_2(\mathbb{L})$ such that $\deg\mathrm{Red}(M^{-1}\circ N\circ M)=d$. So $\deg_{\xi_g}M^{-1}\circ N\circ M=d$. Letting $\xi=M(\xi_g)$, we have $\deg_{\xi}N=d$. Hence $V=\{\xi\}$ is a singleton. This gives the proof of Theorem \ref{tree-properties} $(1)$. \par
\begin{lemma}
$\pi_{H_{r}}(\infty)=\pi_{H_V}(\infty)$.
\end{lemma}
\begin{proof}
Note 
$$H_{\mathrm{fix}}=H_r\cup(\pi_{H_r}(\infty),\infty].$$ 
Then $\mathrm{Val}_{H_\mathrm{fix}}(\pi_{H_{r}}(\infty))\ge 3$. Thus $\pi_{H_{r}}(\infty)\in V$. Hence $\pi_{H_r}(\infty)\in H_V$. Since $H_V\subset H_r$, we have 
$$d(\infty,\pi_{H_r}(\infty))=\min_{\xi'\in H_r}d(\infty,\xi')\le\min_{\xi''\in H_V}d(\infty,\xi'').$$
The conclusion follows from the uniqueness of the visible point.
\end{proof}
For a rational map $\phi\in\mathbb{L}(z)$ and a type II point $\xi\in\mathbb{P}^1_{\mathrm{Ber}}(\mathbb{L})$, Proposition \ref{fix-gauss} claims that $\phi(\xi)=\xi$ if and only if the reduction $\mathrm{Red}(\psi)$ is nonconstant, where $\psi=M^{-1}\circ\phi\circ M$ and $M\in \mathrm{PGL}_2(\mathbb{L})$ such that $M(\xi_g)=\xi$. A fixed point $\xi\in\mathrm{Fix_{Ber}}(\phi)$ is repelling if $\deg_\xi\phi\ge 2$. For the map $N$, the points $r_1,\cdots,r_d$ and $\infty$ are the only fixed points of type I. Now we show
\begin{lemma}\label{fixed-pt}
For the map $N$,
$$\mathrm{Fix_{Ber}}(N)\cap\mathbb{H}_{\mathrm{Ber}}=H_{V}^\infty\setminus\{\infty\}.$$
In particular, the set of repelling fixed points of $N$ in $\mathbb{H}_{\mathrm{Ber}}$ is $V$.
\end{lemma}
\begin{proof}
If $\xi\in H_{V}^\infty$ is a type II point, let $M\in \mathrm{PGL}_2(\mathbb{L})$ such that $M(\xi_g)=\xi$. Then the reduction of $M^{-1}\circ N\circ M$ is nonconstant. Indeed, $M^{-1}\circ N\circ M\in\mathbb{L}(z)$ is a Newton map for a polynomial in $\mathbb{L}[z]$ with multiple roots. If $\xi\in H_{V}^\infty$ is not a type II point, there exists a sequence $\{\xi_n\}\subset H_{V\infty}$ of type II points such that $d(\xi_n,\xi)\to 0$ as $n\to\infty$. Note $N$ is continuous. Then 
$$d(\xi_n,N(\xi))=d(N(\xi_n),N(\xi))\to 0.$$
Thus $\xi\in\mathrm{Fix_{Ber}}(N)$.\par 
If $\xi\not\in H_{V}^\infty$ is a type II point, let $M_1\in\mathrm{PGL}_2(\mathbb{L})$ such that $M_1(\xi_g)=\xi$. Then there exist at least $d-1$ many $r_i$'s such that $\mathrm{Red}(M_1^{-1}(r_i))=\infty$. Thus, the reduction of $M_1^{-1}\circ N\circ M_1$ is constant. So $\xi\not\in\mathrm{Fix_{Ber}}(N)$. For any type III or IV point $\xi\not\in H_{V}^\infty$, the visible point $\eta:=\pi_{H_{\mathrm{fix}}}(\xi)$ is a type II point. Let $\vec{v}\in T_\eta\mathbb{P}^1_{\mathrm{Ber}}$ be such that $\xi\in\mathbf{B}_\eta(\vec{v})^-$. By Lemma \ref{disk-to-disk}, we know $N(\mathbf{B}_\eta(\vec{v})^-)$ is a Berkovich disk. Since $\vec{v}$ is not a fixed point of the map $N_\ast$, 
$$N(\mathbf{B}_\eta(\vec{v})^-)\not=\mathbf{B}_\eta(\vec{v})^-.$$
So $\xi\not\in\mathrm{Fix_{Ber}}(N)$. Therefore, $\mathrm{Fix_{Ber}}(N)\cap\mathbb{H}_{\mathrm{Ber}}\subset H_{V}^\infty$.\par 
Now we show that the set $V$ consisting of repelling fixed points in $\mathbb{H}_\mathbb{L}$. Let $\xi\in\mathrm{Fix_{Ber}}(N)$ be a type II point. Let $M_2\in \mathrm{PGL}_2(\mathbb{L})$ such that $M_2(\xi_g)=\xi$. If $\xi\not\in V$, then $\mathrm{Red}(M_2^{-1}(r_i))$ is either $\infty$ or some constant $c\in\mathbb{C}$. Thus, the reduction $M_2^{-1}\circ N\circ M_2$ has degree $1$. Then by Proposition \ref{algebraic-reduction}, $\deg_\xi N=1$. So $\xi$ is not a repelling fixed point. If $\xi\in V$, then there exist two distinct constants $c_1$ and $c_2$ in $\mathbb{C}$, and exist $r_i$ and $r_j$ with $i\not=j$ such that $\mathrm{Red}(M_2^{-1}(r_i))=c_1$ and $\mathrm{Red}(M_2^{-1}(r_j))=c_2$. Thus, the reduction $M_2^{-1}\circ N\circ M_2$ has degree at least $2$. Again, by Proposition \ref{algebraic-reduction}, we have $\deg_\xi N\ge 2$. Thus $\xi$ is a repelling fixed point.
\end{proof}
\begin{corollary}\label{V-deg}
$V=\{\xi\in H_{V}^\infty:\deg_\xi N\ge 3\}$. 
\end{corollary}
\begin{proof}
First note at $\infty$, the local degree $\deg_\infty N=1$. By the definition of $H_V$ and Lemma \ref{fixed-pt}, we know 
$$V\subset\{\xi\in H_V:\deg_\xi N\ge 3\}=\{\xi\in H_{V}^\infty:\deg_\xi N\ge 3\}.$$\par 
For the other direction, let $\xi\in H_{V}^\infty\setminus V$. By Lemma \ref{fixed-pt}, we have $\xi\in\mathrm{Fix_{Ber}}(N)$. Suppose $\xi$ is of type II. Let $M\in\mathrm{PGL}_2(\mathbb{L})$ be such that $M(\xi_g)=\xi$. Then $M^{-1}\circ N\circ M$ has reduction of degree at most $1$. By Proposition \ref{algebraic-reduction}, we have $\deg_\xi N\le 1$. By the density of type II points and Lemma \ref{dir-multi}, we know $\{\xi\in H_{V}^\infty:\deg_\xi N\ge 3\}\subset V$.
\end{proof}
Recall that a point $c\in\mathbb{L}$ is a free critical point of $N$ if $c\in\mathrm{Crit}(N)\setminus\{r_1,\cdots,r_d\}$. The following lemma gives the visible points of critical points to the tree $H_{V}^\infty$ and hence implies Theorem \ref{tree-properties} $(3)$.
\begin{lemma}
If $c\in\mathrm{Crit}(N)$ is a free critical point, then $\pi_{H_{V}^\infty}(c)\in V$. Moreover, for $v\in V$,
$$\#\{a\in\mathbb{L}: P''_r(a)=0,\pi_{H_{V}^\infty}(a)=v\}=2\deg_v N-2-\#\{r_i:\pi_{H_{V}^\infty}(r_i)=v\},$$
counted with multiplicity.
\end{lemma}
\begin{proof}
To ease the notations, let $\xi=\pi_{H_{V}^\infty}(c)$. Then $\xi\in\mathbb{P}^1_{\mathrm{Ber}}$ is a type II point. Let $\vec{v}\in T_\xi\mathbb{P}^1_{\mathrm{Ber}}$ be the tangent vector such that $c\in\mathbf{B}_\xi(\vec{v})^-$. Then the directional multiplicity $m_{N}(\xi,\vec{v})\ge 2$ since $c$ is a critical point of $N$. Thus by Lemma \ref{multi-dir-multi}, we have  
$$\deg_\xi(N)\ge m_{N}(\xi,\vec{v})\ge 2.$$
By Corollary \ref{V-deg}, we have $\xi\in V$.\par
For $v\in V$, by the definition of $V$ and Lemma \ref{fixed-pt}, we know $v$ is a type II fixed point of $N$. Let $M\in\mathrm{PGL}_2(\mathbb{L})$ such that $M(\xi_g)=v$. Then the reduction $\mathrm{Red}(M^{-1}\circ N\circ M)$ has $2\deg_v N-2$ critical points. Note 
$$\mathrm{Crit}(\mathrm{Red}(M^{-1}\circ N\circ M))=\{\mathrm{Red}(M^{-1}(a)):a\in\mathrm{Crit}(N),\pi_{H_{V}^\infty}(a)=v\}$$ 
and 
$$\mathrm{Crit}(N)=\{a\in\mathbb{L}: P''_r(a)=0\}\cup\{r_1,\cdots,r_d\}.$$
Thus we obtain the conclusion.
\end{proof}
Now we prove a weak version of Theorem \ref{tree-properties}$(4)$ which allows us to characteristic the fixed Berkovich Fatou components of $N$ and hence prove Theorem \ref{tree-properties}$(5)$. After that, we prove Theorem \ref{tree-properties}$(4)$.
\begin{lemma}\label{preimage-H0}
For $v\in V$, let $\xi\in N^{-1}(v)$. Then  
$$\pi_{H_V}(\xi)\in V.$$
Moreover, for $v_1\not=v_2\in V$,
$$\#\{\xi\in\mathbb{P}^1_{\mathrm{Ber}}: \xi\in N^{-1}(v_1),\pi_{H_V}(\xi)=v_2\}=\deg_{v_2}N-1.$$
\end{lemma}
\begin{proof}
Suppose $\eta:=\pi_{H_V}(\xi)\in H_V\setminus V$. By Lemma \ref{fixed-pt} and Corollary \ref{V-deg}, we have $\eta\in\mathrm{Fix_{Ber}}(N)$ and $\deg_\eta N=1$. Let $\vec{v}\in T_\eta\mathbb{P}^1_{\mathrm{Ber}}$ be the tangent vector such that $\xi\in\mathbf{B}_\eta(\vec{v})^-$. Since $N_\ast(\vec{v})\not=\vec{w}$, where $\vec{w}\in T_\eta\mathbb{P}^1_{\mathrm{Ber}}$ such that $\mathbf{B}_\eta(\vec{w})^-\cap V\not=\emptyset$, then 
$$N(\mathbf{B}_\eta(\vec{v})^-)\cap V=\emptyset.$$
Hence $N(\xi)\not\in V$. It is a contradiction. Thus $\eta\in V$.\par
At the point $v_2$, let $\vec{v}_1\in T_{v_2}\mathbb{P}^1_{\mathrm{Ber}}$ such that $v_1\in\mathbf{B}_{v_2}(\vec{v}_1)^-$. Then $\vec{v}_1$ is a fixed point of $N_\ast$ with $\deg_{\vec{v}_1}N_\ast=1$. Thus, we have 
$$\#\{\xi\in\mathbb{P}^1_{\mathrm{Ber}}:\xi\in N^{-1}(v_1),\pi_{H_V}(\xi)=v_2\}=\deg N_\ast-1=\deg_{v_2}N-1.$$
\end{proof}
\begin{corollary}\label{not-whole-Ber}
For $v\in V$, if $\xi\in N^{-1}(v)$ such that $\pi_{H_V}(\xi)=v$, then $\xi=v$. In particular, for any $\vec{v}\in T_{v}\mathbb{P}^1_{\mathrm{Ber}}$ with $\mathbf{B}_v(\vec{v})^-\cap V=\emptyset$,
$$N(\mathbf{B}_v(\vec{v})^-)\not=\mathbb{P}^1_{\mathrm{Ber}}.$$
\end{corollary}
\begin{proof}
Let $V=\{v_1,\cdots,v_n\}$. Note 
$$\deg_{v_i}N=\mathrm{Val}_{H_{\mathrm{fix}}}(v_i)-1.$$
Thus 
\begin{align*}
\deg_{v_i}N+\sum_{j\not=i}(\deg_{\xi_j}N-1)&=\sum_{j=1}^n\deg_{\xi_j}N-(n-1)=\sum_{j=1}^n(\mathrm{Val}_{H_{\mathrm{fix}}}(v_j)-1)-(n-1)\\
&=\sum_{j=1}^n\mathrm{Val}_{H_{\mathrm{fix}}}(v_j)-(2n-1)=d+1+2(n-1)-(2n-1)\\
&=d
\end{align*}
Note $v_i$ has $d$ preimages. Thus, if $\pi_{H_V}(\xi_i)=v_i$ for some $\xi_i\in N^{-1}(v_i)$, then $\xi_i=v_i$.\par 
As a consequence, for $v\in V$ and $\vec{v}\in T_v\mathbb{P}^1_{\mathrm{Ber}}$ with $\mathbf{B}_v(\vec{v})^-\cap V=\emptyset$, we have $v\not\in N(\mathbf{B}_\xi(\vec{v})^-)$. Hence $N(\mathbf{B}_v(\vec{v})^-)\not=\mathbb{P}^1_{\mathrm{Ber}}$.
\end{proof}
Recall that for a rational map $\phi\in\mathbb{L}(z)$, the Berkovich Julia set $J_{\mathrm{Ber}}(\phi)$ consists of the points $\xi\in\mathbb{P}^1_{\mathrm{Ber}}$ such that for all (weak) neighborhood $U$ of $\xi$, the set $\cup_{n=1}^\infty\phi^n(U)$ omits at most two points in $\mathbb{P}^1_{\mathrm{Ber}}$. The complement of the Berkovich Julia set is the Berkovich Fatou set $F_{\mathrm{Ber}}(\phi)$, see \cite[Section 10.5]{Baker10} and \cite[Section 7]{Benedetto10}. If $U$ is a periodic connected component of $F_{\mathrm{Ber}}(\phi)$, then $U$ is either a Rivera domain, mapping bijectively onto itself, or an attracting component, mapping multiple to one onto itself \cite[Theorem 10.76]{Baker10}. A periodic type II point $\xi\in \mathbb{P}^1_{\mathrm{Ber}}$ of period $q$ is repelling if $\deg_\xi\phi^q\ge 2$. Repelling periodic points are contained in the Julia set $J_{\mathrm{Ber}}(\phi)$ \cite[Theorem 2.1]{Kiwi14}. Moreover, all the repelling periodic points in $\mathbb{H}_\mathbb{L}$ are of type II \cite[Lemma 10.80]{Baker10}. We say a subset $A$ of $\mathbb{P}^1_{\mathrm{Ber}}$ is an annulus if $A$ is an intersection of two disks $B,B'$ with distinct boundary points such that $B\cup B'=\mathbb{P}^1_{\mathrm{Ber}}$. For a Newton map $N\in\mathbb{L}(z)$, by Lemma \ref{fixed-pt} and Corollary \ref{V-deg}, we know $V\subset J_{\mathrm{Ber}}(N)$. Moreover, each component of $\mathbb{P}^1_{\mathrm{Ber}}\setminus V$ is either an annulus or a ball. Indeed, if $U$ is a component of $\mathbb{P}^1_{\mathrm{Ber}}\setminus V$ with at least $3$ boundary points, then there exists $v\in V$ such that $v\in U$. It is a contradiction. In contrast to complex dynamics, the point $\infty$ is an indifferent fixed point of $N$. Hence, $\infty\in F_{\mathrm{Ber}}(N)$. 
\begin{lemma}\label{Fatou}
Let $U$ be a component of $F_{\mathrm{Ber}}(N)$ fixed by $N$. 
\begin{enumerate}
\item $U$ is a fixed Rivera domain if and only if $U$ is a component of $\mathbb{P}^1_{\mathrm{Ber}}\setminus V$ which is either an annulus or a disk containing $\infty$.
\item $U$ is a fixed attracting domain if and only if $U$ is a component of $\mathbb{P}^1_{\mathrm{Ber}}\setminus V$ which contains some $r_i$.
\end{enumerate}
\end{lemma}
\begin{proof}
First note if $U$ is a component of $\mathbb{P}^1_{\mathrm{Ber}}\setminus V$, then $U$ is fixed by $N$ if and only if either the boundary of $U$ has more than one point, $U$ contains some $r_i$ or $U$ contains $\infty$. Indeed, if $U$ is a component of $\mathbb{P}^1_{\mathrm{Ber}}\setminus V$ which has only one boundary point, say $\xi$, and does not contain either $r_i$ or $\infty$. Then $\xi\in V$. Let $\vec{v}\in T_\xi\mathbb{P}^1_{\mathrm{Ber}}$ such that $U=\mathbf{B}_{\xi}(\vec{v})^-$. Then $N_\ast$ does not fix $\vec{v}$. Thus $U$ is not fixed by $N$. By Corollary \ref{not-whole-Ber}, we know the component $U$ containing some $r_i$ is fixed, since $N$ fixes the boundary point of $U$ and $r_i$. Similarly, the component $U$ containing $\infty$ is a fixed component. If $U$ is a component with at least $2$ boundary points, then by Lemma \ref{fixed-pt} and noting $\deg N_\ast=1$ at any point on $H_V\setminus V$, we have $U$ is a fixed component.\par 
Since each $r_i$ is a (super)attracting fixed point of $N$ and $N$ has no other attracting fixed points, we get the classifications of fixed Rivera domains and attracting domains for $N$.
\end{proof}
Note each $r_i$ is a superattracting fixed point and $v_i:=\pi_{H_V}(r_i)\in V$ is a repelling fixed point of $N$. The following lemma states that the segment $[v_i,r_i]$ is invariant under $N$ and in the path distance metric $\rho$, the map $N$ pushes points in $(v_i,r_i)$ away from $v_i$ to $r_i$. It deduces Theorem \ref{tree-properties}$(5)$ immediately since by Lemma \ref{fixed-pt}, we have $H_{V}^\infty$ is fixed by $N$ and $H_{\mathrm{fix}}\setminus H_{V}^\infty=H_r\setminus H_V$. Later, we will show $N$ expands uniformly on $[v_i,r_i)$.
\begin{lemma}\label{fix-segment}
Any component $L$ of $H_r\setminus H_V$ is invariant by $N$. Moreover, for any $\xi\in L\cap\mathbb{H}_{\mathrm{Ber}}$, then 
$$\rho(N(\xi),\pi_{H_V}(\xi))\ge\rho(\xi,\pi_{H_V}(\xi))$$
\end{lemma}
\begin{proof}
Suppose there exist a component $L$ of $H_r\setminus H_V$ and a point $\xi\in L$ such that $N(\xi)\not\in L$. We assume $r_i\in L$. Set $v=\pi_{H_V}(\xi)$ . Let $\vec{v}_1\in T_v\mathbb{P}^1_{\mathrm{Ber}}$ and $\vec{v}_2\in T_\xi\mathbb{P}^1_{\mathrm{Ber}}$ be such that $r_i\in\mathbf{B}_{\xi}(\vec{v}_1)^-\subset\mathbf{B}_{v}(\vec{v}_2)^-$. Let $\xi'=N(\xi)$ and $\vec{w}=N_\ast(\vec{v})\in T_{\xi'}\mathbb{P}^1_{\mathrm{Ber}}$. By Corollary \ref{not-whole-Ber}, 
$$N(\mathbf{B}_{\xi}(\vec{v}_1)^-)=\mathbf{B}_{\xi'}(\vec{w})^-.$$
Note $N(r_i)=r_i$. Thus $[r_i,v]\subset\mathbf{B}_{\xi'}(\vec{w})^-$. Thus $N^{-1}(v)\cap\mathbf{B}_{v}(\vec{v}_2)^-\not=\emptyset$. It contradicts with Corollary \ref{not-whole-Ber}. Thus $N(L)\subset L$. Note $N$ fixes the endpoints of $L$. Thus $N(L)=L$.\par
By Lemma \ref{Fatou}, the disk $\mathbf{B}_{v}(\vec{v}_2)^-$ is an attracting Fatou component. Thus 
$$\rho(\xi,r_i)>\rho(N(\xi),r_i).$$
Note $\pi_{H_V}(N(\xi))=\pi_{H_V}(\xi)$. Thus we have 
$$\rho(\xi,\pi_{H_V}(\xi))<\rho(N(\xi),\pi_{H_V}(N(\xi))=d(N(\xi),\pi_{H_V}(\xi)).$$
\end{proof}
Now we can prove Theorem \ref{tree-properties} $(4)$.
\begin{corollary}\label{preimage-H-infty}
For $v\in V$, let $\xi\in N^{-1}(v)$. Then $\pi_{H_{\mathrm{fix}}}(\xi)\in V$.
\end{corollary}
\begin{proof}
By Lemma \ref{preimage-H0}, it is sufficient to show $\pi_{H_{\mathrm{fix}}}(\xi)\not\in H_{{\mathrm{fix}}}\setminus H_V$. Suppose that $\eta:=\pi_{H_{\mathrm{fix}}}(\xi)\in H_{{\mathrm{fix}}}\setminus H_V$. By Corollary \ref{not-whole-Ber}, the Berkovich disk with boundary $\eta$ containing $\xi$ maps to a Berkovich disk $\mathbf{B}$ with boundary $N(\eta)$ and $\pi_{H_V}(\eta)\not\in\mathbf{B}$. By Lemma \ref{fix-segment}, we know 
$$\rho(N(\eta),\pi_{H_V}(\eta))>\rho(\eta,\pi_{H_V}(\eta)).$$
Thus $\mathbf{B}\cap V=\emptyset$. It is a contradiction.
\end{proof}
For a rational map $\phi\in\mathbb{L}(z)$, the Berkovich ramification locus $\mathcal{R}_\phi$ is defined by 
$$\mathcal{R}_\phi=\{\xi\in\mathbb{P}^1_{\mathrm{Ber}}: \deg_\xi\phi\ge 2\}.$$
The Berkovich ramification locus $\mathcal{R}_\phi$ is a closed subset of $\mathbb{P}^1_{\mathrm{Ber}}$ with no isolated point and has at most $d-1$ components \cite[Theorem A]{Faber13I}. Moreover, $\mathcal{R}_\phi$ is contained in the convex hull of critical points of $\phi$ \cite[Corollary 7.13]{Faber13I}. Then for Newton map $N$, we have $\mathcal{R}_N\subset H_{\mathrm{crit}}$. The following lemma gives the precise locations of the Berkovich ramification locus $\mathcal{R}_N$ and deducts Theorem \ref{tree-properties} $(6)$.
\begin{lemma}\label{Newton-ramification}
For the map $N$, the Berkovich ramification locus is
$$\mathcal{R}_N=(H_\mathrm{crit}\setminus H_V)\cup V.$$
\end{lemma}
\begin{proof}
For any point $\xi\in H_V\setminus V$, by Corollary \ref{V-deg}, $\deg_\xi N<2$. Thus $\xi\not\in\mathcal{R}_N$. Hence, $\mathcal{R}_N\subset(H_{\mathrm{crit}}\setminus H_V)\cup V$.\par 
Conversely, by Corollary \ref{V-deg}, $V\subset\mathcal{R}_N$. Now pick $\xi\in H_{\mathrm{crit}}\setminus H_V$. By Corollary \ref{not-whole-Ber}, there exists a tangent vector $\vec{v}\in T_\xi\mathbb{P}^1_{\mathrm{Ber}}$ such that $\mathbf{B}_{\xi}(\vec{v})^-\cap\mathrm{Crit}(N)\not=\emptyset$ and $N(\mathbf{B}_{\xi}(\vec{v})^-)\not=\mathbb{P}^1_{\mathrm{Ber}}$. Thus, the directional multiplicity $m_N(\xi,\vec{v})\ge 2$. By Lemma \ref{multi-dir-multi}, 
$$\deg_\xi N\ge m_N(\xi,\vec{v})\ge 2.$$
Thus, $\xi\in\mathcal{R}_N$.
\end{proof}
To end this section, we show the map $N$ expands uniformly on each segment $[v_i, r_i)$, where $v_i=\pi_{H_V}(r_i)$.
\begin{lemma}
Let $v_i=\pi_{H_V}(r_i)$. Then for any $\xi_1,\xi_2\in [v_i, r_i)$, then 
$$\rho(N(\xi_1),N(\xi_2))\ge 2\rho(\xi_1,\xi_2).$$
\end{lemma}
\begin{proof}
By Lemma \ref{Newton-ramification}, for each point $\xi\in [v_i, r_i]$, we have $\deg_\xi N\ge 2$. In fact, the directional multiplicity $m_N(\xi,\vec{v})\ge 2$, where $\vec{v}\in T_\xi\mathbb{P}^1_{\mathrm{Ber}}$ is such that $r_i\in\mathbf{B}_{\xi}(\vec{v})^-$. Then the conclusion follows from Lemma \ref{dir-multi}.
\end{proof}

\section{Measure Theory}\label{measure}
\subsection{Limiting Measures on $\mathbb{P}^1$}
In this subsection, we associate each point $f\in\mathbb{P}^{2d+1}$ a natural measure $\mu_f$.\par  
For $f\in\mathrm{Rat}_d$, let $\mu_f$ be the unique measure of maximal entropy, which is given by the weak limit 
$$\mu_f=\lim\limits_{n\to\infty}\frac{1}{d^n}\sum_{f^n(z)=a}\delta_z$$
for any nonexceptional point $a\in\mathbb{P}^1$,see \cite{Freire83, Ljubich83, Mane83}. The measure $\mu_f$ has no atoms, and $\text{supp}\ \mu_f=J(f)$. Moreover, $\mu_{f^n}=\mu_f$.\par 
For $f=H_f\hat f$ with $\text{deg}\ \hat f\ge 1$, following DeMarco \cite{DeMarco05}, define 
$$\mu_f=\sum_{n=0}^\infty\frac{1}{d^{n+1}}\sum_{H_f(h)=0}\sum_{\hat f^n(z)=h}\delta_z,$$
where the holes $h$ and all preimages by $\hat f$ are counted with multiplicity. Then $\mu_f$ is an atomic probability measure. If $\text{deg}\hat f=0$, define 
$$\mu_f=\frac{1}{d}\sum_{H_f(h)=0}\delta_h,$$
where the holes $h$ are counted with multiplicity. Recall the indeterminacy locus $I(d)\subset\mathbb{P}^{2d+1}$ is defined by 
$$I(d)=\{f=H_f\hat f\in\mathbb{P}^{2d+1}: \hat f\equiv c\ \text{and}\ H_f(c)=0\}.$$ 
Then if $f\not\in I(d)$, we have $\mu_f=\mu_{f^n}$. We refer \cite{DeMarco05, DeMarco07} for more properties of the measure $\mu_f$ when $f$ is degenerate.\par
\begin{example}
Consider the degenerate cubic Newton map $N_{\{0,0,1\}}\in\mathbb{P}^7$. To ease the notation, we write $N=N_{\{0,0,1\}}$. Since $z=0$ is the unique hole of $N$, then by definition, we know 
$$\mu_N=\sum_{n=0}^\infty\frac{1}{3^{n+1}}\sum_{\hat N^n(z)=0}\delta_z.$$
In particular, $\mu_N(\{0\})=\frac{1}{2}$.
\end{example}
Recall that $d_h(f)$ is the depth of hole $h\in\mathbb{P}^1$ for the map $f$.
\begin{proposition}\label{measure-converge}\cite[Theorem 0.1, Theorem 0.2]{DeMarco05}
Suppose $d\ge2$ and $f\in\mathbb{P}^{2d+1}$. Then $f\in I(d)$ if and only if the map $g\to\mu_g$ is discontinuous at $f$. Moreover, Suppose $\{f_k\}$ is a sequence in $\mathrm{Rat}_d$ converging to $f=H_f\hat f\not\in\mathrm{Rat}_d$.
\begin{enumerate}
\item If $f\not\in I(d)$, $\mu_{f_k}\to\mu_f$ weakly.
\item If $f\in I(d)$, let $\hat f\equiv c$. Then any subsequential limit $\mu$ of $\mu_{f_k}$ satisfies 
$$\mu(\{c\})\ge\frac{d_c(f)}{d+d_c(f)}.$$ 
\end{enumerate}
\end{proposition}
If $f=H_f\hat f=H_fc\in I(d)$ and $d_c(f)=1$, under the assumptions in Proposition \ref{measure-converge}, the lower bound $d_c(f)/(d+d_c(f))$ is sharp \cite[Example 5.1]{DeMarco05}. Furthermore, DeMarco and Faber proved every weak limit of measures of maximal entropy for degenerating sequences is purely atomic \cite[Theorem A]{DeMarco14}.

\subsection{Limiting Measures for Holomorphic Families}
In this subsection, we provide some basic facts about the measurable dynamics on Berkovich space $\mathbb{P}^1_{\mathrm{Ber}}$ and state DeMarco and Faber's results about the limiting measures for holomorphic families, see \cite{DeMarco14, DeMarco16}Œ, which connect measurable dynamics on the Riemann sphere with measurable dynamics on Berkovich space. For more measurable dynamics on $\mathbb{P}^1_{\mathrm{Ber}}$, we refer \cite{Baker10}.\par
Let $\phi:\mathbb{P}^1_{\mathrm{Ber}}\to\mathbb{P}^1_{\mathrm{Ber}}$ be a rational map of degree $d\ge 2$. Let $\mu_\phi$ be the equilibrium measure on $\mathbb{P}^1_{\mathrm{Ber}}$ relative to $\phi$. The measure $\mu_\phi$ is also known as the canonical measure, see \cite{Baker10}, which is the unique Borel probability measure $\mu$ that satisfies $\phi^\ast\mu=d\mu$ and that does not charge any type I point in $\mathbb{P}^1_{\mathrm{Ber}}$
\cite[Theorem A]{Favre10}. In fact, for any nonexceptional point $\xi\in\mathbb{P}^1_{\mathrm{Ber}}$, the measure $\mu_\phi$ is the weak limit 
$$\mu_\phi=\lim\limits_{n\to\infty}\frac{1}{d^n}\sum_{\phi^n(\eta)=\xi}\delta_\eta,$$
where the sum is counted with multiplicity \cite[Theorem A]{Favre10}. \par 
Following DeMarco and Faber\cite{DeMarco16}, we say $\Gamma\subset\mathbb{P}^1_{\mathrm{Ber}}$ is a vertex set if $\Gamma$ is a finite nonempty set of type II points. The connected components of $\mathbb{P}^1_{\mathrm{Ber}}\setminus\Gamma$ is called the $\Gamma$-domains. Denote by $\mathcal{P}(\Gamma)$ be the partition of $\mathbb{P}^1_{\mathrm{Ber}}$ consisting of the elements of $\Gamma$ and all the corresponding $\Gamma$-domains. Then the equilibrium $\Gamma$-measure $\omega_{\phi,\Gamma}$ is defined by $\omega_{\phi,\Gamma}(U):=\mu_\phi(U)$ for each $U\in\mathcal{P}(\Gamma)$. Then the measure $\omega_{\phi,\Gamma}$ supports on a countable subset of $\mathcal{P}(\Gamma)$ and has total mass $1$. Indeed, the support of $\mu_\phi$ is the Berkovich Julia set $J_{\mathrm{Ber}}(\phi)$ \cite[Section 10.5]{Baker10} and there are countably many elements in $\mathcal{P}(\Gamma)$ intersect $J_{\mathrm{Ber}}(\phi)$ \cite[Lemma 2.4, Proposition 2.5]{DeMarco16}. \par 
Since the tangent space $T_{\xi_g}\mathbb{P}^1_{\mathrm{Ber}}$ can be canonically identified with $\mathbb{P}^1$, in the case $\Gamma=\{\xi_g\}$, the equilibrium $\Gamma$-measure $\omega_{\phi,\Gamma}$ gives us a natural Borel measure on $\mathbb{P}^1$. Indeed, the branches of Berkovich space attached to the Gauss point correspond to points in its tangent space, which is identified with $\mathbb{P}^1$ via reduction, so a measure on branches induces a measure on $\mathbb{P}^1$. This measure on $\mathbb{P}^1$ is called the residual equilibrium measure. \par
For a holomorphic family $\{f_t\}$,  the maximal measures $\mu_{f_t}$ converge weakly to the residual equilibrium measure for the induced map $\mathbf{f}$ on $\mathbb{P}^1_{\mathrm{Ber}}$.
\begin{proposition}\label{holo-limit-measure}\cite[Theorem B]{DeMarco14}
Suppose that $d\ge 2$. Let $\{f_t\}$ be a degenerate holomorphic family in $\mathrm{Rat}_d$. Then $\mu_{f_t}$ converges weakly to a limiting probability measure $\mu$ as $t\to 0$. Moreover, the measure $\mu$ is equal to the residual equilibrium measure for the induced rational map $\mathbf{f}:\mathbb{P}^1_{\mathrm{Ber}}\to\mathbb{P}^1_{\mathrm{Ber}}$.
\end{proposition}
To compute the limit measure $\mu$ in Proposition \ref{holo-limit-measure}, DeMarco and Faber studied the  pairs $(\phi, \Gamma)$ consisting of a degree $d\ge 2$ rational map $\phi:\mathbb{P}^1_{\mathrm{Ber}}\to\mathbb{P}^1_{\mathrm{Ber}}$ and a vertex set $\Gamma$ in $\mathbb{P}^1_{\mathrm{Ber}}$. A $\Gamma$-domain $U$ is called a $F$-domain if $\phi^n(U)\cap\Gamma=\emptyset$ for all $n\ge 1$. Otherwise, we say $U$ is a $J$-domain. Denote by $\mathcal{J}(\Gamma)\subset\mathcal{P}(\Gamma)$ the subset consisting of all $J$-domains and elements of $\Gamma$. For a given vertex set $\Gamma$, the set $\mathcal{J}(\Gamma)$ is countable \cite[Lemma 2.4]{DeMarco16} and the union of sets in $\mathcal{J}(\Gamma)$ contains the Julia set $J_{\mathrm{Ber}}(\mathbf{f})$ \cite[Proposition 2.5]{DeMarco16}. We say the pair $(\phi,\Gamma)$ is analytically stable if for each $\xi\in\Gamma$, either $\phi(\xi)\in\Gamma$ or $\phi(\xi)\in U$ for some $F$-domain $U$. If $(\phi,\Gamma)$ is analytically stable, then $\phi$ maps a $F$-domain into another $F$-domain \cite[Lemma 2.7]{DeMarco16}. Moreover,
for $U,V\in\mathcal{J}(\Gamma)$ and $y\in U$, the quantity
$$m_{U,V}:=\#(\phi^{-1}(y)\cap V).$$
is independent of the choice of $y$ \cite[Lemma 2.8]{DeMarco16}.
\begin{proposition}\label{equilibrium measure}\cite[Theorem C]{DeMarco16}
Let $\phi:\mathbb{P}^1_{\mathrm{Ber}}\to\mathbb{P}^1_{\mathrm{Ber}}$ be a degree $d\ge 2$ rational map and let $\Gamma\subset\mathbb{P}^1_{\mathrm{Ber}}$ be a vertex set. Let $\mu_\phi$ be the equilibrium measure for $\phi$. Suppose that $(\phi,\Gamma)$ is analytically stable and $J_{\mathrm{Ber}}(\phi)$ is not contained in $\Gamma$. Let $P$ be the $|\mathcal{J}(\Gamma)|\times|\mathcal{J}(\Gamma)|$ matrix defined by the $(U,V)$-entry 
$$P_{U,V}=\frac{m_{U,V}}{d}.$$ Then $P$ is the transition matrix for a countable state Markov chain with a unique stationary probability vector $\nu:\mathcal{J}(\Gamma)\to [0,1]$. The rows of $P^n$ converge pointwise to $\nu$ and the $U$-entry of $\nu$ satisfies $\nu(U)=\mu_\phi(U)$ for each $U\in\mathcal{J}(\Gamma)$.
\end{proposition}
For a degenerate holomorphic family $\{f_t\}$ in $\mathrm{Rat}_d$, consider the associated map $\mathbf{f}:\mathbb{P}_{\mathrm{Ber}}^1\to\mathbb{P}_{\mathrm{Ber}}^1$. If the pair $(\mathbf{f},\{\xi_g\})$ is analytically stable, then each $J$-domain is a disk with boundary $\xi_g$. By Proposition \ref{equilibrium measure}, we can compute the measure for each $J$-domain. Hence, we can get the direction $\vec{v}\in T_{\xi_g}\mathbb{P}^1_{\mathrm{Ber}}$ with $\mu_{\mathbf{f}}(\mathbf{B}_{\xi_g}(\vec{v})^-)>0$. Then by Proposition \ref{holo-limit-measure}, we can get the weak limit $\mu$ of measures $\mu_{f_t}$. However, in general, the pair $(\mathbf{f},\{\xi_g\})$ is not analytically stable. In fact, $(\mathbf{f},\{\xi_g\})$ is analytically stable if and only if $f_0\not\in I(d)$ \cite[Proposition 5.1]{DeMarco16}. DeMarco and Faber proved for any vertex set $\Gamma\subset\mathbb{P}^1_{\mathrm{Ber}}$, there exists a vertex set $\Gamma'$ containing $\Gamma$ such that $(\mathbf{f},\Gamma')$ is analytically stable \cite[Theorem D]{DeMarco16}.

\subsection{Limiting Measures for Newton Maps}
By Proposition \ref{Newton-closure-indeterminacy-singleton}, there is a unique point in the intersection of the closure of Newton maps and the indeterminacy locus $I(d)$. In this subsection, we let $\{N_{r(t)}\}$ be a holomorphic family of degree $d\ge 2$ Newton maps such that $N_{r(t)}$ converges to the point in $I(d)$, as $t\to 0$. Write $r(t)=\{r_1(t),\cdots, r_d(t)\}$. Then by Lemma \ref{Newton-indeterminacy}, we have as $t\to 0$, $r_i(t)\to\infty$ for $i=1,\cdots,d$. To ease notations, we write $N_t$ for the Newton map $N_{r(t)}$ and write $\mu_t$ for the maximal measure for $N_t$. We study the weak limit of measures $\mu_t$ and prove Theorem \ref{theorem-measure}.\par 
Let $\mathbf{N}:\mathbb{P}^1_{\mathrm{Ber}}\to\mathbb{P}^1_{\mathrm{Ber}}$ be the associated map for $\{N_t\}$. Then $\mathbf{N}(\xi_g)\in\mathbf{B}_{\xi_g}(\vec{v}_\infty)^-$, where $\vec{v}_\infty\in T_{\xi_g}\mathbb{P}^1_{\mathrm{Ber}}$ such that $\infty\in\mathbf{B}_{\xi_g}(\vec{v}_\infty)^-$. Indeed, it follows immediately from the following fact.
\begin{lemma}\label{Gauss-image}
Let $\{f_t\}\subset\mathrm{Rat}_d$ be a holomorphic family such that, in projective coordinates $f_t([X:Y])$ converges to $H(X,Y)c$, where $c\in\mathbb{P}^1$ and $H(X,Y)$ is a homogeneous polynomial. Let $\mathbf{f}:\mathbb{P}^1_{\mathrm{Ber}}\to\mathbb{P}^1_{\mathrm{Ber}}$ be the associated map for $\{f_t\}$. Then $\mathbf{f}(\xi_g)\in\mathbf{B}_{\xi_g}(\vec{v}_c)^-$, where $\vec{v}_c\in T_{\xi_g}\mathbb{P}^1_{\mathrm{Ber}}$ such that $c\in\mathbf{B}_{\xi_g}(\vec{v}_c)^-$.
\end{lemma}
Now, combining with the Berkovich dynamics for Newton maps in section \ref{Berkovich}, we can figure out the weak limit of measures $\mu_t$. The discussion contains several cases according to the orbit of $\xi_g$ under $\mathbf{N}$.\par
Recall $V\subset\mathbb{P}^1_{\mathrm{Ber}}$ is the set of type II repelling fixed points of $\mathbf{N}$ and recall $H_{V}^\infty$ is the convex hull of $V$ and $\{\infty\}$.
\begin{proposition}\label{Fatou-measure-1}
Let $\vec{v}_\infty\in T_{\xi_g}\mathbb{P}^1_{\mathrm{Ber}}$ such that $\infty\in\mathbf{B}_{\xi_g}(\vec{v}_\infty)^-$. Suppose that $\mathbb{P}_{\mathrm{Ber}}^1\setminus\{\mathbf{B}_{\xi_g}(\vec{v}_\infty)^-\cup\{\xi_g\}\}$ is contained in the Berkovich Fatou set $F_{\mathrm{Ber}}(\mathbf{N})$. Then the measures $\mu_t$ converge to $\delta_\infty$.
\end{proposition}
\begin{proof}
Since $\mathbb{P}_{\mathrm{Ber}}^1\setminus\{\mathbf{B}_{\xi_g}(\vec{v}_\infty)^-\cup\{\xi_g\}\}\subset F_{\mathrm{Ber}}(\mathbf{N})$, then the equilibrium measure $\mu_{\mathbf{N}}$ does not charge $\mathbb{P}_{\mathrm{Ber}}^1\setminus\{\mathbf{B}_{\xi_g}(\vec{v}_\infty)^-\cup\{\xi_g\}\}$. Since the Berkovich Julia set $J_{\mathrm{Ber}}(\mathbf{N})$ is not an isolated point, then the measure $\mu_{\mathbf{N}}$ does not charge $\mathbb{P}_{\mathrm{Ber}}^1\setminus\mathbf{B}_{\xi_g}(\vec{v}_\infty)^-$. Thus $\mu_{\mathbf{N}}(\mathbf{B}_{\xi_g}(\vec{v}_\infty)^-)=1$. By Proposition \ref{holo-limit-measure}, the measures $\mu_t$ converge to $\delta_\infty$.
\end{proof}
\begin{corollary}\label{Fatou-measure-1-corollary}
Suppose $\mathbf{N}^k(\xi_g)\not\in H_{V}^\infty$ for all $k\ge 0$ and assume there exists $k_0\ge 0$ such that the visible point $\pi_{H_V}(\mathbf{N}^{k_0}(\xi_g))\not\in V$. Then the measures $\mu_t$ converge to $\delta_\infty$.
\end{corollary}
\begin{proof}
To ease notations, set $\xi_k=\mathbf{N}^k(\xi_g)$. Let $\vec{v}^k_\infty\in T_{\xi_k}\mathbb{P}^1_{\mathrm{Ber}}$ such that $\infty\in\mathbf{B}_{\xi_k}(\vec{v}^k_\infty)^-$. Since $\xi_k\not\in H_{V}^\infty$, then $H_{V}^\infty\subset\mathbf{B}_{\xi_k}(\vec{v}^k_\infty)^-$. Without loss of generality, we can assume $k_0$ is the smallest nonnegative integer such that $\pi_{H_V}(\mathbf{N}^{k_0}(\xi_g))\not\in V$. Then by Lemma \ref{Fatou}, the set $\mathbb{P}^1_{\mathrm{Ber}}\setminus\mathbf{B}_{\xi_{k_0}}(\vec{v}^{k_0}_\infty)^-$ is contained in a fixed Revera domain of the Berkovich Fatou set $F_{\mathrm{Ber}}(\mathbf{N})$. We claim for any $\vec{v}\in T_{\xi_g}\mathbb{P}^1_{\mathrm{Ber}}$ with $\vec{v}\not=\vec{v}^0_\infty$, the set $\mathbf{N}^{k_0}(\mathbf{B}_{\xi_g}(\vec{v})^-)$ is a disk contained in $\mathbb{P}^1_{\mathrm{Ber}}\setminus\mathbf{B}_{\xi_{k_0}}(\vec{v}^{k_0}_\infty)^-$. If $k_0=0$, the claim holds trivially. If $k_0\ge 1$, by Corollary \ref{not-whole-Ber}, we know $\mathbf{N}(\mathbf{B}_{\xi_g}(\vec{v})^-)$ is a Berkovich disk $\mathbf{B}_{\xi_1}(\vec{w})^-$ for some $\vec{w}\in T_{\xi_1}\mathbb{P}^1_{\mathrm{Ber}}$, and there is no preimage of $\pi_{H_V}(\xi_g)\in V$ in the Berkovich disk $\mathbf{B}_{\xi_g}(\vec{v})^-$. Then we have $\vec{w}\not=\vec{v}^1_\infty$. Applying the same argument to $\xi_i$ for $1\le i\le k_0-1$, we have $\mathbf{N}^{k_0}(\mathbf{B}_{\xi_g}(\vec{v})^-)\subset\mathbb{P}^1_{\mathrm{Ber}}\setminus\mathbf{B}_{\xi_{k_0}}(\vec{v}^{k_0}_\infty)^-$. Thus, $\mathbf{B}_{\xi_g}(\vec{v})^-\subset F_{\mathrm{Ber}}(\mathbf{N})$ for any $\vec{v}\in T_{\xi_g}\mathbb{P}^1_{\mathrm{Ber}}$ with $\vec{v}\not=\vec{v}^0_\infty$. By Proposition \ref{Fatou-measure-1}, the measures $\mu_t$ converge to $\delta_\infty$.
\end{proof}
Denote by $H_{V}^{\xi_g}$ the convex hull of $V\cup\{\xi_g\}$.\par
\begin{theorem}\label{limit-measure-case-1}
Suppose that $\mathbf{N}^n(\xi_g)\not\in H_{V}^{\xi_g}$ for any $n\ge 1$. Then the measures $\mu_t$ converge to $\delta_\infty$. 
\end{theorem}
\begin{proof}
Let $\eta:=\pi_{H_V}(\infty)$ and set $\xi_k=\mathbf{N}^k(\xi_g)$ for $k\ge 0$. First we further assume $\xi_k\not\in(\eta,\infty)$ for all $k\ge 0$. Then $H_V$, $\infty$ and $\xi_g$ lie in the same component of $\mathbb{P}_{\mathrm{Ber}}^1\setminus\{\xi_k\}$ for any $k\ge 0$. Denote these directions by $\vec{v}^k_\infty\in T_{\xi_k}\mathbb{P}^1_{\mathrm{Ber}}$. By Corollary \ref{Fatou-measure-1-corollary}, we may assume $\pi_{H_V}(\xi_k)\in V$ for all $k\ge 0$. Set $\Gamma=\{\xi_g\}$. We claim for any $\vec{v}\in T_{\xi_g}\mathbb{P}^1_{\mathrm{Ber}}$ with $\vec{v}\not=\vec{v}^0_\infty$, the disk $\mathbf{B}_{\xi_g}(\vec{v})^-$ is a $F$-domain. By Corollary \ref{not-whole-Ber}, we know $\mathbf{N}(\mathbf{B}_{\xi_g}(\vec{v})^-)$ is a Berkovich disk $\mathbf{B}_{\xi_1}(\vec{w})^-$ for some $\vec{w}\in T_{\xi_1}\mathbb{P}^1_{\mathrm{Ber}}$. Again, by Corollary \ref{not-whole-Ber}, there is no preimage of $\pi_{H_V}(\xi_g)\in V$ in the Berkovich disk $\mathbf{B}_{\xi_g}(\vec{v})^-$. By the assumptions, we have $\vec{w}\not=\vec{v}^1_\infty$. Hence $\mathbf{B}_{\xi_1}(\vec{w})^-\cap\Gamma=\emptyset$. Then apply the argument to $\mathbf{B}_{\xi_1}(\vec{w})^-$. Inductively, we can show $\mathbf{N}^n(\mathbf{B}_{\xi_g}(\vec{v})^-)\cap\Gamma=\emptyset$ for all $n\ge 1$. Thus $\mathbf{B}_{\xi_g}(\vec{v})^-$ is a $F$-domain. Hence the equilibrium $\Gamma$-measure $\omega_{\mathbf{N},\Gamma}$ does not charge $\mathbf{B}_{\xi_g}(\vec{v})^-$. Therefore, 
$$\omega_{\mathbf{N},\Gamma}(\mathbf{B}_{\xi_g}(\vec{v}_\infty)^-)=1.$$
By Proposition \ref{holo-limit-measure}, the weak limit of $\mu_t$ is $\delta_\infty$.\par
Now we assume there exists $n_0\ge 1$ such that $\xi_{n_0}\in(\eta,\infty)$. Without loss of generality, let $n_0$ be the smallest such positive integer. Set 
$$\Gamma=\{\xi_g,\xi_1,\cdots,\xi_{n_0}\}\cup V.$$ 
By Lemma \ref{fixed-pt}, the point $\xi_{n_0}$ is a fixed point for $\mathbf{N}$ and each point in $V$ is also a fixed point for $\mathbf{N}$. Hence $\mathbf{N}(\Gamma)\subset\Gamma$. So the pair $(\mathbf{N},\Gamma)$ is analytically stable. For any $\vec{v}\in T_{\xi_g}\mathbb{P}^1_{\mathrm{Ber}}$ with $\vec{v}\not=\vec{v}_\infty$. We claim $\mathbf{B}_{\xi_g}(\vec{v})^-$ is a $F$-domain. By Corollary \ref{not-whole-Ber}, we know $\mathbf{N}^k(\mathbf{B}_{\xi_g}(\vec{v})^-)$ does not contains $\pi_{H_V}(\mathbf{N}^{k-1}(\xi_g))$ for $1\le k\le n_0$. Thus 
$$\mathbf{N}^{n_0}(\mathbf{B}_{\xi_g}(\vec{v})^-)\not=\mathbf{B}_{\xi_{n_0}}(\vec{w})^-,$$
where $\vec{w}\in T_{\xi_{n_0}}\mathbb{P}^1_{\mathrm{Ber}}$ such that $\mathbf{B}_{\xi_{n_0}}(\vec{w})^-\cap V\not=\emptyset$. Hence 
$$\mathbf{N}^{n_0}(\mathbf{B}_{\xi_g}(\vec{v})^-)\cap\Gamma=\emptyset.$$
Since $\deg\mathbf{N}_\ast=1$ at $\xi_{n_0}$ and $\vec{w}$ is totally invariant under $\mathbf{N}_\ast$, then for any $\ell\ge 1$, we have 
$$\mathbf{N}^{\ell}(\mathbf{B}_{\xi_g}(\vec{v})^-)\cap\Gamma=\emptyset.$$
Hence $\mathbf{B}_{\xi_g}(\vec{v})^-$ is a $F$-domain. So the equilibrium $\Gamma$ does not charge $\mathbf{B}_{\xi_g}(\vec{v})^-$. Therefore, by Proposition \ref{holo-limit-measure}, we know the weak limit of $\mu_t$ is $\delta_\infty$.\par 
\end{proof}
\begin{corollary}\label{Gauss-periodic-limit-measure}
Suppose the Gauss point $\xi_g$ is a periodic point for the map $\mathbf{N}$. Then the measures $\mu_t$ converge to $\delta_\infty$.
\end{corollary}
\begin{proof}
By Theorem \ref{limit-measure-case-1}, it is sufficient to check  $\mathbf{N}^n(\xi_g)\not\in H_{V}^{\xi_g}$ for any $n\ge 1$. Suppose there exits $n_0\ge 1$ such that $\mathbf{N}^{n_0}(\xi_g)\in H_{V}^{\xi_g}$. Without loss of generality, we assume $n_0$ is the smallest such integer. Let $\xi=\pi_{H_V}(\xi_g)$ be the visible point from $\xi_g$ to $H_V$. Then 
$$H_{V}^{\xi_g}=H_V\cup(\xi_g,\xi).$$
By Lemma \ref{fixed-pt}, the convex hull $H_V$ is fixed by $\mathbf{N}$ pointwisely. Thus $\mathbf{N}^{n_0}(\xi_g)\in(\xi_g,\xi)$. Otherwise, $\xi_g$ is not a periodic point. Let $\vec{v}\in T_{\xi}\mathbb{P}^1_{\mathrm{Ber}}$ such that $\xi_g\in\mathbf{B}_\xi(\vec{v})^-$.  
If $\xi\not\in V$, then at $\xi$, $\deg\mathbf{N}_\ast=1$, Note $\vec{v}$ is not a fixed point of $\mathbf{N}_\ast$. Thus $\xi_g$ can not be periodic. Hence $\xi\in V$. Thus $\mathbf{N}^{k}(\xi_g)\not=\xi_g$ if $n_0\nmid k$. If $n_0\mid k$, we have $\mathbf{N}^k(\xi_g)\notin\mathbf{B}_{\xi_{n_0}}(\vec{v})^-$, where $\vec{v}\in T_{\xi_{n_0}}\mathbb{P}^1_{\mathrm{Ber}}$ such that $\xi_g\in\mathbf{B}_{\xi_{n_0}}(\vec{v})^-$. Hence $\xi_g$ can not be periodic. It is a contradiction. Thus $\mathbf{N}^n(\xi_g)\not\in H_{V}^{\xi_g}$ for any $n\ge 1$.
\end{proof}
The following example show the assumption $\mathbf{N}^n(\xi_g)\not\in H_{V}^{\xi_g}$ for all $n\ge 1$ in Theorem \ref{limit-measure-case-1} is necessary.
\begin{example}
Let $r(t)=\{1/t,-1/t,1/t^3\}$. Consider the cubic Newton map $N_{r_t}(z)$. Note the symmetric functions are 
\begin{align*}
&\sigma_1(t)=\frac{1}{t^3},\\
&\sigma_2(t)=-\frac{1}{t^2},\\
&\sigma_3(t)=-\frac{1}{t^5}.
\end{align*}
We can write 
$$N_t(z):=N_{r(t)}(z)=\frac{2z^3-\sigma_1(t)z^2+\sigma_3(t)}{3z^2-2\sigma_1(t)z+\sigma_2(t)}.$$
Then the associated map $\mathbf{N}$ maps Gauss point $\xi_g$ to $\xi_{0,|t^{-2}|}$. Note $V=\{\xi_{0,|t^{-1}|},\xi_{0,|t^{-3}|}\}$. Hence 
$$\mathbf{N}(\xi_g)=\xi_{0,|t^{-2}|}\in H_V\subset H_{V}^{\xi_g}.$$
Now we compute the weak limit $\mu$ of the maximal measures $\mu_{N_t}$ by Propositions \ref{holo-limit-measure} and \ref{equilibrium measure}. Set $\Gamma=\{\xi_g,\xi_{0,|t^{-2}|}\}$. As $\mathbf{N}(\xi_{0,|t^{-2}|})=\xi_{0,|t^{-2}|}$, the pair $(\mathbf{N},\Gamma)$ is analytically stable. Let $V_1$ be the component of $\mathbb{P}_{\mathrm{Ber}}^1\setminus\Gamma$ with boundary $\xi_{0,|t^{-2}|}$ and containing $\infty$. Let $V_2$ be the component of $\mathbb{P}_{\mathrm{Ber}}^1\setminus\Gamma$ with boundary $\xi_g$ and $\xi_{0,|t^{-2}|}$. Let $V_3$ be the component of $\mathbb{P}_{\mathrm{Ber}}^1\setminus\Gamma$ with boundary $\xi_g$ and containing $0$. Then the set of states in this case is 
$$\mathcal{J}=\{\xi_g,\xi_{0,|t^{-2}|}, V_1,V_2,V_3\},$$
and the transition matrix $P$ is given by 
\[
\begin{blockarray}{cccccc}
\ &\xi_g & \xi_{0,|t^{-2}|} & V_1 & V_2 & V_3 \\
\begin{block}{c(ccccc)}
  \xi_g & 0 & 0 & 1/3 & 2/3 & 0 \\
  \xi_{0,|t^{-2}|} & 1/3 & 1/3 & 1/3 & 0 & 0 \\
  V_1 & 0 & 0 & 2/3 & 0 & 1/3 \\
  V_2 & 0 & 0 & 1/3 & 2/3 & 0 \\
  V_3 & 0 & 0 & 1/3 & 2/3 & 0 \\
\end{block}
\end{blockarray}
 \]
Then the unique stationary probability vector for $P$ is $(0,0,1/2,1/3,1/6)$. Thus $\mu_{\mathbf{N}}(V_1)=1/2$, $\mu_{\mathbf{N}}(V_2)=1/3$ and $\mu_{\mathbf{N}}(V_3)=1/6$, where $\mu_{\mathbf{N}}$ is the equilibrium measure for $\mathbf{N}$. Thus, the weak limit $\mu$ of the measures $\mu_{N_t}$ is 
$$\mu=\frac{5}{6}\delta_\infty+\frac{1}{6}\delta_0.$$
\end{example}
To prove Theorem \ref{theorem-measure}, based on Theorem \ref{limit-measure-case-1}, now we consider the case when there exists $n_0\ge 1$ such that $\mathbf{N}^{n_0}(\xi_g)\in H_{V}^{\xi_g}$.
\begin{theorem}\label{limit-measure-case-II}
Suppose there exists $n\ge 1$ such that $\mathbf{N}^{n}(\xi_g)\in H_{V}^{\xi_g}$. Let $\mu$ be the weak limit of measures $\mu_t$. Then 
$$\mu(\{\infty\})\ge\frac{d}{2d-1}.$$
\end{theorem}
\begin{proof}
According to the orbit $\mathcal{O}(\xi_g)$ of the Gauss point $\xi_g$, we construct different vertices set $\Gamma$ such that the pair $(\mathbf{N},\Gamma)$ is analytically stable. And then we apply Propositions  \ref{holo-limit-measure} and \ref{equilibrium measure}.\par 
To ease notations, in the following proof we let $\xi_k:=\mathbf{N}^k(\xi_g)$ for $k\ge 0$, and let $\vec{v}_\infty\in T_{\xi_g}\mathbb{P}^1_{\mathrm{Ber}}$ such that $\infty\in\mathbf{B}_{\xi_g}(\vec{v}_\infty)^-$. Note 
$$H_{V\xi_g}=H_V\cup(\xi_g,\pi_{H_V}(\xi_g)).$$ 
Hence, we have two cases. \par
Case I: There exists $n_0\ge 1$ such that $\xi_{n_0}\in H_V$. Without loss of generality, let $n_0$ be the smallest such positive integer. Set 
$$\Gamma=\{\xi_g,\xi_1,\cdots,\xi_{n_0}\}\cup V.$$
Then the pair $(\mathbf{N},\Gamma)$ is analytically stable. In this case there may exist $\vec{v}\in T_{\xi_g}\mathbb{P}^1_{\mathrm{Ber}}$ such that $\mathbf{B}_{\xi_g}(\vec{v})^-$ is a $J$-domain. Indeed, if $\xi_{n_0}\in H_V\setminus V$, there exists at most one such $\vec{v}$; if $\xi_{n_0}\in V$, there may exist most than one but finitely many such $\vec{v}$. If there is no such $\vec{v}$, then the weak limit of $\mu_t$ is $\delta_\infty$. Now we suppose that there exist $\vec{v}_1,\cdots,\vec{v}_m\in T_{\xi_g}\mathbb{P}^1_{\mathrm{Ber}}$ such that each $U_j:=\mathbf{B}_{\xi_g}(\vec{v}_j)^-$ is a $J$-domain for $1\le j\le m$. For now, we assume $m=1$. It needs no more effort to prove the case when $m>1$. Let $W_1=\mathbf{N}(U_1)$. Then $W_1$ is a Berkovich disk with boundary $\xi_1$. In fact, $W_1$ contains a union of elements in $\mathcal{J}(\Gamma)$. We may assume $W_1$ is in $\mathcal{J}(\Gamma)$. Otherwise, we consider the sum of equilibrium $\Gamma$-measure of the corresponding elements in $\mathcal{J}(\Gamma)$. Let $P$ be the transition matrix for the $J$-domains and let $\nu$ be the unique stationary probability vector. Then we have the $(W_1,U_1)$-entry $P_{W_1,U_1}$ in $P$ is nonzero and all other entries in $U_1$-column are zeros. Thus 
$$P_{W_1,U_1}\nu(W_1)=\nu(U_1),$$ 
where $\nu(W_1)$ and $\nu(U_1)$ are the $W_1$-th and $U_1$-th entries, respectively. Note $\nu(W_1)+\nu(U_1)\le 1$. Thus 
$$\frac{1+P_{W_1,U_1}}{P_{W_1,U_1}}\nu(U_1)\le 1.$$
Since the local degree $\deg_{\pi_{H_V}(\xi_g)}\mathbf{N}\le d-1$, then  
$$P_{W_1,U_1}\le\frac{d-1}{d}.$$
So we have 
$$\nu(U_1)\le\frac{P_{W_1,U_1}}{1+P_{W_1,U_1}}\le\frac{d-1}{2d-1}.$$
By Proposition \ref{equilibrium measure}, the equilibrium $\Gamma$-measure charges $\nu(U_1)$ at $U_1$ and charges $1-\nu(U_1)$ at $\mathbf{B}_{\xi_g}(\vec{v}_\infty)^-$. By Proposition \ref{holo-limit-measure}, we have
$$\mu(\{\infty\})=1-\nu(U_1)\ge\frac{d}{2d-1}.$$\par
Case II: Suppose that the Cases I does not hold and there exists $n_0\ge 1$ such that $\xi_{n_0}\in(\xi_g,\pi_{H_V}(\xi_g))$. Without loss of generality, let $n_0$ be the smallest such positive integer. Now consider the segments $(\xi_i,\pi_{H_V}(\xi_i))$. If there exists $i$ such that $\pi_{H_V}(\xi_i)\not=\pi_{H_V}(\xi_{i+1})$, then there exits $\eta_i^{(k)}\in(\xi_i,\pi_{H_V}(\xi_i))$ such that $\mathbf{N}^{(k-1)n_0+1}(\eta_i^{(k)})= \pi_{H_V}(\xi_{i+1})$. Let $\eta_i=\lim\limits_{k\to\infty}\eta_i^{(k)}$. Then $\eta_i\in(\xi_i,\pi_{H_V}(\xi_i))$. Let $\mathcal{O}(\eta_i)$ be the orbit of $\eta_i$. Note $\mathcal{O}(\eta_i)$ has $n_0$ elements. Set 
$$\Gamma=\{\xi_g,\xi_1,\cdots,\xi_{n_0-1}\}\cup V\cup\mathcal{O}(\eta_i).$$
Let $\eta=\mathcal{O}(\eta_i)\cap(\xi_g,\pi_{H_V}(\xi_g))$. Then $\xi_{n_0}\in(\xi_g,\eta)$ an hence $\xi_{n}\in(\xi_q,\mathbf{N}^q(\eta))$, where $n\equiv q\ \mathrm{mod}\ n_0$. Note $(\xi_i,\mathbf{N}^i(\eta))$ is in a $F$-domain for $i\ge 0$. Hence the pair $(\mathbf{N},\Gamma)$ is analytically stable. Note the equilibrium measure $\mu_{\mathbf{N}}$ on $\mathbb{P}^1_{\mathrm{Ber}}$ does not charge the Gauss point $\xi_g$. Then by the same argument in Case I, we can get 
$$\mu(\{\infty\})\ge\frac{d}{2d-1}.$$
\end{proof}
\begin{remark}
Since $N_t$ converges subalgebraically to the point $N$ in $I(d)$, by Lemma \ref{Newton-indeterminacy}, we know $N$ has a unique hole at $[1:0]\in\mathbb{P}^1$ with depth $d$. Then by Proposition \ref{measure-converge} $(2)$,  we have $\mu(\{\infty\})\ge 1/2$. Theorems \ref{limit-measure-case-1} and \ref{limit-measure-case-II} give a better lower bound for $\mu(\{\infty\})$.
\end{remark}
For the quadratic case, we have the following corollary.
\begin{corollary}\label{quadratic-Newton-measure}
If $d=2$, then $\mu=\delta_\infty$.
\end{corollary}
\begin{proof}
By Theorems \ref{limit-measure-case-1} and \ref{limit-measure-case-II}, we only need to show the case that there exists $n_0$ such that $\mathbf{N}^{n_0}(\xi_g)\in[v,\xi_g)$, where $v\in V$ is the unique point. We claim $\mathbf{N}^{n_0}(\xi_g)\not\in[v,\xi_g)$ for all $n\ge 1$. Since $\deg_v\mathbf{N}=2$, we have $\mathbf{N}^{n}(\xi_g)\not=v$ for all $n\ge 1$. If $\mathbf{N}^{\ell}(\xi_g)\in(v,\xi_g)$ for some $\ell\ge 1$, then for all $n\ge 1$, the segments $(v,\mathbf{N}^{n}(\xi_g))$ are disjoint with the ramification locus $\mathcal{R}_{\mathbf{N}}$. Hence $\rho(v,\xi_g)=\rho(v,\mathbf{N}^{\ell}(\xi_g))$, Thus $\mathbf{N}^{\ell}(\xi_g)=\xi_g$. It is a contradiction. So $\mu=\delta_\infty$. 
\end{proof}

\section{Rescaling Limits}\label{rescaling}
\subsection{Definitions and Known Results}
In this section, following Kiwi \cite{Kiwi15}, we give the definitions and known results about rescalings and rescaling limits.\par 
Recall that a moving frame is a holomorphic family of degree $1$ rational maps.
\begin{definition}\label{rescaling-definition}
Let $\{f_t\}\subset\mathrm{Rat}_d$ be a holomophic family. A moving frame $\{M_t\}\subset\mathrm{PGL}_2(\mathbb{C})$ is called a \textit{rescaling} for $\{f_t\}$ of period $q$ if there exist $g\in\mathrm{Rat}_e$ with $e\ge 2$ and a finite subset $S\subset\mathbb{P}^1$ such that as $t\to 0$,
$$M_t^{-1}\circ f_t^q\circ M_t\xrightarrow{\bullet} g\ \ \text{on}\ \ \mathbb{P}^1\setminus S.$$
 We say $g$ is a \textit{rescaling limit} on $\mathbb{P}^1\setminus S$. The minimal $q\ge 1$ such the above holds is called the period of the rescaling $\{M_t\}$. 
\end{definition}
Recall that the associated map $\mathbf{f}$ for a holomorphic family  $\{f_t\}$ of degree $d\ge 1$ rational maps is a rational map in $\mathbb{L}(z)$, which can act on Berkovich space $\mathbb{P}^1_{\mathrm{Ber}}$. Let $\mathbf{f}$ and $\mathbf{M}$ be the associated maps for $\{f_t\}$ and $\{M_t\}$ in Definition \ref{rescaling-definition}. By Lemmas \ref{subalg-imply-locally-uniform} and \ref{locally-uniform-imply-subalg}, the convergence holds in Definition \ref{rescaling-definition} if and only if there exits $q\ge 1$ such that the reduction $\mathrm{Red}(\mathbf{M}^{-1}\circ\mathbf{f}^q\circ\mathbf{M})$ has degree at least $2$. In general, let $M\in\mathrm{PGL}_2(\mathbb{L})$. Then $M$ induces naturally a family $\{M_t\}$ of $\mathrm{PGL}_2(\mathbb{C})$. We say $\{M_t\}$ is a generalized rescaling for the holomorphic family $\{f_t\}$ if there exists $q\ge 1$ such that the reduction $\mathrm{Red}(M^{-1}\circ\mathbf{f}^q\circ M)$ has degree at least $2$. For example, $\{M_t(z)=t^{-1/2}z\}$ is a generalized rescaling for the holomorphic family $\{f_t(z)=tz^3\}$, but it is not a rescaling for $\{f_t\}$. \par
Naturally we are interested in the sequence in the moduli space $\mathrm{rat}_d$. So we define equivalent relations on the rescalings, and count the number of the rescalings, up to these equivalence relations. 
\begin{definition}\label{rescaling-equivalent}
Let $\{M_t\}\subset\mathrm{PGL}_2(\mathbb{C})$ and $\{L_t\}\subset \mathrm{PGL}_2(\mathbb{C})$. We say $\{M_t\}$ and $\{L_t\}$ are \textit{equivalent} if there exists $M\in\mathrm{PGL}_2(\mathbb{C})$ such that $M_t^{-1}\circ L_t\to M$.
\end{definition}
The ``equivalent" in Definition \ref{rescaling-equivalent} defines an equivalence relation. Let $[\{M_t\}]$ be the equivalence class of $\{M_t\}$.\par
\begin{lemma}\cite[Lemma 3.6]{Kiwi15} 
Suppose $\{M_t\}$ and $\{L_t\}$ are two moving frames. Let  $\mathbf{M}$ and $\mathbf{L}$ be the associated maps for $\{M_t\}$ and $\{L_t\}$, respectively. Then $\{M_t\}$ and $\{L_t\}$ are equivalent if and only if $\mathbf{M}(\xi_g)=\mathbf{L}(\xi_g)$.
\end{lemma}
Note any rational map in $\mathbb{L}(z)$ maps the Gauss point $\xi_g$ to a type II point. For any type II point $\xi\in\mathbb{P}^1_{\mathrm{Ber}}$, there exists an affine map $A(z)\in\mathbb{L}[z]$ such that $A(\xi_g)=\xi$. If $\xi$ is the image of $\xi_g$ under the associated map $\mathbf{M}$ for some move frame $\{M_t\}$, we can choose the associated map  $\mathbf{A}$ for an affine move frame $\{A_t\}$ such that $\mathbf{A}(\xi_g)=\xi$.
\begin{lemma}\label{affine-equ-mobius}
For any moving frame $\{M_t(z)\}$, there exists an affine moving frame $\{A_t(z)\}$ such that $\{A_t\}$ is equivalent to $\{M_t\}$.
\end{lemma}
\begin{proof}
First note that for any type II point $\xi_{a,r}$ in $\mathbb{P}^1_{\mathrm{Ber}}$, there exists an affine map $A(z)\in\mathbb{L}[z]$ such that $A(\xi_g)=\xi_{a,r}$. Indeed, since $\xi_{a,r}$ is a type II point, then there exits $b\in\mathbb{L}$ such that $|b-a|_\mathbb{L}=r$. Then let $A(z)=bz+a$. Now let $\mathbf{M}$ be the associated map for $\{M_t\}$ and let $\xi_{a,r}=\mathbf{M}(\xi_g)$. If $\mathbf{M}\in\mathbb{L}(z)$ has no pole in $D(0,1)\subset\mathbb{L}$, then $\mathbf{M}(D(0,1))$ is a disk in $\mathbb{L}$. We can choose $a=M(0)$. Thus $a\in\mathbb{C}((t))$. Note $\mathbf{M}(1)\in\mathbb{C}((t))$. Let $n\in\mathbb{Z}$ such that $|t^n|=|\mathbf{M}(1)-\mathbf{M}(0)|$. Then $r=|t^n|$. Thus we can set $A_t(z)=t^nz+a(t)$. If $\mathbf{M}\in\mathbb{L}(z)$ has a pole in $D(0,1)\subset\mathbb{L}$, we can write $\mathbf{M}=\mathbf{M}_1\circ\mathbf{M}_2\circ\mathbf{M}_3$, where $\mathbf{M}_1$ and $\mathbf{M}_3$ are affine map and $\mathbf{M}_2(z)=1/z$. It is sufficient to show that there exist $a_n=a_n(t)\in\mathbb{C}((t))$ and $r=|t^m|$, where $m=m(n)\in\mathbb{Z}$, such that $\xi_{a_n,r_n}:=\mathbf{M}_2(\xi_{0,|t^n|})$ for any $n\in\mathbb{Z}$. It follows immediately from $\mathbf{M}_2(\xi_{0,|t^n|})=\xi_{0,|t^{-n}|}$.
\end{proof}
With Lemma \ref{affine-equ-mobius} we can count the number of equivalent classes of rescalings by only considering the affine representatives.\par 
The following result allows us to consider the action of a holomorphic $\{f_t\}$ on the set $\{[\{M_t\}]\}$ of equivalence classes of moving frames.
\begin{lemma}\cite[Lemma 3.7]{Kiwi15}\label{moving-frame-action}
Let $\{f_t\}$ be a holomorphic family of degree $d\ge 1$ rational maps. If $\{M_t\}$ is a moving frame, then there exists a moving frame $\{L_t\}$ such that for the associated maps $\mathbf{f},\mathbf{M}$ and $\mathbf{L}$,
$$\mathbf{f}\circ\mathbf{M}(\xi_g)=\mathbf{L}(\xi_g).$$
\end{lemma}
In fact, the moving frame $\{L_t\}$ in Lemma \ref{moving-frame-action} is unique up to equivalence. Thus it gives us an action 
$$\{f_t\}([\{M_t\}])=[\{L_t\}].$$
The dynamical dependence of moving frames is defined by considering the orbits of the equivalence classes of moving frames under the action of $\{f_t\}$.
\begin{definition}\label{dynamically-independent}
Let $\{f_t\}\subset\mathrm{Rat}_d$ be a holomorphic family of degree $d\ge 1$ rational maps. Let $\{M_t\}$ and $\{L_t\}$ be moving frames. We say $\{M_t\}$ and $\{L_t\}$ are \textit{dynamically dependent} if there is $\ell\ge 0$ such that 
$$\{f^\ell_t\}([\{M_t\}])=[\{L_t\}].$$
Otherwise, we say $\{M_t\}$ and $\{L_t\}$ are  dynamically independent
\end{definition}\par
The dynamical dependence also gives us an equivalent relation. From the definitions, we know if two moving frames are equivalent, then they are dynamically dependent. But the converse is not true in general.\par 
\begin{remark}\label{rescaling-limit-relation}
Let $\{M_t\}$ and $\{L_t\}$ be two rescalings for a holomorphic family $\{f_t\}$ of degree $d\ge 2$ rational maps. 
\begin{enumerate}
\item If $\{M_t\}$ and $\{L_t\}$ are equivalent, then the corresponding rescaling limits are conjugate.
\item If $\{M_t\}$ and $\{L_t\}$ are dynamically dependent, then there exist rational maps $g_1$ and $g_2$ such that the corresponding rescaling limits are $g_1\circ g_2$ and $g_2\circ g_1$, respectively \cite[Lemma 3.10]{Kiwi15}.
\end{enumerate} 
\end{remark}
To count the number of dynamically independent rescalings for a holomorphic family $\{f_t\}$, Kiwi \cite{Kiwi15} related to the dynamically independent rescalings to the type II periodic points of the associated map $\mathbf{f}$ on $\mathbb{P}^1_{\mathrm{Ber}}$.
\begin{proposition}\cite[Proposition 3.4]{Kiwi15}\label{rescaling-repelling-cycles}
Let $\{f_t\}$ be a holomorphic family of degree $d\ge 2$ rational maps and $\{M_t\}$ be a moving frame. Let $\mathbf{f}$ and $\mathbf{M}$ be the associated rational maps for $\{f_t\}$ and $\{M_t\}$, respectively. Then for all $\ell\ge 1$, the following are equivalent.
\begin{enumerate}
\item There exists a rational map $g:\mathbb{P}^1\to\mathbb{P}^1$ of degree $e\ge 1$ such that, as $t\to 0$, 
$$M_t^{-1}\circ f_t^\ell\circ M_t\xrightarrow{\bullet} g(z)\ \ \text{on}\ \ \mathbb{P}^1\setminus S,$$
 where $S\subset\mathbb{P}^1$ is a finite set.
\item $\mathbf{f}^\ell(\xi)=\xi$, where $\xi=\mathbf{M}(\xi_g)$, and $\deg_{\xi}\mathbf{f}^\ell=e$.
\end{enumerate}
In the case in which $(1)$ and $(2)$ hold, $T_\xi\mathbf{f}^\ell:T_{\xi}\mathbb{P}^1_{\mathrm{Ber}}\to T_{\xi}\mathbb{P}^1_{\mathrm{Ber}}$ is conjugate via a $\mathbb{P}^1$-isomorphism to $g:\mathbb{P}^1\to\mathbb{P}^1$.
\end{proposition}
By counting the repelling cycles of type II points, Kiwi \cite{Kiwi15} gave the number of dynamically independent rescalings that lead to nonpostcritical finite rescaling limits. And by combining with the properties of quadratic rational maps \cite{Kiwi14}, he gave a complete description of rescaling limits for a quadratic holomorphic family.
\begin{proposition}\cite[Theorem 1]{Kiwi15}\label{RL Number}
Let $d\ge 2$, and consider a holomorphic family $\{f_t\}\subset\mathrm{Rat}_d$. There are at most $2d-2$ pairwise dynamically independent rescalings whose rescaling limits are not postcritically finite.\par
Moreover, if $d=2$, then there are at most two dynamically independent rescalings of period at least $2$. Furthermore, in the case that a rescaling of period at least
$2$ exists, exactly one of the following holds:
\begin{enumerate}
\item $\{f_t\}$ has exactly two dynamically independent rescalings of periods $q'>q>1$. The period $q$ rescaling limit is a quadratic rational map with a multiple fixed point and a prefixed critical point. The period $q'$ rescaling limit is a quadratic polynomial, modulo conjugacy.
\item $\{f_t\}$ has a rescaling whose corresponding limit is a quadratic rational map with a multiple fixed point and every other rescaling is dynamically dependent on it.
\end{enumerate}
\end{proposition}

\subsection{Period $1$ Rescalings for Newton Maps}
In this subsection, we study the period $1$ rescalings for a holimorphic family of Newton maps and we prove Theorem \ref{theorem-rescaling-limits-Newton}$(1)$. Then using the rescaling limits induced by the these rescalings, we construct a compactification of the moduli space $\mathrm{nm}_d$.\par
Let $\{N_{r(t)}\}$ be a holomorphic family of degree $d\ge 2$ Newton maps with $r(t)=\{r_1(t),\cdots,r_d(t)\}$. Regard $r_i(t)$ as a point in $\mathbb{L}$ and denote by $\mathbf{r_i}$.  Let $\mathbf{N}$ be the associated map for $\{N_{r(t)}\}$. Recall the subset $V$ is the set of internal vertices of the convex hull $H_{\mathrm{fix}}=\mathrm{Hull}(\{\mathbf{r_1},\cdots, \mathbf{r_d},\infty\})$. Then any point $\xi\in V$ has valance at least $3$ in $H_{\mathrm{fix}}$. Moreover, by Proposition \ref{fixed-pt}, we know $V$ is the set of repelling fixed points of $\mathbf{N}$ in $\mathbb{H}_{\mathrm{Ber}}$. 
\begin{theorem}\label{period-1}
Let $\{N_{r(t)}(z)\}$ be a holomorphic family of degree $d\ge 2$ Newton maps. Then up to equivalence, $\{N_{r(t)}(z)\}$ has at most $d-1$ rescalings of period $1$. Moreover, let $\{M_t\}$ be a period $1$ rescaling for $\{N_{r(t)}(z)\}$, then, as $t\to 0$, the subalgebraic limit of $M_t^{-1}\circ N_{r(t)}(z)\circ M_t$ is conjugate to a (degenerate) map in $\overline{\mathrm{NM}}_d$. 
\end{theorem}
\begin{proof}
Let $V$ be as above and let $\mathbf{N}:\mathbb{P}^1_{\mathrm{Ber}}\to\mathbb{P}^1_{\mathrm{Ber}}$ be the associated map for $\{N_{r(t)}\}$. Note two period $1$ rescalings are equivalent if and only if they are dynamically dependent. By Proposition \ref{RL Number}, we know there are finitely many dynamically independent rescalings for $\{N_{r(t)}(z)\}$. Hence there are finitely many inequivalent rescalings. Suppose $\{M_t^{(1)}(z)\},\cdots,\{M_t^{(k)}(z)\}$ are the pairwise inequivalent period $1$ rescalings for $\{N_{r(t)}(z)\}$. Let $\mathbf{M}^{(1)},\cdots,\mathbf{M}^{(k)}$ be the associated map for $\{M_t^{(1)}\},\cdots,\{M_t^{(k)}\}$, respectively. Then by Proposition \ref{rescaling-repelling-cycles}, for each $i$, we have $\mathbf{M}^{(i)}(\xi_g)$ is a repelling fixed point of $\mathbf{N}$. Thus, by Lemma \ref{fixed-pt}, it follows that $\mathbf{M}^{(i)}(\xi_g)\in V$. Note $V$ has at most $d-1$ elements. Thus $k\le d-1$.\par
By Lemma \ref{affine-equ-mobius}, we can choose an affine representative in each equivalence class $[\{M_t^{(i)}\}]$. By Remark \ref{rescaling-limit-relation}, we know for different representatives of $[\{M_t^{(i)}\}]$, the corresponding rescaling limits are conjugate. Thus, we may assume $M_t^{(i)}$ is affine. Note $(M_t^{(i)})^{-1}\circ N_{r(t)}\circ M_t^{(i)}$ is still a degree $d$ Newton maps. Letting $t\to 0$, we get that the limit of $(M_t^{(i)})^{-1}\circ N_{r(t)}\circ M_t^{(i)}$ is a point in $\overline{\mathrm{NM}}_d$.
\end{proof}
\begin{remark}
For any rational map $\phi\in\mathbb{L}(z)$ of degree $d\ge 2$, Rumely \cite[Corollary 6.2]{Rumely17} proved $\phi$ has at most $d-1$ repelling fixed points in the hyperbolic space $\mathbb{H}_{\mathrm{Ber}}$. Hence for any holomorphic family $\{f_t\}$ of degree $d\ge 2$ rational maps, by Proposition \ref{rescaling-repelling-cycles}, we know $\{f_t\}$ has at most $d-1$ rescalings of period $1$, up to equivalence.
\end{remark}
Now we give some examples to illustrate the period $1$ rescalings and corresponding rescaling limits for holomorphic families of degree $d\ge 2$ Newton maps.
\begin{example}
Cubic Newton maps.\par
Consider the cubic Newton map $N_{r(t)}(z)$, where $r(t)=\{0,1,r_3(t)\}$ for some holomorphic function $r_3(t)$. If $N_{r(t)}(z)$ is nondegenerate, then up to equivalence, $\{M_{t}(z)=z\}$ is the unique period $1$ rescaling for $\{N_{r(t)}(z)\}$. If $N_{r(t)}(z)$ is degenerate, without loss of generality, we suppose $r_3(t)\to 0$ as $t\to 0$. Let $\mathbf{r}_3=r_3(t)\in\mathbb{L}$. Then the set $V$ contains only two points $\xi_{0,|\mathbf{r}_3|}$ and the Gauss point $\xi_g$. Thus $N_{r(t)}(z)$ has two period $1$ rescalings.\par 
Let $L_{t}(z)=r_3(t)z$. Then the associated map $\mathbf{L}$ maps $\xi_g$ to the point $\xi_{0,|\mathbf{r}_3|}$. Thus $\{L_t\}$ is a period $1$ rescaling. Moreover, we have 
$$L_t^{-1}\circ N_{r(t)}\circ L_t\xrightarrow{sa} N_{\{0,1,\infty\}}=H_{N_{\{0,1,\infty\}}}\widehat N_{\{0,1,\infty\}}.$$
Thus $L_t^{-1}\circ N_{r(t)}\circ L_t$ converges to $\widehat N_{\{0,1,\infty\}}=N_{\{0,1\}}$ locally uniformly on $\mathbb{P}^1\setminus\{\infty\}$.\par 
Note the associated map of the moving frame $\{M_{t}(z)=z\}$ fixes the Gauss point. Thus $\{M_t\}$ is also a period $1$ rescaling. And we have 
$$M_t^{-1}\circ N_{r(t)}\circ M_t\xrightarrow{sa} N_{\{0,0,1\}}=H_{N_{\{0,0,1\}}}\widehat N_{\{0,0,1\}}.$$
Thus $M_t^{-1}\circ N_{r(t)}\circ M_t$ converges to $\widehat N_{\{0,0,1\}}$ locally uniformly on $\mathbb{P}^1\setminus\{0\}$.\par  
In this example, the rescaling limits $N_{\{0,1\}}$ and $\widehat N_{\{0,0,1\}}$ are points in $\overline{\mathrm{NM}}_3$.
\end{example}
\begin{example}
Quartic Newton maps.\par
Let $r(t)=\{0,1,t,2t\}$. Consider the Newton map $N_{r(t)}$. Then the set $V$ contains two points $\xi_{0,|t|}$ and $\xi_g$. Thus $N_{r(t)}(z)$ has two period $1$ rescalings.\par 
Let $L_{t}(z)=tz$. Then the associated map $\mathbf{L}$ maps $\xi_g$ to the point $\xi_{0,|t|}$. Thus $\{L_t\}$ is a period $1$ rescaling. And we have 
$$L_t^{-1}\circ N_{r(t)}\circ L_t\xrightarrow{sa} N_{\{0,1,2,\infty\}}=H_{N_{\{0,1,2,\infty\}}}\widehat N_{\{0,1,2,\infty\}}.$$
Thus $L_t^{-1}\circ N_{r(t)}\circ L_t$ converges to $N_{\{0,1,2\}}$ locally uniformly on $\mathbb{P}^1\setminus\{\infty\}$.\par 
Note the associated map of the moving frame $\{M_{t}(z)=z\}$ fixes the Gauss point. Thus $\{M_t\}$ is also a period $1$ rescaling. And we have 
we have 
$$M_t^{-1}\circ N_{r(t)}\circ M_t\xrightarrow{sa}N_{\{0,0,0,1\}}=H_{N_{\{0,0,0,1\}}}\widehat N_{\{0,0,0,1\}}.$$
Thus $M_t^{-1}\circ N_{r(t)}\circ M_t$ converges to $\widehat N_{\{0,0,0,1\}}$ locally uniformly on $\mathbb{P}^1\setminus\{0\}$.\par 
In this example, the limits $N_{\{0,1,2,\infty\}}$ and $N_{\{0,0,0,1\}}$ are points in $\overline{\mathrm{NM}}_4$.\par
Now let $r(t)=\{0,1,t,t^2\}$ and consider the Newton map $N_{r(t)}$. Then the set $V$ contains three points $\xi_{0,|t^2|},\xi_{0,|t|}$ and $\xi_g$. Thus $N_{r(t)}(z)$ has three period $1$ rescalings.\par 
Let $L_{t}(z)=t^2z$. Then the associated map $\mathbf{L}$ maps $\xi_g$ to the point $\xi_{0,|t^2|}$. Thus $\{L_t\}$ is a period $1$ rescaling. And we have 
$$L_t^{-1}\circ N_{r(t)}\circ L_t\xrightarrow{sa}N_{\{0,1,\infty,\infty\}}=H_{N_{\{0,1,\infty,\infty\}}}\widehat N_{\{0,1,\infty,\infty\}}.$$
Thus $L_t^{-1}\circ N_{r(t)}\circ L_t$ converges to $N_{\{0,1\}}$ locally uniformly on $\mathbb{P}^1\setminus\{\infty\}$.\par  
Let $K_{t}(z)=tz$. Then the associated map $\mathbf{K}$ maps $\xi_g$ to the point $\xi_{0,|t|}$. Thus $\{K_t\}$ is a period $1$ rescaling. And we have 
$$K_t^{-1}\circ N_{r(t)}\circ K_t\xrightarrow{sa}N_{\{0,0,1,\infty\}}=H_{N_{\{0,0,1,\infty\}}}\widehat N_{\{0,0,1,\infty\}}.$$
Thus $K_t^{-1}\circ N_{r(t)}\circ K_t$ converges to $\widehat N_{\{0,0,1,\infty\}}$ locally uniformly on $\mathbb{P}^1\setminus\{0,\infty\}$.\par 
Note the associated map of the moving frame $\{M_{t}(z)=z\}$ fixes the Gauss point. Thus $\{M_t\}$ is also a period $1$ rescaling. And we have  
$$M_t^{-1}\circ N_{r(t)}\circ M_t\xrightarrow{sa}N_{\{0,0,0,1\}}=H_{N_{\{0,0,0,1\}}}\widehat N_{\{0,0,0,1\}}.$$
Thus $M_t^{-1}\circ N_{r(t)}\circ M_t$ converges to $\widehat N_{\{0,0,0,1\}}$ locally uniformly on $\mathbb{P}^1\setminus\{0\}$.\par 
In this example, the limits $N_{\{0,1,\infty,\infty\}}$, $N_{\{0,0,1,\infty\}}$ and $N_{\{0,0,0,1\}}$ are points in $\overline{\mathrm{NM}}_4$.
\end{example}
Since $\infty$ is the unique repelling fixed points of Newton maps, the moduli space of degree $d\ge 2$ Newton maps is naturally defined by 
$$\mathrm{nm}_d:=\mathrm{NM}_d/\mathrm{Aut}(\mathbb{C}), $$
modulo the action by conjugation of affine maps. In the remaining of this subsection, we will use the rescaling limits induced by rescalings of period $1$ to give a new compactification for the space $\mathrm{nm}_d$, which is quite different from the compactification $\overline{\mathrm{nm}}_d$ induced by geometric invariant theory in \cite{Nie-2}.\par 
We first define a relation $\sim$ on $\mathbb{P}^{2d+1}\times\mathbb{P}^{2d+1}$. For $f=H_f\hat f, g=H_g\hat g\in\mathbb{P}^{2d+1}$, we say $f\sim g$ if there exists $M\in\mathrm{PGL}_2(\mathbb{C})$ such that $M^{-1}\circ\hat f\circ M=\hat g$. Denote by $[f]_\sim$ for the equivalent class of $f$. Let $\mathrm{NM}_d^{\{0,1\}}$ be the subset of $\mathrm{NM}_d$ consisting of degree $d\ge 2$ Newton maps $N_r$, where $r=\{0,1,r_3,\cdots,r_d\}\subset\mathbb{C}$ is a set of $d$ distinct points. Let $\overline{\mathrm{NM}}_d^{\{0,1\}}$ be the compactification of $\mathrm{NM}_d^{\{0,1\}}$ in $\mathbb{P}^{2d+1}$. Then for any $f=H_f\hat f\in\overline{\mathrm{NM}}_d^{\{0,1\}}$, we have $\deg\hat f\ge 2$. Moreover, for any point $g=H_g\hat g\in\overline{\mathrm{NM}}_d$ with $\deg\hat g\ge 2$, there exists $f\in\overline{\mathrm{NM}}_d^{\{0,1\}}$ such that $f\sim g$. Thus, we have
\begin{lemma}
For $d\ge 3$,
$$\overline{\mathrm{NM}}_d^{\{0,1\}}/\sim=\overline{\mathrm{NM}}_d\cap\{H_f\hat f:\deg\hat f\ge 2\}/\sim.$$
\end{lemma}
To relate $\overline{\mathrm{NM}}_d^{\{0,1\}}/\sim$ to the rescaling limits, 
we now need to characteristic the set of rescaling limits induced by period $1$ rescalings for a holomorphic family of Newton maps.
\begin{lemma}\label{rescaling-limit-1}
For any $H_f\hat f\in\overline{\mathrm{NM}}_d$ with $\deg\hat f\ge 2$, there exist a holomorphic family $\{N_t\}$ of degree $d\ge 2$ Newton maps and a period $1$ rescaling $\{M_t\}$ for $\{N_t\}$ such that the corresponding rescaling limit of $\{M_t\}$ for $\{N_t\}$ is $\hat f$.
\end{lemma}
\begin{proof}
Since $H_f\hat f\in\overline{\mathrm{NM}}_d$, there exists a holomorphic family $\{N_{r(t)}\}\subset\mathrm{NM}_d$ such that $N_t\xrightarrow{sa} H_f\hat f$, as $t\to 0$, in $\mathbb{P}^{2d+1}$. Set $M_t(z)=z$. Then $\{M_t\}$ is a period $1$ rescaling for $\{N_t\}$ with rescaling limit $\hat f$ since $\deg\hat f\ge 2$.
\end{proof}
Fix $d\ge 3$ and denote by $\mathrm{RL}(\mathrm{NM}_d,1)$ the set of rescaling limits induced by a period $1$ rescaling for some holomorphic family of degree $d$ Newton maps. Then by Theorem \ref{period-1} and Lemma \ref{rescaling-limit-1}, 
$$\mathrm{RL}(\mathrm{NM}_d,1)=\overline{\mathrm{NM}}_d\cap\{H_f\hat f:\deg\hat f\ge 2\}.$$
Consider the quotient space $\mathrm{RL}(\mathrm{NM}_d,1)/\sim$ associated with the quotient topology. Then 
$$\mathrm{RL}(\mathrm{NM}_d,1)/\sim=\overline{\mathrm{NM}}_d^{\{0,1\}}/\sim.$$
Since $\overline{\mathrm{NM}}_d^{\{0,1\}}$ is compact, then the quotient space $\overline{\mathrm{NM}}_d^{\{0,1\}}/\sim$ is also compact. Hence $\mathrm{RL}(\mathrm{NM}_d,1)/\sim$ is compact.\par 
Define 
$$\overline{\mathrm{nm}}_d^{\mathrm{RL}}:=\mathrm{RL}(\mathrm{NM}_d,1)/\sim$$
Then $\overline{\mathrm{nm}}_d^{\mathrm{RL}}$ is compact and contains $\mathrm{nm}_d$ as a dense subset. We call $\overline{\mathrm{nm}}_d^{\mathrm{RL}}$ is the compactification of the moduli space $\mathrm{nm}_d$ via rescaling limits.
\begin{proposition}
For $d\ge 3$, the space $\overline{\mathrm{nm}}_d^{\mathrm{RL}}$ is not Hausdorff.
\end{proposition}
\begin{proof}
To show $\overline{\mathrm{nm}}_d^{\mathrm{RL}}$ is not Hausdorff, it is sufficient to show the relation $\sim$ is not closed on $\overline{\mathrm{NM}}_d^{\{0,1\}}\times\overline{\mathrm{NM}}_d^{\{0,1\}}$. Set $r(t)=\{0,1,t,r_4(t),\cdots,r_d(t)\}$ and $r'(t)=\{0,1/t,1,r_4(t)/t,\cdots,r_d(t)/t\},$
where $r_i(t)\not=r_j(t)$ if $i\not=j$ and $|r_i(t)|=o(|t|)$, as $t\to 0$, for $i=4,\cdots,d$.
Let $M_t(z)=tz$. Then we have if $t\not=0$,
$$M_t^{-1}\circ N_{r(t)}\circ M_t=N_{r'(t)}.$$
Hence $[N_{r(t)}]_\sim=[N_{r'(t)}]_\sim$.\par 
However, as $t\to 0$, we have $N_{r(t)}\xrightarrow{sa} N_{\{0,1, 0,\cdots,0\}}$ and $N_{r'(t)}\xrightarrow{sa} N_{\{0,\infty,1, 0,\cdots,0\}}$. Note 
$$N_{\{0,1, 0,\cdots,0\}}\not\sim N_{\{0,\infty,1, 0,\cdots,0\}}.$$
Thus $\sim$ is not closed on $\overline{\mathrm{NM}}_d^{\{0,1\}}\times\overline{\mathrm{NM}}_d^{\{0,1\}}$.
\end{proof}
For cubic Newton maps, we have
$$\overline{\mathrm{nm}}_3^{\mathrm{RL}}=\mathrm{nm}_3\cup\{[N_{\{0,1,\infty\}}]_\sim, [N_{\{0,0,1\}}]_\sim\},$$
and there are no open sets can separate $[N_{\{0,1,\infty\}}]$ and $[N_{\{0,0,1\}}]$.\par
As another example, we state explicitly the elements  in the space $\overline{\mathrm{nm}}_4^{\mathrm{RL}}$. Note $\overline{\mathrm{NM}}_4^{\{0,1\}}=\{N_{\{0,1,r_3,r_4\}}:r_3,r_4\in\widehat{\mathbb{C}}\}$. Thus 
\begin{align*}
&\overline{\mathrm{nm}}_4^{\mathrm{RL}}\setminus\mathrm{nm}_4\\
&=\{[N_{\{0,1,r,\infty\}}]_\sim,[N_{\{0,1,\infty,\infty\}}]_\sim,[N_{\{0,0,1,r\}}]_\sim,[N_{\{0,0,0,1\}}]_\sim,N_{\{0,0,1,1\}}]_\sim:r\in\mathbb{C}\setminus\{0,1\}\}.
\end{align*}

\subsection{Higher Period Rescalings for Newton Maps}
In this subsection, we first prove a sufficient and necessary condition for the existence of higher periods generalized rescalings for a holomorphic family of Newton maps. Then we prove Theorem \ref{theorem-rescaling-limits-Newton}$(2)$.\par
Recall that for two points $\xi_1,\xi_2$ in the affine Berkovich space $\mathbb{A}^1_{\mathrm{Ber}}$, the point $\xi_1\vee\xi_2$ corresponds to the smallest closed disk containing the disks related to $\xi_1$ and $\xi_2$ in $\mathbb{L}$, and the path distance metric $\rho(\xi_1,\xi_2)$ on $\mathbb{H}_{\mathrm{Ber}}$ is defined by 
$$\rho(\xi_1,\xi_2)=\log\frac{\mathrm{diam}(\xi_1\vee\xi_2)}{\mathrm{diam}(\xi_1)}+\log\frac{\mathrm{diam}(\xi_1\vee\xi_2)}{\mathrm{diam}(\xi_2)}.$$\par
First, the following example shows that there do exist higher period rescalings for a holomorphic family of quartic Newton maps. Later we will show for any quadratic polynomial $P$, there exits a holomorphic family $\{N_t\}$ of quartic Newton maps such that $P$ is a rescaling limit of a period $2$ rescaling for $\{N_t\}$.
\begin{example}\label{higher-rescaling-example}
Quartic Newton maps. See Figure \ref{higher-period}.\par
We construct a holomorphic family $\{N_{r(t)}\}$ of quartic Newton maps such that the associated map $\mathbf{N}:\mathbb{P}_{\mathrm{Ber}}^1\to\mathbb{P}_{\mathrm{Ber}}^1$ has two type II repelling points, say $V=\{\xi_1,\xi_2\}$, and a free critical point $\mathbf{c}$ satisfying the following:  Let $\xi_1$ be the visible point from $\mathbf{c}$ on the convex hull $H_V$, and let $\vec{v}\in T_{\xi_1}\mathbb{P}^1_{\mathrm{Ber}}$ be the direction such that $\mathbf{c}\in\mathbf{B}_{\xi_1}(\vec{v})^-$. Then $(1)$ at $\xi_1$, for the induced map $\mathbf{N}_\ast:T_{\xi}\mathbb{P}_\mathrm{Ber}^1\to T_{\mathbf{N}(\xi)}\mathbb{P}_\mathrm{Ber}^1$, we have $\mathbf{N}_\ast(\vec{v})$ is an attracting fixed point of $\mathbf{N}_\ast$; $(2)$ $\mathbf{N}^2(\mathbf{c})\in\mathbf{B}_{\xi_1}(\vec{v})^-$; and $(3)$ $\rho(\mathbf{c}\vee\mathbf{N}^2(\mathbf{c}),\xi_1)\ge 2\rho(\xi_1,\xi_2)$.\par
Consider $r(t)=\{r_1(t),r_2(t),r_3(t),r_4(t)\}$, where
\begin{align*}
r_1(t)&=-1-\sqrt{3}i+(5+\frac{40\sqrt{3}}{9}i)t^2,\\
r_2(t)&=-1+\sqrt{3}i+(5-\frac{40\sqrt{3}}{9}i)t^2,\\
r_3(t)&=2+\frac{\sqrt{30}}{3}t-5t^2,\\
r_4(t)&=2-\frac{\sqrt{30}}{3}t-5t^2.
\end{align*}
Let $M_t(z)=5t^2z/18$. Since as $t\to 0$
$$M_t^{-1}\circ N_{r(t)}\circ M_t\xrightarrow{\bullet}\infty\ \ \text{on}\ \ \mathbb{P}^1\setminus\{\infty\}.$$
Thus $\{M_t\}$ is not a period $1$ rescaling for the holomorphic family $\{N_{r(t)}\}$. However, $\{M_t\}$ is a period $2$ rescaling for $\{N_{r(t)}\}$. Indeed, 
$$M_t^{-1}\circ N_{r(t)}^2\circ M_t\xrightarrow{\bullet} z^2+62\ \ \text{on}\ \ \mathbb{P}^1\setminus\{\infty\}.$$
\end{example}
For a rational map $\phi:\mathbb{P}^1_{\mathrm{Ber}}\to\mathbb{P}^1_{\mathrm{Ber}}$ and a point $\xi\in\mathbb{P}^1_{\mathrm{Ber}}$, we say $\vec{v}\in T_\xi\mathbb{P}^1_{\mathrm{Ber}}$ is a critical direction at $\xi$ if $\vec{v}$ is a critical point of the induced map $\phi_\ast:T_\xi\mathbb{P}^1_{\mathrm{Ber}}\to T_{\phi(\xi)}\mathbb{P}^1_{\mathrm{Ber}}$. If $\vec{v}$ is a critical direction of $\phi$ at $\xi$, then there exists at least one critical point of $\phi$ in the Berkovich disk $\mathbf{B}_\xi(\vec{v})^-$. Recall that $\mathrm{Crit}(\phi)$ is the set of critical points for a rational map $\phi(z)\in\mathbb{L}(z)$ in  $\mathbb{P}_\mathbb{L}^1$.\par
Recall that for a holomorphic family $\{N_t\}$ of Newton maps, the subtree $H_V\subset\mathbb{P}^1_{\mathrm{Ber}}$ is the convex hull of the set $V$ of type II repelling fixed points of the associated map $\mathbf{N}$. Inspired by Example \ref{higher-rescaling-example}, we now state a sufficient and necessary condition for the existence of higher periods generalized rescalings for holomorphic families of Newton maps. 
\begin{proposition}\label{higher-period-rescalings-conditions}
For $d\ge 3$, let $\{N_t\}$ be a holomorphic family of degree $d$ Newton maps, and let $\mathbf{N}:\mathbb{P}^1_{\mathrm{Ber}}\to\mathbb{P}^1_{\mathrm{Ber}}$ be the associated map for $\{N_t\}$. Then $\{N_t\}$ has a period $q\ge 2$ generalized rescaling if and only if there exist a point $v\in V$, a critical direction $\vec{v}\in T_{v}\mathbb{P}^1_{\mathrm{Ber}}$ of $\mathbf{N}_\ast$ at $v$, a critical point $\mathbf{c}\in\mathrm{Crit}(\mathbf{N})$ in the Berkovich disk $\mathbf{B}_v(\vec{v})^-$ and an integer $1\le\ell\le q-1$ satisfying the following:
\begin{enumerate}
\item Let $\{M_{t}\}\subset\mathrm{PGL}_2(\mathbb{C})$ be such that the associated map $\mathbf{M}$ satisfies $\mathbf{M}(\xi_g)=v$. At $v$, after identifying $T_v\mathbb{P}^1_{\mathrm{Ber}}$ to $\mathbb{P}^1$, the direction $\mathbf{N}_\ast^\ell(\vec{v})$ is a hole of the subalgebraic limit of $M^{-1}_t\circ N_t\circ M_{t}$;
\item $\mathbf{N}^q(\mathbf{c})\in\mathbf{B}_v(\vec{v})^-$;
\item Let $\eta=\mathbf{c}\vee\mathbf{N}^q(\mathbf{c})$. For $k\ge 1$, there exists $\xi_k\in\mathbf{N}^{-kq}(v)\cap(v,\eta)$ such that 
$$\rho(\eta,v)\ge\rho(\xi_\infty,v),$$
where $\xi_\infty=\lim\limits_{k\to\infty}\xi_k$.
\end{enumerate} 
In particular, if $(1)-(3)$ holds, the point $\xi_\infty$ is periodic of period $q$. 
\end{proposition}
Before we prove Proposition \ref{higher-period-rescalings-conditions}, we illustrate the conditions $(1)-(3)$ using Example \ref{higher-rescaling-example}. In this example, 
$$v=\xi_g\in V=\{\xi_g,\xi_{2,|t|}\},$$
and $q=2$, $\ell=1$. The visible points from both free critical points $\mathbf{c}_1$ and $\mathbf{c}_2$ to $H_V$ are $\xi_g$, that is, $\pi_{H_V}(\mathbf{c}_1)=\pi_{H_V}(\mathbf{c}_2)=\xi_g$. Let $\mathbf{c}_1$ be the critical point such that $\pi_{H_V}(\mathbf{N}(\mathbf{c}_1))=\xi_{2,|t|}$ and let $\vec{v}\in T_{\xi_g}\mathbb{P}^1_{\mathrm{Ber}}$ such that $\mathbf{c}_1\in\mathbf{B}_{\xi_g}(\vec{v})^-$. Then at $\xi_g$, $\mathbf{N}_\ast(\vec{v})\in T_{\xi_g}\mathbb{P}^1_{\mathrm{Ber}}$ and $\xi_{2,|t|}\in\mathbf{B}_{\xi_g}(\mathbf{N}_\ast(\vec{v}))^-$. Let $\vec{w}\in T_{\xi_{2,|t|}}\mathbb{P}^1_{\mathrm{Ber}}$ such that $\mathbf{N}(\mathbf{c}_1)\in\mathbf{B}_{\xi_{2,|t|}}(\vec{w})^-$. Then at $\xi_{2,|t|}$, we have $\mathbf{N}_\ast(\vec{w})\in T_{\xi_{2,|t|}}\mathbb{P}^1_{\mathrm{Ber}}$ and $\xi_g\in\mathbf{B}_{\xi_{2,|t|}}(\mathbf{N}_\ast(\vec{w}))^-$. Moreover, we have $\mathbf{N}^2(\mathbf{c}_1)\in\mathbf{B}_{\xi_g}(\vec{v})^-$. Considering $\xi_k$ and $\xi_\infty$ as in Proposition \ref{higher-period-rescalings-conditions}, we have $\xi_\infty$ is a $2$-periodic point for $\mathbf{N}$ and $\eta=\xi_\infty$, as it is shown in Figure \ref{higher-period}. Moreover, we have 
$$\rho(\xi_g,\xi_\infty)=2\rho(\xi_1,\xi_g)=2\rho(\xi_g,\xi_{2,|t|})=2.$$
\begin{figure}[h!]
\centering
\includegraphics[width=100mm]{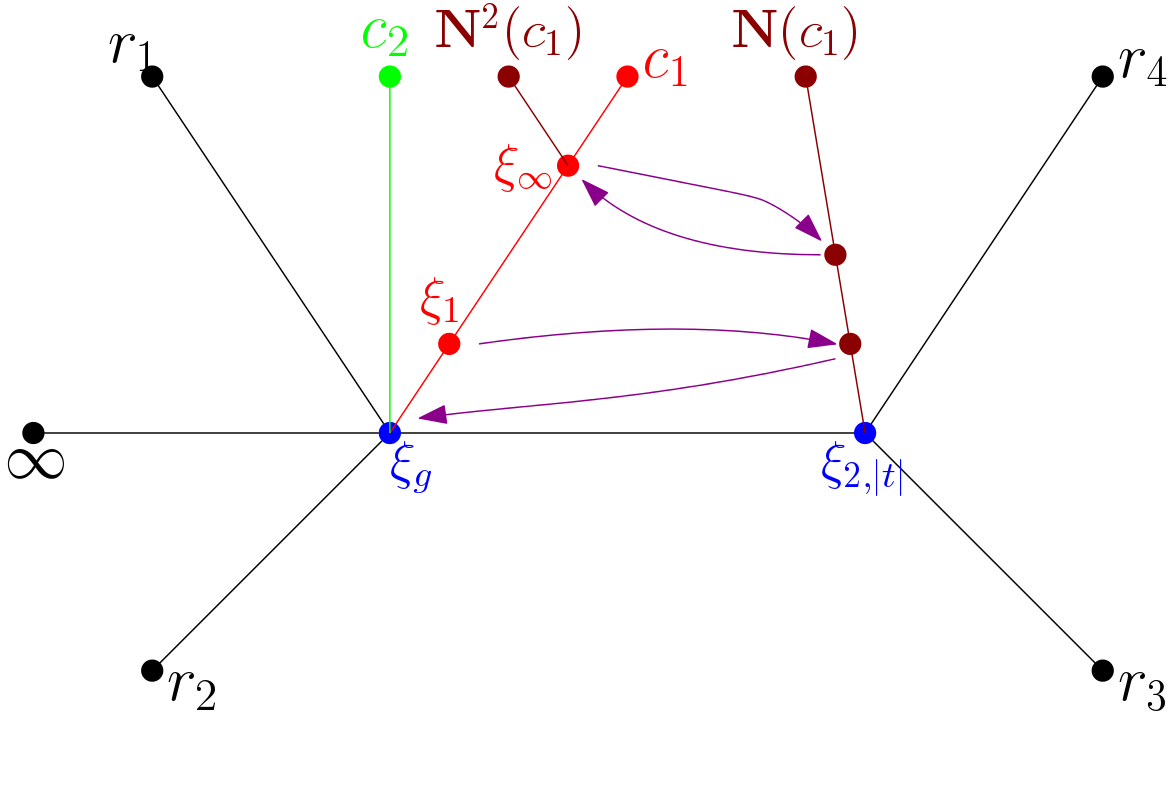}
\caption{The $2$-periodic type II cycle for Example \ref{higher-rescaling-example}. }
\label{higher-period}
\end{figure}\par
Now we prove Proposition \ref{higher-period-rescalings-conditions}.
\begin{proof}
By Proposition \ref{rescaling-repelling-cycles}, the family $\{N_t\}$ has a period $q\ge 2$ generalized rescaling if and only if the associated map $\mathbf{N}$ has a $q$-periodic repelling cycles of type II points in $\mathbb{P}^1_{\mathrm{Ber}}$.\par 
We first show the ``only if" part. If $(1)$ does not hold, then any critical point of $\mathbf{N}$ are contained in the Berkovich Fatou set $F_{\mathrm{Ber}}(\mathbf{N})$. Moreover, for any $\mathbf{c}\in\mathrm{Crit}(\mathbf{N})$, the segment $[\mathbf{c},\pi_{H_V}(\mathbf{c}))$ is in $F_{\mathrm{Ber}}(\mathbf{N})$. Note the repelling cycles of type II points intersect with the ramification locus $\mathcal{R}_{\mathbf{N}}$. By Lemmas \ref{fixed-pt} and \ref{Newton-ramification}, we know the repelling cycles of type II points are contained in $V$. Thus, all of them has period $1$. It contradicts with the existence of higher period rescalings. If $(2)$ does not hold, then for any $\mathbf{c}\in\mathrm{Crit}(\mathbf{N})$, the segment $[\mathbf{c},\pi_{H_V}(\mathbf{c}))$ disjoints with the periodic cycles. Thus, again the repelling cycles of type II points are contained in $V$. It is contradiction. Now we show $(3)$ holds. From $(1)$, $(2)$ and the existence of higher period generalized rescalings, we know for $k\ge 1$, there exist unique $\xi_k\in(\mathbf{N})^{-kq}(v)\cap(v,\eta)$. By Lemma \ref{Newton-ramification}, we know $(v,\eta)$ is in the ramification locus $\mathcal{R}_{\mathbf{N}}$. Then 
$$\rho(\xi_k,\xi_{k+1})\ge 2\rho(\xi_{k+1},\xi_{k+2})$$
for all $k\ge 0$. So $\rho(\xi_k,\xi_{k+1})\to 0$, as $k\to\infty$. Hence the sequence $\{\xi_k\}$ converges. Therefore, $\xi_\infty$ exists. Moreover, we have $\mathbf{N}^q(\xi_\infty)=\xi_\infty$. Note for any point $\xi'\in(v,\xi_\infty)$, we have
$$\rho(\mathbf{N}^q(\xi'),\xi_\infty)>\rho(\xi',\xi_\infty).$$
Hence there is no $q$-periodic point of $\mathbf{N}$ in the segment $(v,\xi_\infty)$. Thus  $(3)$ holds.\par 
Now we prove the ``if" part. Note the conditions $(1)-(3)$ imply the point $\xi_\infty$ is a fixed point of $\mathbf{N}^q$. Moreover, by Lemma \ref{Newton-ramification}, we know $\xi_\infty\in\mathcal{R}_{\mathbf{N}}$. Thus $\xi_\infty$ is a repelling fixed point of $\mathbf{N}^q$ and $\rho(\xi_k,\xi_{k+1})\ge 2\rho(\xi_{k+1},\xi_{k+2})$. Hence $\xi_\infty$ is a type II point. So $\{N_t\}$ has a generalized rescaling of period $q$.
\end{proof}
\begin{remark}\label{repelling-point-dynamics}
As in Proposition \ref{higher-period-rescalings-conditions}, from the proof, for any point $\xi'\in(v,\xi_\infty)$, we know $\mathbf{N}^q(\xi')\not=\xi'$. Let $\vec{v}\in T_{\xi_\infty}\mathbb{P}_{\mathrm{Ber}}^1$ with $v\in\mathbf{B}_{\xi_\infty}(\vec{v})^-$. Then $\mathbf{N}^q(\xi')\in\mathbf{B}_{\xi_\infty}(\vec{v})^-$ with $\rho(\mathbf{N}^q(\xi'),\xi_\infty)>\rho(\xi',\xi_\infty)$.
\end{remark}
In general, the orbit of $\xi_\infty$ may intersect the ramification locus $\mathcal{R}_\mathbf{N}$ with more than one points. For $0\le j< q$, let $v_j=\pi_{H_V}(\mathbf{N}^j(\mathbf{c}))$. If $\pi_{[v_j,\mathbf{N}^j(\xi_\infty)]}(\mathbf{a})\not\in(v_j,\mathbf{N}^j(\xi_\infty))$ for all $\mathbf{a}\in\mathrm{Crit}(\mathbf{N})$, i.e. the visible points from all critical points of $\mathbf{N}$ to the segment $[v_j,\mathbf{N}^j(\xi_\infty)]$ lie in the endpoints $\{v_j,\mathbf{N}^j(\xi_\infty)\}$, then there exists $K\ge 2$ such that 
$$\rho(\xi_k,\xi_{k+1})\ge K\rho(\xi_{k+1},\xi_{k+2}).$$
Indeed, let $\vec{v}_j\in T_{v_j}\mathbb{P}^1_{\mathrm{Ber}}$ such that $\mathbf{N}^j(\mathbf{c})\in\mathbf{B}_{v_j}(\vec{v}_j)^-$ and set 
$$K=\prod_{j=0}^{q-1}\deg_{\vec{v}_j}\mathbf{N}_\ast.$$
Moreover, in this case, we have 
$$\rho(v,\xi_\infty)=\frac{K}{K-1}\rho(v,\xi_1).$$\par
The following two results show there are only finitely many such $\xi_\infty$s for Newton map $\mathbf{N}$.
\begin{proposition}\label{repelling-period-critical}
For $d\ge 3$, let $\{N_t\}$ be a holomorphic family of degree $d$ Newton maps, and let $\mathbf{N}:\mathbb{P}^1_{\mathrm{Ber}}\to\mathbb{P}^1_{\mathrm{Ber}}$ be the associated map for $\{N_t\}$. Then for each $\mathbf{c}\in\mathrm{Crit}(\mathbf{N})$, the segment $(\mathbf{c},\pi_{H_V}(\mathbf{c}))$ contains at most one type II repelling periodic point.
\end{proposition}
\begin{proof}
By Lemma \ref{fix-segment}, we know there is no type II repelling periodic point in $(\mathbf{c},\pi_{H_V}(\mathbf{c}))$ if $\pi_{H_\mathbf{r}}(\mathbf{c})\notin V$. Now we consider the critical points $\mathbf{c}\in\mathrm{Crit}(\mathbf{N})$ with $\pi_{H_{\mathbf{r}}}(\mathbf{c})\in V$. For such $\mathbf{c}$, to ease notation, let $\mathrm{L}_\mathbf{c}:=(\mathbf{c},\pi_{H_V}(\mathbf{c}))$. Suppose $\mathrm{L}_\mathbf{c}$ contains a type II repelling periodic point $\xi_\infty$ of period $q\ge 2$. Set $\xi=\pi_{H_V}(\xi_\infty)$. By Remark \ref{repelling-point-dynamics}, we can choose $\xi_\infty$ such that there is no repelling periodic point in $(\xi,\xi_\infty)$. Suppose $\xi'\in\mathrm{L}_\mathbf{c}\setminus(\xi,\xi_\infty]$ is another type II repelling periodic point. Let $\vec{v}\in T_{\xi_\infty}\mathbb{P}_{\mathrm{Ber}}^1$ such that $\xi'\in\mathbf{B}_{\xi_\infty}(\vec{v})^-$. By Corollary \ref{not-whole-Ber}, for $1\le\ell\le q$, we know $\mathbf{N}^\ell(\mathbf{B}_{\xi_\infty}(\vec{v})^-)$ is a Bekovich disk with boundary at $\mathbf{N}^\ell(\xi_\infty)$. Thus $\xi'$ has period $kq$ for some integer $k\ge 1$. Now consider the segment $[\xi_\infty,\xi']$. For $1\le\ell\le kq$, again by Corollary \ref{not-whole-Ber}, we know $\mathbf{N}^\ell([\xi_\infty,\xi'])=[\mathbf{N}^\ell(\xi_\infty),\mathbf{N}^\ell(\xi')]$ is a segment. Thus
$$\mathbf{N}^{kq}([\xi_\infty,\xi'])=[\mathbf{N}^{kq}(\xi_\infty),\mathbf{N}^{kq}(\xi')]=[\xi_\infty,\xi'].$$
However, from Lemma \ref{dir-multi}, we know 
$$\rho(\mathbf{N}^{kq}(\xi_\infty),\mathbf{N}^{kq}(\xi'))\ge 2\rho(\xi_\infty,\xi').$$
It is a contradiction. Thus $\xi_\infty$ is the only type II repelling periodic points in $\mathrm{L}_c$. 
\end{proof}
We say a critical point $c\in\mathbb{P}_{\mathbb{L}}^1$ for a Newton map $N\in\mathbb{L}(z)$ is totally free if at the point $\xi:=\pi_{H_V}(c)$, the vector $\vec{v}_c$ is not in the immediate basin of any (super)attracting fixed point of $N_\ast$, where $\vec{v}_c\in T_{\xi}\mathbb{P}_{\mathrm{Ber}}^1$ with $c\in\mathbf{B}_{\xi}(\vec{v}_c)^-$. If a critical point $c\in N$ is not totally free, then the segment $(c,\pi_{H_V}(c))$ is in the Berkovich Fatou set $F_{\mathrm{Ber}}(N)$. Hence, in this case, there is no type II repelling periodic point in $(c,\pi_{H_V}(c))$. Thus, by Proposition \ref{repelling-period-critical}, the number of totally free critical points give a upper bound of the the number of type II repelling cycles for $\mathbf{N}$. \par
\begin{proposition}\label{number-repelling-cycles}
Let $\{N_t\}$ be a holomorphic family of degree $d\ge 3$ Newton maps and let $\mathbf{N}:\mathbb{P}^1_{\mathrm{Ber}}\to\mathbb{P}^1_{\mathrm{Ber}}$ be the associated map for $\{N_t\}$. Then $\mathbf{N}$ has at most $d-3$ totally free critical points. In particular, $\mathbf{N}$ has at most $d-3$ type II repelling cycles. 
\end{proposition}
\begin{proof}
At a point $\xi\in V$, let $C_\xi$ be the set of totally free critical points $\mathbf{c}$ of $\mathbf{N}$ such that $\pi_{H_V}(\mathbf{c})=\xi$. Let $E_\xi$ be the set of attracting fixed points of $\mathbf{N}_\ast$. Then at $\xi$, we have 
$$\#C_\xi\le 2\deg_\xi\mathbf{N}-2-\#\{r_i:\pi_{H_V}(r_i)=\xi\}-\# E_\xi,$$
where $r_i$s are the superattracting fixed points of $\mathbf{N}$. Note there are $\#V-1$ many points in $V$ such that the corresponding limits have $\infty$ as a hole, and hence 
$$\sum_{\xi\in V}\# E_\xi=\sum_{\xi\in V}\mathrm{Val}_{H_V}(\xi)-(\#V-1)=2(\#V-1)-(\#V-1)=\# V-1.$$
Then we have 
\begin{align*}
\sum_{\xi\in V}\#C_\xi&\le\sum_{\xi\in V}(2\deg_\xi\mathbf{N}-2)-\sum_{\xi\in V}\#\{r_i:\pi_{H_0}(r_i)=\xi\}-\sum_{\xi\in V}\#E_\xi\\
&=2d-2-d-(\# V-1)=d-1-\# V.
\end{align*}
Note $\#V\ge 2$. Therefore, $\mathbf{N}$ has at most $d-3$ totally free critical points. 
\end{proof}
\begin{remark}
Recall that the classical Julia set $J_I(\phi)$ for a rational map $\phi\in K(z)$ over a complete and algebraically closed non-archimedean field $K$ is the Julia set for the map $\phi:\mathbb{P}_K^1\to\mathbb{P}^1_K$. It is known that $J_I(\phi)=J_{\mathrm{Ber}}(\phi)\cap\mathbb{P}^1_K$ \cite[Theorem 10.67]{Baker10}. The Hsia's conjecture \cite[Conjecture 4.3]{Hsia00} asserts that the repelling periodic points in $\mathbb{P}_K^1$ are dense in $J_I(\phi)$. B{\'e}zivin \cite[Theorem 3]{Bezivin01} proved the Hsia's conjecture holds if there exists at least one repelling periodic point in $J_I(\phi)$.  For a Newton map $N(z)\in\mathbb{L}(z)$, same argument shows $N$ has at most $d-3$ type II repelling cycles. Rivera-Letelier \cite[Theorem 10.88]{Baker10} proved the closure of the repelling cycles in $\mathbb{P}^1_{\mathrm{Ber}}$ is the Berkovich Julia set. Thus, there are type I repelling cycles for Newton maps over $\mathbb{L}$. Hence the Hsia's conjecture holds for Newton maps $N\in\mathbb{L}(z)$.
\end{remark}
Now we can give an upper bound of the number of higher periodic (generalized) rescalings for holomorphic families of degree $d\ge 3$ Newton maps, which sharps the Kiwi's general results in \cite{Kiwi15}. First,   
\begin{theorem}\label{higher-rescaling-number-degree}
Let $\{N_t\}$ be a holomorphic family of degree $d\ge 3$ Newton maps. Then there are at most $d-3$ dynamically independent (generalized) rescalings of periods of period at least $2$ for $\{N_t\}$. Moreover, all the corresponding rescaling limits have degree at most $2^{d-3}$.
\end{theorem}
\begin{proof}
By Propositions \ref{rescaling-repelling-cycles} and \ref{number-repelling-cycles}, it follows that $\{N_t\}$ has at most $d-3$ (generalized) rescalings of periods at least $2$. Let $\{M_t\}$ be a rescaling of period $q\ge 2$ for $\{N_t\}$ and $\hat f$ is the corresponding rescaling limit. Denote $\mathbf{N}:\mathbb{P}^1_{\mathrm{Ber}}\to\mathbb{P}^1_{\mathrm{Ber}}$ the associated map for $\{N_t\}$. Then by Proposition \ref{rescaling-repelling-cycles}, we have $\deg\hat f=\deg_\xi\mathbf{N}^q$, where $\xi=\mathbf{M}(\xi_g)$ and $\mathbf{M}$ is the associated map for $\{N_t\}$. Let 
$$I=\{i\in\{0,\cdots,q-1\}:\deg_{\mathbf{N}^i(\xi)}\mathbf{N}\ge 2\}.$$
So we have 
$$\deg\hat f=\prod_{i\in I}\deg_{\mathbf{N}^i(\xi)}\mathbf{N}.$$
By Proposition \ref{number-repelling-cycles}, we have
$$\sum_{i\in I}\deg_{\mathbf{N}^i(\xi)}\mathbf{N}\le 2(d-3).$$
Let $\ell:=\# I$. Then, by the relation of arithmetic and geometric means, we have
$$\deg\hat f\le\left(\frac{2(d-3)}{\ell}\right)^{\ell}.$$
Let $m=2(d-3)/\ell$. Note $\ell$ and $m$ are integers. Then 
$$\left(\frac{2(d-3)}{\ell}\right)^{\ell}=m^\frac{2(d-3)}{m}\le 2^{d-3},$$
where we get the inequality by taking logarithm.
\end{proof}
\begin{corollary}
Let $\{N_t\}$ be a holomorphic family of cubic Newton maps. Then there is no (generalized) rescalings of higher periods for $\{N_t\}$.
\end{corollary}
For a holomorphic family $\{N_t\}$ of degree $d\ge 3$ Newton maps, if $\{M_t\}$ is an affine rescaling of period $q\ge 2$ for $\{N_t\}$, then $M_t^{-1}\circ N_t\circ M_t$ converges subalgebraically to the point in $I(d)$. For if otherwise, by the continuity of iteration away from the indeterminacy locus, $M_t^{-1}\circ N_t\circ M_t$ converges locally uniformly to a rational map of degree at least $2$ because 
$$M_t^{-1}\circ N_t^q\circ M_t=(M_t^{-1}\circ N_t\circ M_t)^q.$$
\begin{theorem}\label{rescaling-limit-poly}
Let $\{N_t\}$ be a holomorphic family of degree $d\ge 3$ Newton maps, and let $\{M_t\}$ be a rescaling of period $q\ge 2$ for $\{N_t\}$. Then the rescaling limit of $\{M_t\}$ for $\{N_t\}$ is conjugate to a polynomial. 
\end{theorem}
\begin{proof}
By Remark \ref{rescaling-limit-relation}, it is sufficient to show the rescaling limit is a polynomial if $M_t$ is affine. Now we assume $M_t$ is affine and suppose $M_t^{-1}\circ N_t^q\circ M_t$ converges subalgebraically to $f=H_f\hat f$. Then $\hat f$ is the corresponding rescaling limit of the rescaling $\{M_t\}$ for $\{N_t\}$ of period $q$. Now we show $\hat f$ is a polynomial.\par 
Set
$$\widetilde{N}_t=M_t^{-1}\circ N_t\circ M_t.$$
Then in projective coordinates, as $t\to 0$,
$$\widetilde{N}_t([X:Y])\to Y^d[1:0].$$
Let $\widetilde{\mathbf{N}}:\mathbb{P}^1_{\mathrm{Ber}}\to\mathbb{P}^1_{\mathrm{Ber}}$ be the associated map for $\{\tilde{N}_t\}$. Then by Proposition \ref{rescaling-repelling-cycles}, we have  
$$\widetilde{\mathbf{N}}^q(\xi_g)=\xi_g.$$\par
Let $\mu$ be the weak limit of measures $\mu_{\widetilde{N}_t}$. By Proposition \ref{limit-measure-case-1}, we have $\mu=\delta_\infty$. By Proposition \ref{measure-converge}, as $t\to 0$, the measure $\mu_{\widetilde{N}^q_t}$ converges weakly to $\mu_f$ since $f\not\in I(d^q)$ and $\mu_{\widetilde{N}^q_t}=\mu_{\widetilde{N}_t}$. Then we have
$$\mu_f=\delta_\infty.$$
Write $f([X:Y])=H_f(X,Y)\hat f([X:Y])$. Then by the definition of $\mu_f$ in Section \ref{measure}, we have $[1:0]$ is the unique zero of $H_f(X,Y)$ and $\hat f^{-1}([1:0])=[1:0]$. Thus $\hat f$ is a polynomial. 
\end{proof}
\begin{remark}\label{dynamics-poly}
Alternatively, Theorem \ref{rescaling-limit-poly} is also a consequence of Corollary \ref{not-whole-Ber}. Indeed, let  $\mathbf{N}$ be the associated map of $\{N_t\}$ and let $\xi\in\mathbb{P}^1_{\mathrm{Ber}}$ be a type II repelling periodic point of period $q\ge 2$. Let $\vec{v}_\infty\in T_{\xi}\mathbb{P}^1_{\mathrm{Ber}}$ be such that $\infty\in\mathbf{B}_\xi(\vec{v}_\infty)^-$. Then $\vec{v}_\infty$ is a totally invariant direction for the induced map $\mathbf{N}^q_\ast:T_{\xi}\mathbb{P}^1_{\mathrm{Ber}}\to T_{\xi}\mathbb{P}^1_{\mathrm{Ber}}$. Hence, at $\xi$, the map $\mathbf{N}^q_\ast$ is a polynomial. 
\end{remark}
\begin{remark}
If the holomorphic family $\{N_t\}$ of Newton maps has a periodic $q\ge 2$ rescaling, let $\xi\in\mathbb{P}^1_{\mathrm{Ber}}$ be a corresponding periodic point of  period $q$ for the associated map $\mathbf{N}$, which lies in the Ramification locus $\mathcal{R}_{\mathbf{N}}$. By Proposition \ref{higher-period-rescalings-conditions}, we know $\xi$ in the Berkovich Julia set $J_{\mathrm{Ber}}(\mathbf{N})$. As in the proof of Theorem \ref{limit-measure-case-1}, if $\vec{v}\in T_\xi\mathbb{P}^1_{\mathrm{Ber}}$ satisfies that $\mathbf{B}_{\xi}(\vec{v})^-\cap J_{\mathrm{Ber}}(\mathbf{N})\not=\emptyset$, then $\infty\in\mathbf{B}_{\xi}(\vec{v})^-$. Hence, if we consider the vertices set $\Gamma=\{\xi\}$, then the equilibrium $\Gamma$-measure only charges the disk $\mathbf{B}_{\xi}(\vec{v})^-$ containing $\infty$.  
\end{remark}
Recall that $\mathrm{Rat}_d^{ss}\subset\mathbb{P}^{2d+1}$ is the semistable locus. Silverman gave the following criterion of semistablity.
\begin{lemma}\cite[Proposition 2.2]{Silverman98}
Let $f\in\mathbb{P}^{2d+1}$. Then $f\not\in\mathrm{Rat}_d^{ss}$ if and only if there exists $M\in \mathrm{PGL}_2(\mathbb{C})$ such that 
$$M^{-1}\circ f\circ M=[a_d:\cdots:a_0:b_d:\cdots:b_0]$$
 with $a_i=0$ for all $i\ge(d+1)/2$ and $b_i=0$ for all $i\ge(d-1)/2$.\par
\end{lemma}
Considering the holes and depths, DeMarco \cite{DeMarco07} gave another way to describe $\mathrm{Rat}_d^{ss}$.
\begin{lemma}\cite[Section 3]{DeMarco07}\label{semistable-hole}
Let $f=H_f\hat f\in\mathbb{P}^{2d+1}$. Then $f=H_f\hat f\in\mathrm{Rat}_d^{ss}$ if and only if the depth of each hole is $\le(d+1)/2$ and if $d_h(f)\ge d/2$, then $\hat f(h)\not=h$.
\end{lemma}
The following result claims that higher periodic rescalings lead to non-semistable subalgebraical limits.
\begin{proposition}\label{rescaling-limit-not-semistable}
Let $\{N_t\}$ be a holomorphic family of degree $d\ge 3$ Newton maps, and let $\{M_t\}$ be a rescaling of period $q\ge 2$ for $\{N_t\}$. Suppose that $M_t^{-1}\circ N^q_t\circ M_t$ converges to $f$ subalgebraically. Then $f\not\in\mathrm{Rat}_{d^q}^{ss}$.
\end{proposition}
\begin{proof}
We write $f=H_f\hat f$. By Theorem \ref{rescaling-limit-poly}, we can assume $\hat f$ is a polynomial. From the proof of Theorem \ref{rescaling-limit-poly}, we know $\mathrm{Hole}(f)=\{\infty\}$. To show $f\not\in\mathrm{Rat}_{d^q}^{ss}$, by Lemma \ref{semistable-hole}, it is sufficient to show the hole $z=\infty$ has depth $d_{\infty}(f)\ge d^q/2$.\par 
If $q>d-3$, then by Theorem \ref{higher-rescaling-number-degree}, we have 
$$\frac{\deg\hat f}{d^q}\le\frac{2^{d-3}}{d^q}<\frac{1}{2}.$$\par
If $2\le q\le d-3$, let $\{\xi_0,\xi_1,\cdots,\xi_{q-1}\}$ be the corresponding type II repelling cycle. Without loss of generality, we suppose $\{\xi_0,\xi_1,\cdots,\xi_{q-1}\}\cap\mathcal{R}_\mathbf{N}=\{\xi_0,\cdots,\xi_{\ell-1}\}$. Then 
$$\deg\hat f=\prod_{i=0}^{\ell-1}\deg_{\xi_i}\mathbf{N}.$$
At a point $\xi_i$, let $\vec{v}_i\in T_{\xi_i}\mathbb{P}_{\mathrm{Ber}}^1$ with $\infty\in\mathbf{B}_{\xi_i}(\vec{v})^-$. For each $i=0,\cdots,\ell-1$, suppose there are $n_i$ critical points of $\mathbf{N}$ that are not in $\mathbf{B}_{\xi_i}(\vec{v})^-$. By Lemma \ref{fix-segment} and Proposition \ref{number-repelling-cycles}, we have 
$$n_0+\cdots+n_{\ell+1}\le d-3.$$
Now we claim $\deg_{\xi_i}\mathbf{N}=n_i+1$ for $0\le i\le\ell-1$. By Lemma \ref{dir-multi}, there exists $\xi_i'\in\mathcal{B}_{\xi_i}(\vec{v}_i)^-$ such that the directional multiplicity $m_\mathbf{N}(\xi_i,\vec{v}_i)=\deg_{\xi_i''}\mathbf{N}$ for any $\xi_i''\in(\xi_1,\xi_1')$. By Lemma \ref{disk-to-disk}, we know $\deg_{\xi_i''}\mathbf{N}\ge\deg_{\xi_i}\mathbf{N}$. Thus, $m_\mathbf{N}(\xi_i,\vec{v}_i)\ge\deg_{\xi_i}\mathbf{N}$. Hence 
$$m_\mathbf{N}(\xi_i,\vec{v}_i)=\deg_{\xi_i}\mathbf{N}.$$
Let $M_1,M_2\in\mathrm{PGL}_2(\mathbb{L})$ such that $M_1(\xi_g)=\xi_i$ and $M_2(\xi_g)=\mathbf{N}(\xi_i)$. Then at $\xi_i$, the vector $(M_1^{-1})_\ast(\vec{v}_i)$ is a fixed point of $(M_2^{-1}\circ\mathbf{N}\circ M_1)_\ast$, and 
$$\deg_{(M_1^{-1})_\ast(\vec{v}_i)}(M_2^{-1}\circ\mathbf{N}\circ M_1)_\ast=\deg(M_2^{-1}\circ\mathbf{N}\circ M_1)_\ast.$$
Thus, $\vec{v}_i$ is a critical point of $\mathbf{N}_\ast$ with multiplicity $\deg\mathbf{N}_\ast-1$. Now without of loss generality, we suppose the $n_i$ critical points of $\mathbf{N}$ that are not in $\mathbf{B}_{\xi_i}(\vec{v})^-$ are in distinct directions at $\xi_i$. Then we have 
$$n_i+\deg\mathbf{N}_\ast-1=2\deg\mathbf{N}_\ast-2.$$
Thus $n_i=\deg\mathbf{N}_\ast-1$. It follows that
$$\deg\hat f=\prod_{i=0}^{\ell-1}(n_i+1).$$
Note $$(n_0+1)+\cdots+(n_{\ell-1}+1)\le d-3+\ell.$$
So we have 
$$\deg\hat f\le(\frac{d-3+\ell}{\ell})^\ell.$$
Thus 
$$\frac{\deg\hat f}{d^q}\le\frac{1}{d^{q-\ell}}(\frac{d-3+\ell}{\ell d})^\ell\le\frac{1}{d}\le\frac{1}{2}.$$\par
Therefore, $d_\infty(f)=d^q-\deg\hat f\ge d^q/2$.
\end{proof}
For a holomorphic family $\{N_t\}$ of quartic Newton maps, by Theorems \ref{higher-rescaling-number-degree} and \ref{rescaling-limit-poly}, we know that $\{N_t\}$ has at most $1$ rescaling of period at least $2$ and the corresponding  rescaling limit is a quadratic polynomial. 
To end this subsection, based on Example \ref{higher-rescaling-example}, we consider $\tilde{r}(t)=\{\tilde{r}_1(t),\tilde{r}_2(t),\tilde{r}_3(t),\tilde{r}_4(t)\}$, where
\begin{align*}
\tilde{r}_1(t)&=-1-\sqrt{3}i+(5+\frac{40\sqrt{3}}{9}i)t^2,\\
\tilde{r}_2(t)&=-1+\sqrt{3}i+(5-\frac{40\sqrt{3}}{9}i)t^2,\\
\tilde{r}_3(t)&=2+\frac{\sqrt{30}}{3}t-5t^2+at^3,\\
\tilde{r}_4(t)&=2-\frac{\sqrt{30}}{3}t-5t^2-at^3.
\end{align*}
Then by computation, we get $\{M_t(z)=5t^2z/18\}$ is a period $2$ rescaling for $\{N_{\tilde{r}(t)}\}$ and the corresponding rescaling limit is 
$$f_a(z)=z^2+62+\frac{144a\sqrt{30}}{25}.$$
Note for any quadratic polynomial $P$, there exists $a\in\mathbb{C}$ such that $P$ is conjugate to $f_a$. Thus, for an arbitrary quadratic polynomial, there exists a holomorphic family $\{N_t\}$ of quartic Newton maps such that this polynomial is a rescaling limit of a period $2$ rescaling for $\{N_t\}$.

 \section*{Acknowledgements}
This paper is a part of the author's PhD thesis. He thanks his advisor Kevin Pilgrim for introducing him this topic and for fruitful discussions. The author would like to thank Jan Kiwi, Laura DeMarco and Juan Revira-Letelier for valuable comments. The author is also grateful to Matthieu Arfeux and Ken Jacobs for useful conversations.
\bibliographystyle{plain}
\bibliography{references}

\end{document}